\pdfoutput=1
\documentclass[english]{jnsao}
\usepackage[utf8]{inputenc}
\usepackage{lineno}

\usepackage{amsfonts}
\usepackage{graphicx}
\usepackage{epstopdf}
\usepackage{algorithmic}
\usepackage{amsopn}
\usepackage[noadjust]{cite}
\usepackage{mathtools}
\usepackage[capitalize,nameinlink]{cleveref}

\ifpdf
\DeclareGraphicsExtensions{.eps,.pdf,.png,.jpg}
\else
\DeclareGraphicsExtensions{.eps}
\fi

\crefname{hypothesis}{Hypothesis}{Hypotheses}

\usepackage{esint}

\newcommand{\shortspace}{\text{ \ \ }}
\newcommand{\mediumspace}{\text{ \ \ \ \ }}
\newcommand{\largespace}{\text{ \ \ \ \ \ \ }}

\newcommand{\bdot}{\boldsymbol{.}}
\newcommand{\boverdot}[1]{\overset{\bdot}{#1}}

\newcommand{\symnabla}{\nabla^{s}}

\newcommand{\CO}[1]{C([0,T];#1)}
\renewcommand{\C}[2]{C^{#1}([0,T];#2)}

\renewcommand{\L}[2]{L^{#1}(0,T;#2)}

\newcommand{\W}[2]{W^{1,#1}(0,T;#2)}
\newcommand{\LZ}[1]{\L{2}{#1}}
\newcommand{\WZ}[1]{H^{1}(0,T;#1)}

\newcommand{\dualpair}[3]{{\langle #1 , #2 \rangle}_{#3}}
\newcommand{\scalarproduct}[3]{\left( #1 , #2 \right)_{#3}}

\newcommand{\norm}[2]{\|#1\|_{#2}}

\DeclareMathOperator*{\argmin}{arg\,min}

\newcommand{\mathspace}[1]{\mathcal{#1}}

\newcommand{\sequence}[2]{\{ #1_{#2} \}_{#2 \in \mathbb{N}}}

\newcommand{\GG}{\mathcal{G}}
\newcommand{\HH}{\mathcal{H}}
\newcommand{\YY}{\mathcal{Y}}
\newcommand{\ZZ}{\mathcal{Z}}
\newcommand{\XX}{\mathcal{X}}
\newcommand{\UU}{\mathcal{U}}
\newcommand{\WW}{\mathcal{W}}
\newcommand{\LL}{\mathcal{L}}

\newcommand{\KK}{\mathcal{K}}
\newcommand{\EE}{\mathcal{E}}
\newcommand{\N}{\mathbb{N}}
\newcommand{\R}{\mathbb{R}}
\newcommand{\embed}{\hookrightarrow}
\newcommand{\weak}{\rightharpoonup}
\newcommand{\dual}[2]{\langle #1 , #2 \rangle}
\newcommand{\ddp}[2]{\frac{\partial #1}{\partial #2}}
\newcommand{\tddp}[2]{\tfrac{\partial #1}{\partial #2}}
\newcommand{\Rs}{\R^{d\times d}_\textup{sym}}
\newcommand{\Bb}{\mathbb{B}}

\newcommand{\Div}{\operatorname{Div}}
\newcommand{\mfu}{\mathfrak{u}}
\newcommand{\mfS}{\mathfrak{S}}
\newcommand{\maxs}{\operatorname{max}}

\newcommand{\per}{\textup{per}}

\newcommand{\defn}{\coloneqq}
\newcommand{\nfed}{\eqqcolon}

\newtheorem{assumption}[theorem]{Assumption}
\newtheorem{notation and assumption}[theorem]{Notation and Assumption}

\title{Optimal control of an abstract evolution variational inequality
with application to homogenized plasticity}

\author{Hannes Meinlschmidt\thanks{Johann Radon Institute for Computational
        and Applied Mathematics (RICAM), Alternberger Stra{\ss}e 66a, 4040
    Linz, Austria (\email{hannes.meinlschmidt@ricam.oeaw.ac.at}, \url{https://people.ricam.oeaw.ac.at/h.meinlschmidt/})}
    \and
    Christian Meyer\thanks{TU Dortmund, Faculty of Mathematics,
        Vogelpothsweg 87, 44227 Dortmund, Germany 
        (\email{christian2.meyer@tu-dortmund.de},
        \url{http://www.mathematik.tu-dortmund.de/lsx/cms/de/mitarbeiter/cmeyer.html},
        \email{stephan.walther@tu-dortmund.de},
    \url{http://www.mathematik.tu-dortmund.de/lsx/cms/de/mitarbeiter/swalther.html}).}
\and Stephan Walther\footnotemark[2]}
\shortauthor{H. Meinlschmidt, C. Meyer, S. Walther}

\acknowledgements{
    This research was supported by the German Research Foundation (DFG) under grant 
    number~ME 3281/9-1 within the priority program Non-smooth and Complementarity-based
    Distributed Parameter Systems: Simulation and Hierarchical Optimization (SPP~1962).
}

\manuscriptsubmitted{2019-10-01}
\manuscriptaccepted{2020-03-19}
\manuscriptvolume{1}
\manuscriptyear{2020}
\manuscriptnumber{5800}
\manuscriptdoi{10.46298/jnsao-2020-5800}

\begin{document}

\maketitle

\begin{abstract}
    The paper is concerned with an optimal control problem governed by a state equation in form of 
    a generalized abstract operator differential equation involving a maximal monotone operator. The state equation
    is uniquely solvable, but the associated solution operator is in general not G\^ateaux differentiable. 
    In order to derive optimality conditions, we therefore regularize the state equation and its solution operator, 
    respectively, by means of a (smoothed) Yosida approximation. We show convergence of global minimizers 
    for regularization parameter tending to zero and derive necessary and sufficient optimality conditions for 
    the regularized problems. The paper ends with an application of the abstract theory to optimal control of 
    homogenized quasi-static elastoplasticity. 
\end{abstract}

\section{Introduction}
\label{sec:1}

This paper is concerned with an optimal control problem of the following form, governed by an
operator differential equation:
\begin{equation}\label{eq:optprob}
    \tag{P}
    \left\{\quad
        \begin{aligned}
            \min \shortspace & J(z,\ell) \defn  \Psi(z,\ell) + \Phi(\ell), \\
            \text{ s.t. } \shortspace & \boverdot{z} \in A(R\ell - Qz),	\mediumspace z(0) = z_0, \\
            & (z,\ell) \in \WZ{\mathspace{H}} \times \big{(}
            \WZ{\mathspace{X}_c} \cap \mathcal{U}(z_0;M) \big{)}.
    \end{aligned}\right.
\end{equation}
Herein, $A$ is a maximal monotone operator, while $R$ and $Q$ are
linear and bounded operators in a Hilbert space $\HH$.  The control
variable is denoted by $\ell$, whereas $z$ is the state of the system.
The precise assumptions on the data are given in \cref{sec:2} below.

The particular feature of the problem under consideration is the
set-valued mapping $A$.  Due to its maximal monotony, one can show
that there is a well-defined single-valued control-to-state mapping
$\ell \mapsto z$ (in suitable function spaces), but this mapping is in
general \emph{not G\^ateaux differen\-tiable}.  We are thus faced with
a non-smooth optimal control problem, for which the derivation of
necessary and sufficient optimality conditions is a particular
challenge.

Depending on the precise choice of $A$, $R$, and $Q$, problem~\eqref{eq:optprob} covers various application problems.  For instance,
quasi-static elastoplasticity is frequently modeled in this way.
Here, $R$ is the solution operator associated with the equations of
linear elasticity for given load distribution $\ell$.  Moreover, $Q$
is the sum of the solution operator of linear elasticity for given
stress distribution, the elasticity tensor, and a coercive operator
modeling hardening effects. Finally, $A$ is the convex subdifferential
of the indicator functional of the closed and convex set of feasible
stresses defined by a suitable yield condition.  Details on models in
elastoplasticity can be found in~\cite{HanReddy1999}. Another model
which is covered by the state equation in~\eqref{eq:optprob} is the
system of homogenized elastoplasticity, which we will study in detail in
\cref{sec:7} below.

Let us put our work into perspective. Assume for a moment that $A$ is
the convex subdifferential of a proper, convex, and lower
semicontinuous functional $\phi$ and that $Q$ is self adjoint.  Then,
by convex duality, the state equation is equivalent to
\begin{equation}\label{eq:rateindep}
    0 \in \partial \phi^*(\boverdot{z}) + \EE'(z), \quad z(0) = z_0,
\end{equation}
where $\EE$ is the quadratic energy functional given by
\begin{equation*}
    \EE(z) \defn  \tfrac{1}{2} (Q z,z)_\HH - \dual{R\ell}{z}.
\end{equation*}
Systems of this type have been intensively studied concerning
existence of solutions and their numerical approximation, and we only
refer to~\cite{MielkeRoubicek2015} and the references therein. In
contrast to this, the literature on optimization problems governed by~\eqref{eq:rateindep} is rather scarce.  The research on optimal
control of equations of type~\eqref{eq:rateindep} probably started
with the sweeping process, where $\phi = I_{-C(t)}$ is the indicator
functional of a moving convex set $C(t)$, see~\cite{Mor73}.  In the
optimal control setting, $C(t)$ is most frequently set to
$C(t) = \ell(t) - Z$ with a convex set $Z$ and a driving force $\ell$.
This fits into the setting of~\eqref{eq:rateindep} by defining
$\phi \defn  I_Z$ and $Q=R=\text{id}$ (identity).  Optimal control
problems of this type are investigated in~\cite{Bro87, BK13, BK15,
AO14, CHHM15, CHHM16, CP16, AC18}, where the underlying Hilbert
space is mostly finite dimensional.  Problems in an infinite
dimensional Hilbert space are investigated in~\cite{SWW17, GW18}. To
be more precise, in these contributions, $\HH$ is the Sobolev space
$H^1_0(\Omega)$ and $\EE$ is the Dirichlet energy.  Moreover, $\phi^*$
is set to $\phi^*(z) = \|z\|_{L^1(\Omega)}$ and its viscous
regularization, respectively, i.e.,
$\phi^*_\delta(z) = \|z\|_{L^1(\Omega)} + \frac{\delta}{2}
\|z\|_{H^1_0(\Omega)}^2$.  Optimal control problems governed by
quasi-static elastoplasticity with linear kinematic hardening and von
Mises yield conditions are treated in~\cite{Wac12, Wac15, Wac16}.  As
already indicated above, $\phi$ is the indicator functional of the
convex set of feasible stresses in this case.  All mentioned problems
fit into our framework and can thus be seen as special cases of our
non-smooth evolution.  Our analysis therefore represents a
generalization of existing results on optimal control of non-smooth
evolution problems and can also be applied to application problems
that were not treated in the literature so far such as optimal control
of homogenized plasticity, which is investigated in \cref{sec:7}.  We
emphasize that problems with non-convex energies such as damage
evolution are not covered by our analysis.  Optimal control problems
governed by~\eqref{eq:rateindep} with non-convex energy are
investigated in~\cite{Rin08, Rin09}.

Our strategy to analyze~\eqref{eq:optprob} is as follows: After
showing well-posedness of the state equation and the optimal control
problem, we employ the Yosida-regularization with an additional
smoothing to obtain a smooth (i.e., Fr\'echet differentiable)
control-to-state map. We will prove that accumulation points of global
minimizers of the regularized optimal control problems for vanishing
regularization are solutions of the original non-smooth
problem~\eqref{eq:optprob}. Moreover, first-order necessary and
second-order 
sufficient optimality conditions for the regularized problems are
derived. The passage to the limit to establish optimality conditions
for the original problem goes beyond the scope of this paper and is
subject to future work.  The results of~\cite{Wac16, SWW17} indicate
that the optimality conditions obtained in this way are rather weak
and we expect the same all the more for our general setting.  Let us
underline that regularization is a widely used approach to treat
optimal control problems governed by non-smooth evolutions. We only
refer to~\cite{BK13, Wac15, GW18} and the references therein.

The paper is organized as follows. After stating our standing
assumptions in \cref{sec:2}, we investigate the state equation and its
regularization in \cref{sec:3}. In \cref{sec:exopt}, we then turn to
the optimal control problem and show that it admits an optimal
solution under our standing assumptions. Moreover, we establish the
convergence result for vanishing regularization indicated
above. Section~\ref{sec:5} is devoted to first-order necessary
optimality conditions for the regularized problems in form of a
KKT-system involving an adjoint equation.  In \cref{sec:6}, we address
second-order sufficient conditions for the regularized problem. The
paper ends with the adaptation of our general results to a concrete
application problem, namely the optimal control of homogenized
elastoplasticity.

\section{Notation and Standing Assumptions}
\label{sec:2}

We start with a short introduction in the notation used throughout the
paper.

\paragraph{Notation}
Given two vector spaces $\XX$ and $\YY$, we denote the space of linear
and continuous functions from $\mathspace{X}$ into $\mathspace {Y}$ by
$L(\mathspace{X}, \mathspace{Y})$. If $\XX = \YY$, we simply write
$L(\XX)$.  The dual space of $\XX$ is denoted by
$\mathspace{X}^* = L(\mathspace{X}, \mathbb{R})$.  If $\HH$ is a
Hilbert space, we denote its inner product by
$\scalarproduct{\cdot}{\cdot}{\mathspace{H}}$.  For the whole paper,
we fix the final time $T > 0$.  We denote the Bochner space of
square-integrable functions by $\LZ{\mathspace{X}}$ and the
Bochner-Sobolev space by $\WZ{\mathspace{X}}$.  Given a coercive
operator $G : \mathspace{H} \rightarrow \mathspace{H}$ in a Hilbert
space $\HH$, we denote its coercivity constant by $\gamma_G$, i.e.,
$\scalarproduct{Gh}{h}{\mathspace{H}} \geq \gamma_G
\norm{h}{\mathspace{H}}^2$ for all $h \in \mathspace{H}$.  Finally,
$c> 0$ and $C>0$ denote generic constants.

\subsection*{Standing Assumptions}

The following standing assumptions are tacitly assumed for the rest of
the paper without mentioning them every time.

\paragraph{Spaces}
Throughout the paper,
$\mathspace{X}, \mathspace{X}_c, \mathspace{Y}, \mathspace{Z},
\mathspace{W}$ are real Banach spaces.  Moreover, $\mathspace{X}_c$
reflexive and $\mathspace{H}$ is a separable Hilbert space.  The space
$\mathspace{X}_c$ is compactly embedded into $\mathspace{X}$ and the
embeddings $\YY \embed \ZZ \embed \HH \embed \WW$ are continuous.

\paragraph{Operators}
The operator $A \colon  \mathspace{H} \rightarrow 2^\mathspace{H}$ is
maximal monotone, its domain $D(A)$ is closed and we define
\begin{equation}\label{eq:defA0}
    A^0 \colon  D(A) \rightarrow \mathspace{H}, \quad 
    h \mapsto \argmin_{v \in A(h)} \norm{v}{\mathspace{H}}. 
\end{equation}
Furthermore, by $A_\lambda \colon  \mathspace{H} \rightarrow \mathspace{H}$,
$\lambda > 0$, we denote the Yosida-approximation of $A$ and by
$R_\lambda = (I + \lambda A)^{-1}$ the resolvent of $A$, so that
$A_\lambda = \frac{1}{\lambda}(I - R_\lambda)$.  We assume that the
operator $A^0 \colon  D(A) \rightarrow \mathspace{H}$ is bounded on bounded
sets.  For further reference on maximal monotone operators, we refer
to~\cite{brezis},~\cite[Ch.~32]{zeidler2a},~\cite[Ch.~55]{zeidler3},
and~\cite[Ch.~55]{showalter}.  Furthermore,
$R \in L(\mathspace{X};\mathspace{Y})$ and
$Q \in L(\mathspace{W};\mathspace{W})$, are linear and bounded
operators, and the restriction of $Q$ to $\mathspace{H}$,
$\mathspace{Z}$, or $\mathspace{Y}$ maps into these spaces and is
again linear and bounded.  To ease notation, we denote this
restriction by the same symbol.  Moreover,
$Q \colon  \mathspace{H} \rightarrow \mathspace{H}$ is coercive and
self-adjoint.

\paragraph{Optimization Problem}
By
$J \colon  \WZ{\mathspace{W}} \times \WZ{\mathspace{X}_c} \rightarrow
\mathbb{R}$ we denote the objective function.  We assume that both
$\Psi\colon  \WZ{\mathspace{W}} \times \WZ{\mathspace{X}_c} \to \R$ and
$\Phi\colon  \WZ{\mathspace{X}_c} \to \R$ are  weakly lower
semicontinuous. Moreover,
$\Psi$ is bounded from below
and continuous in the first argument,
while $\Phi$
is coercive.  The set $M$ is a nonempty and closed subset of $D(A)$
and $z_0 \in \YY$ is a given initial state.

\begin{remark}\label{rem:lessassus}
    We emphasize that not all of the above assumptions are always
    needed. For instance, in the next two sections, $Q$ and $R$ are only
    considered as operators with values in $\HH$ and the spaces $\YY$
    and $\ZZ$ are not needed, before the investigation of optimality
    conditions starts in Sections~\ref{sec:5} and~\ref{sec:6}.  However,
    to keep the discussion concise, we present the standing assumption
    in the present form.
\end{remark}

\section{State Equation}
\label{sec:3}

\subsection{Existence and Uniqueness}
\label{subsec:3.1}

We start the investigation of~\eqref{eq:optprob} with the discussion
of the state equation, i.e.,
\begin{equation}\label{eq:aux}
    \boverdot{z} \in A(R\ell - Qz), \mediumspace z(0) = z_0.
\end{equation}

\begin{definition}
    Let $\ell \in \WZ{\mathspace{X}}$ and $z_0 \in \mathspace{H}$.  Then
    $z \in \WZ{\mathspace{H}}$ is called solution of \cref{eq:aux}, if
    $z(0) = z_0$ and $\boverdot{z}(t) \in A(Rl(t) - Qz(t))$ holds for
    almost all $t \in [0,T]$.
\end{definition}

In order to obtain the existence of a solution to~\eqref{eq:aux}, the
data have to fulfill a certain compatibility condition. For this
reason, we introduce the following
\begin{definition}
    For $z_0 \in \mathspace{H}$ and $M \subset D(A)$, we define the set
    \begin{align*}
        \mathcal{U}(z_0,M) \defn  \bigl\{ \ell \in \WZ{\mathspace{X}}
        \colon  R\ell(0) - Qz_0 \in M \bigr\} 
    \end{align*}
    of admissible loads.
\end{definition}

\begin{theorem}[Existence result for the state equation]
    \label{thm:auxExistence}
    Let $z_0 \in \mathspace{H}$ and $\ell \in
    \mathcal{U}(z_0,D(A))$. Then there exists a unique solution
    $z \in \WZ{\mathspace{H}}$ of \cref{eq:aux}. Furthermore, there
    exists a constant $C$, independent of $z_0$ and $\ell$, such that
    \begin{align}
        \norm{z}{\CO{\mathspace{H}}} &\leq C \big(1 + \norm{z_0}{\mathspace{H}}
        + \norm{\ell}{\CO{\mathspace{X}}} + \norm{\boverdot{\ell}}{\L{1}{\mathspace{X}}}\big), \label{eq:lsgestC}\\
        \norm{\boverdot{z}}{\LZ{\mathspace{H}}}
        &\leq C \big( \norm{\boverdot{\ell}}{\LZ{\mathspace{X}}}
        + \sup_{\tau \in [0,T]} \norm{A^0(R\ell(\tau) - Qz(\tau))}{\mathspace{H}} \big), \label{eq:lsgestH1}
    \end{align}
    where $A^0$ is as defined in~\eqref{eq:defA0}.
\end{theorem}
\begin{proof} The proof essentially follows the lines of~\cite[Theorem~4.1]{groger} and~\cite[Proposition~3.4]{brezis}.  For
    convenience of the reader, we sketch the main arguments.  At first,
    one employs the transformation
    $\WZ{\mathspace{H}} \ni z \mapsto q \defn  R\ell - Qz \in
    \WZ{\mathspace{H}}$ with its inverse
    $\WZ{\mathspace{H}} \ni q \mapsto z \defn  Q^{-1}(R\ell - q) \in
    \WZ{\mathspace{H}}$ to see that \cref{eq:aux} is equivalent to
    \begin{align}
        \label{eq:transformedAux}
        \boverdot{q} + QA(q) \ni R\boverdot{\ell}, \mediumspace q(0) = R\ell(0) - Qz_0.
    \end{align}
    Since $Q$ is coercive the operator,
    $\tilde{A} \colon  \mathspace{H} \rightarrow 2^\mathspace{H}$,
    $h \mapsto QA(h)$ is maximal monotone with respect to the scalar
    product
    \begin{align*}
        \scalarproduct{h_1}{h_2}{\mathspace{H},Q^{-1}} \defn  \scalarproduct{Q^{-1}h_1}{h_2}{\mathspace{H}}, \quad h_1,h_2 \in \mathspace{H}.
    \end{align*}
    Therefore,~\cite[Proposition~3.4]{brezis} yields the existence of a
    unique solution $q \in \WZ{\mathspace{H}}$ of
    \cref{eq:transformedAux}.  To verify the estimate in~\eqref{eq:lsgestC}, we employ~\cite[Lemme~3.1]{brezis}, which gives
    \begin{align*}
        \norm{q(t) - \tilde{q}(t)}{\mathspace{H},Q^{-1}} \leq \norm{R\ell(0) - Qz_0 - a}{\mathspace{H},Q^{-1}}
        + \int_0^t \norm{R\boverdot{\ell}(\tau)}{\mathspace{H},Q^{-1}} \,d\tau,
    \end{align*}
    where $\tilde{q}$ is the unique solution of
    \begin{align*}
        \boverdot{\tilde{q}} + QA(\tilde{q}) \ni 0, \mediumspace q(0) = a
    \end{align*}
    with an arbitrary element $a \in D(A)$. This gives the desired first inequality.

    To prove the second inequality, we deduce from~\cite[Proposition~3.4]{brezis} and the associated proof that
    \begin{align*}
        \norm{q(t) - q(s)}{\mathspace{H},Q^{-1}}
        \leq \int_s^t \norm{R\boverdot{\ell}(\tau)}{\mathspace{H},Q^{-1}} \,d\tau
        + \sup_{\tau \in [0,T]} \norm{\tilde{A}^0(q(\tau))}{\mathspace{H},Q^{-1}} (t - s).
    \end{align*}
    Dividing this inequality by $(t - s)$ and letting $t \rightarrow s$
    yields
    \begin{equation}\label{eq:brezisest}
        \norm{\boverdot{q}(s)}{\mathspace{H},Q^{-1}} \leq \norm{R\boverdot{\ell}(s)}{\mathspace{H},Q^{-1}}
        + \sup_{\tau \in [0,T]} \norm{\tilde{A}^0(q(\tau))}{\mathspace{H},Q^{-1}}
    \end{equation}
    for almost all $s \in [0,T]$. From the definition of $\tilde{A}^0$
    (with respect to
    $\scalarproduct{\cdot}{\cdot}{\mathspace{H},Q^{-1}}$) we see that
    $\norm{\tilde{A}^0(h)}{\mathspace{H},Q^{-1}} \leq
    \norm{Qv}{\mathspace{H},Q^{-1}}$ for all $v \in A(h)$. This holds in
    particular for $v = A^0(h)$ so that $z = Q^{-1}(R\ell - q)$ and~\eqref{eq:brezisest} imply
    \begin{align*}
        \norm{\boverdot{z}(s)}{\mathspace{H}}
        \leq C\big( \norm{\boverdot{\ell}(s)}{\mathspace{X}} +  \norm{\boverdot{q}(s)}{\mathspace{H},Q^{-1}}\big)
        \leq C\Big( \norm{\boverdot{\ell}(s)}{\mathspace{X}} 
        + \sup_{\tau \in [0,T]} \norm{A^0(R\ell(\tau) - Qz(\tau))}{\mathspace{H}} \Big),
    \end{align*}
    which gives the second inequality.  
\end{proof}

\begin{remark}
    In order to prove \cref{thm:auxExistence}, it is sufficient to
    require that $A^0$ is bounded on compact subsets (in addition to the
    closedness of $D(A)$), cf.~\cite[Proposition~3.4]{brezis}. However,
    the boundedness on bounded sets of $A^0$ is needed to prove
    \cref{thm:continuityOfS} below and therefore, we impose it as a
    standing assumption.
\end{remark}

Based on \cref{thm:auxExistence}, we may introduce the solution
operator associated with~\eqref{eq:aux} and reduce~\eqref{eq:optprob}
to an optimization problem in the control variable only, see
\cref{def:reduced} below.  Due to the set-valued operator $A$, this
solution operator will in general be \emph{non-smooth}, which
complicates the derivation of first- and second-order optimality
conditions. A prominent way to overcome this issue is to regularize
the state equation in order to obtain a smooth solution operator. This
is frequently done by means of the Yosida-approximation, see
e.g.~\cite{barbu}, and we will pursue the same approach. For this
purpose, we will investigate the Yosida-approximation and its
convergence properties in the next subsection.

\subsection{Regularization and Convergence Results}
\label{sec:4}

For the rest of this section, we fix $z_0\in \HH$ and
$\ell \in \mathcal{U}(z_0,D(A))$ and denote the unique solution of~\eqref{eq:aux} by $z$.  We start with a convergence result of the
Yosida-approximation for fixed data $z_0$ and $\ell$ and then turn to
perturbation of the data.

\begin{proposition}[Convergence of the Yosida-approximation for fixed
    data]
    \label{prp:LInftyBoundAgainstDerivative}
    Let $z_\lambda \in \WZ{\mathspace{H}}$ be the solution of
    \begin{align}
        \label{eq:yosida}
        \boverdot{z}_\lambda = A_\lambda(R\ell - Qz_\lambda),
        \largespace z_\lambda(0) = z_0
    \end{align}
    for all $\lambda > 0$.  Then $z_\lambda \rightarrow z$ in
    $\WZ{\mathspace{H}}$ as $\lambda \searrow 0$ and the following
    inequality holds
    \begin{equation}\label{eq:yosidaest}
        \norm{z_\lambda - z}{\CO{\mathspace{H}}}^2
        + \frac{\lambda}{\gamma_Q} \norm{\boverdot{z}_\lambda}{\L{2}{\mathspace{H}}}^2
        + \frac{\lambda}{\gamma_Q}
        \norm{\boverdot{z}_\lambda - \boverdot{z}}{\L{2}{\mathspace{H}}}^2
        \leq \frac{\lambda}{\gamma_Q}
        \norm{\boverdot{z}}{\L{2}{\mathspace{H}}}^2.
    \end{equation}
\end{proposition}

\begin{proof} The proof in principle follows the lines of~\cite[Proposition~3.11]{brezis}, since our assumptions and
    assertions however are slightly different, we provide the arguments
    in detail.

    First of all, since $z \mapsto A_\lambda(R\ell - Qz)$ is
    Lipschitz-continuous by~\cite[Proposition~55.2(b)]{zeidler3}, the existence of a unique
    solution of \cref{eq:yosida} follows from Banach's contraction
    principle by standard arguments, cf.~e.g.~\cite{emmrich}.  Moreover,~\cite[Proposition~55.2(a)]{zeidler3} and the definition of
    $A_\lambda$ give
    \begin{align*}
        & \frac{d}{dt} \scalarproduct{Q(z_\lambda(t) - z(t))}{z_\lambda(t) - z(t)}{\mathspace{H}}
        = 2 \scalarproduct{\boverdot{z}_\lambda(t)
        - \boverdot{z}(t)}{Q(z_\lambda(t) - z(t))}{\mathspace{H}} \\
        &\quad = -2 \scalarproduct{\boverdot{z}_\lambda(t)
        - \boverdot{z}(t)}{R_\lambda\big[R\ell(t) - Qz_\lambda(t)\big] - \big[R\ell(t) - Qz(t)\big]}{\mathspace{H}} \\
        &\qquad -2 \scalarproduct{\boverdot{z}_\lambda(t)
            - \boverdot{z}(t)}{R\ell(t) - Qz_\lambda(t)
        - R_\lambda\big[R\ell(t) - Qz_\lambda(t)\big]}{\mathspace{H}} \\
        &\quad \leq -2 \lambda \scalarproduct{\boverdot{z}_\lambda(t) - \boverdot{z}(t)}
        {\boverdot{z}_\lambda(t)}{\mathspace{H}}
        = \lambda \Big{(} \norm{\boverdot{z}(t)}{\mathspace{H}}^2
        - \norm{\boverdot{z}_\lambda(t)}{\mathspace{H}}^2
        - \norm{\boverdot{z}_\lambda(t) - \boverdot{z}(t)}{\mathspace{H}}^2 \Big{)}.
    \end{align*}
    By integrating this inequality and using the coercivity of $Q$, we
    obtain the desired inequality.

    In order to prove the strong convergence of $z_\lambda$ to $z$ in
    $\WZ{\mathspace{H}}$, we note that $z_\lambda \rightarrow z$ in
    $\CO{\mathspace{H}}$ and
    $\norm{\boverdot{z}_\lambda}{\L{2}{\mathspace{H}}} \leq
    \norm{\boverdot{z}}{\L{2}{\mathspace{H}}}$ follow from the gained
    inequality. Hence, $z_\lambda \rightharpoonup z$ in
    $\WZ{\mathspace{H}}$ and the desired strong convergence follows from
    \cite[Proposition 3.32]{brezis2010functional}.
\end{proof}

\begin{remark}\label{rem:yosidaest}
    The above proof shows that the inequality in~\eqref{eq:yosidaest} is
    by no means restricted to the specific setting in~\eqref{eq:aux},
    i.e., whenever $\zeta \in H^1(0,T;\HH)$ and
    $\zeta_\lambda \in H^1(0,T;\HH)$ solve
    \begin{equation*}
        \begin{aligned}
            \boverdot\zeta &= \mathcal{A}(R g - Q \zeta), & \zeta(0) &= \zeta_0,\\
            \boverdot\zeta_\lambda &= \mathcal{A}_\lambda(R g - Q
            \zeta_\lambda), & \zeta_\lambda(0) &= \zeta_0,
        \end{aligned}
    \end{equation*}
    where $\mathcal{A}\colon  \HH \to 2^{\HH}$ is a maximal monotone operator,
    $\mathcal{A}_\lambda\colon  \HH \to \HH$ its Yosida-ap\-prox\-i\-ma\-tion and
    $g\in L^2(0,T;\XX)$ and $\zeta_0\in \HH$ are given, then an
    inequality analogue to~\eqref{eq:yosidaest} holds (with $\zeta$ and
    $\zeta_\lambda$ instead of $z$ and $z_\lambda$).
\end{remark}

Since we are concerned with an optimal control problem with the
external loads as control variable, the continuity of the solution
operator of~\eqref{eq:aux} and its regularization w.r.t.~variations in
the external loads is of particular interest, for instance when it
comes to the existence of optimal controls, see \cref{sec:exopt}
below. Since we aim to have a less restrictive control space in order
to allow for as many control functions as possible, the topology for
the variations of the loads needed for our continuity results should
be as weak as possible.  In particular, we aim to avoid \emph{strong
convergence} of (time-)derivatives of the loads.  The underlying
idea is similar to~\cite[Theorem~3.16]{brezis} and leads to the
following

\begin{lemma}
    \label{lem:weakConvergenceInW1RByBoundednessOfTheDerivative}
    Let $\{ z_{n,0}\}_{n \in \mathbb{N}} \subset \mathspace{H}$ and
    $\sequence{\ell}{n} \subset \L{2}{\mathspace{X}}$ be sequences such
    that $z_{n,0} \rightarrow z_0$ in $\mathspace{H}$ and
    $\ell_n \rightarrow \ell$ in $\L{1}{\mathspace{X}}$.  Assume further
    that $\sequence{A}{n}$ is a sequence of maximal monotone operators
    such that
    \begin{equation}
        \label{eq:yosidaApproximationAssumption}
        A_{n,\lambda}(h) \rightarrow A_\lambda(h)
    \end{equation}
    for all $\lambda > 0$ and all $h \in (R\ell - Qz_\lambda)([0,T])$,
    as $n \rightarrow \infty$, where $z_\lambda$ is the solution of~\eqref{eq:yosida} and $A_{n,\lambda}$ denotes the
    Yosida-approximation of $A_n$.  Then, if a sequence
    $\sequence{z}{n} \subset \WZ{\mathspace{H}}$ satisfies
    \begin{equation}\label{eq:Aneq}
        \boverdot{z}_n \in A_n(R\ell_n - Qz_n), \mediumspace z_n(0) = z_{n,0}.
    \end{equation}
    and the derivatives $\boverdot{z}_n$ are bounded in
    $\LZ{\mathspace{H}}$, then $z_n \rightharpoonup z$ in
    $\WZ{\mathspace{H}}$ and $z_n \rightarrow z$ in
    $\CO{\mathspace{H}}$.
\end{lemma}

\begin{proof} Let $\lambda > 0$ be fixed, but arbitrary and define
    $z_{n,\lambda} \in H^1(0,T; \HH)$ as solution of
    \begin{equation*}
        \boverdot{z}_{n,\lambda} = A_{n,\lambda}(R\ell_n - Qz_{n,\lambda}),
        \quad z_{n,\lambda}(0) = z_{n,0},
    \end{equation*}
    whose existence and uniqueness can again be shown by Banach's
    contraction principle as in case of~\eqref{eq:yosida}.  Owing to~\cite[Proposition~55.2(b)]{zeidler3}, we obtain
    \begin{align*}
        \norm{\boverdot{z}_\lambda(t) - \boverdot{z}_{n,\lambda}(t)}{\mathspace{H}}
        &\leq \norm{A_\lambda(R\ell(t) - Qz_\lambda(t))
        - A_{n,\lambda}(R\ell(t) - Qz_\lambda(t))}{\mathspace{H}} \\
        &\qquad + \norm{A_{n,\lambda}(R\ell(t) - Qz_\lambda(t))
        - A_{n,\lambda}(R\ell_n(t) - Qz_{n,\lambda}(t))}{\mathspace{H}} \\
        &\leq \norm{A_\lambda(R\ell(t) - Qz_\lambda(t))
        - A_{n,\lambda}(R\ell(t) - Qz_\lambda(t))}{\mathspace{H}} \\
        &\qquad + \frac{\norm{Q}{L(\mathspace{H};\mathspace{H})}}{\lambda}
        \norm{z_\lambda(t) - z_{n,\lambda}(t)}{\mathspace{H}}
        + \frac{\norm{R}{L(\mathspace{X};\mathspace{H})}}{\lambda}
        \norm{\ell(t) - \ell_n(t)}{\mathspace{X}},
    \end{align*}
    and therefore, Gronwall's inequality implies
    \begin{align*}
        \norm{z_\lambda - z_{n,\lambda}}{\CO{\mathspace{H}}} &\leq C(\lambda)
        \Big{(} \norm{z_0 - z_{n,0}}{\mathspace{H}} + \norm{\ell - \ell_n}{\L{1}{\mathspace{X}}} \\
        &\qquad\qquad + \norm{A_\lambda(R\ell - Qz_\lambda)
        - A_{n,\lambda}(R\ell - Qz_\lambda)}{\L{1}{\mathspace{H}}} \Big{)}.
    \end{align*}
    The operators $A_{n,\lambda}$ are uniformly Lipschitz continuous
    with Lipschitz constant $\lambda^{-1}$.  Thus, thanks to assumption
    \cref{eq:yosidaApproximationAssumption}, we can apply
    \cref{lem:pointwiseConvergenceInpliesUniformlyConvergence} with
    $\mathspace{M} \defn  (R\ell - Qz_\lambda)[0,T]$,
    $\mathspace{N} \defn  \mathspace{H}$, $G_n \defn  A_{n,\lambda}$ and
    $G \defn  A_\lambda$.  Together with the assumptions on $\ell_n$ and
    $z_{n,0}$ this gives that the right side in the inequality above
    converges to zero as $n \rightarrow \infty$. Using this,
    \cref{prp:LInftyBoundAgainstDerivative}, and \cref{rem:yosidaest}
    (with $\mathcal{A} = A_n$), we conclude
    \begin{equation}\label{eq:nlimlambdafixed}
        \begin{aligned}
            \limsup_{n \rightarrow \infty} \norm{z -
                z_n}{\CO{\mathspace{H}}} &\leq \norm{z -
            z_\lambda}{\CO{\mathspace{H}}}
            + \limsup_{n \rightarrow \infty} \norm{z_{n,\lambda} - z_n}{\CO{\mathspace{H}}} \\
            &\leq \sqrt{\frac{\lambda}{\gamma_Q}} \big(
                \norm{\boverdot{z}}{\L{2}{\mathspace{H}}} + \sup_{n \in
            \mathbb{N}} \norm{\boverdot{z}_n}{\L{2}{\mathspace{H}}} \big).
        \end{aligned}
    \end{equation}

    Now, since $\lambda$ was arbitrary,~\eqref{eq:nlimlambdafixed} holds
    for every $\lambda > 0$.  Therefore, as $\boverdot{z}_n$ is bounded
    in $\L{2}{\mathspace{H}}$ by assumption, we obtain
    $z_n \rightarrow z$ in $\CO{\mathspace{H}}$. Moreover, again due to
    the boundedness assumption on $\boverdot{z}_n$, there is a weakly
    converging subsequence in $\WZ{\mathspace{H}}$.  Due to
    $z_n \rightarrow z$ in $\CO{\mathspace{H}}$, the weak limit is
    unique and hence, the whole sequence $z_n$ converges weakly to $z$
    in $\WZ{\mathspace{H}}$.  
\end{proof}

\begin{lemma}
    \label{lem:yosidaFulfillsYosidaApproximationAssumption}
    Let $\sequence{\lambda}{n} \subset (0,\infty)$ be a sequence
    converging towards zero. Then the sequence $A_n \defn  A_{\lambda_n}$,
    $n \in \mathbb{N}$, of maximal monotone operators fulfills
    \cref{eq:yosidaApproximationAssumption} for all $\lambda > 0$ and
    all $h \in \mathcal{H}$.
\end{lemma}

\begin{proof} At first we prove that, for all $h \in \mathspace{H}$
    and $2\lambda > \mu > 0$, the following inequality holds
    \begin{align}
        \label{eq:resolventInequality}
        \norm{R_\lambda(h) - R_{\lambda + \mu}(h)}{\mathspace{H}}
        \leq \sqrt{\frac{\mu}{2\lambda - \mu}} \,\norm{h - R_\lambda(h)}{\mathspace{H}}.
    \end{align}
    For this purpose, let $h \in \mathspace{H}$ be arbitrary and set
    $y_1 \defn  R_\lambda(h)$ and $y_2 \defn  R_{\lambda + \mu}(h)$.  Then we
    have $h \in y_1 + \lambda A(y_1)$, hence,
    $\frac{h - y_1}{\lambda} \in A(y_1)$ and analogously
    $\frac{h - y_2}{\lambda + \mu} \in A(y_2)$. The monotonicity of $A$
    thus implies
    \begin{align*}
        0 \leq \scalarproduct{\frac{\lambda + \mu}{\lambda}(h - y_1) - (h - y_2)}{y_1 - y_2}{\mathspace{H}}
        \leq \Big(\frac{\mu}{2 \lambda} - 1\Big) \norm{y_1 - y_2}{\mathspace{H}}^2
        + \frac{\mu}{2 \lambda} \norm{h - y_1}{\mathspace{H}}^2,
    \end{align*}
    hence,
    \begin{align*}
        \norm{y_1 - y_2}{\mathspace{H}}^2 \leq \frac{\mu}{2\lambda - \mu} \norm{h - y_1}{\mathspace{H}}^2,
    \end{align*}
    which yields \cref{eq:resolventInequality}. With this inequality and~\cite[Proposition 55.2 (d)]{zeidler3} at hand, we obtain
    \begin{align*}
        (A_{\lambda_n})_\lambda(h) = A_{\lambda_n + \lambda}(h)
        = \frac{1}{\lambda_n + \lambda} (h - R_{\lambda_n + \lambda}(h)) 
        \rightarrow \frac{1}{\lambda} (h - R_{\lambda}(h)) = A_{\lambda}(h),
    \end{align*}
    which completes the proof.  
\end{proof}

Now, we are in the position to state our main convergence results in
\cref{thm:SolOpcont} and \cref{thm:continuityOfS}, where the loads and
initial data are no longer fixed. The first theorem addresses the
continuity properties of the solution operator to the original
equation~\eqref{eq:aux}, whereas \cref{thm:continuityOfS} deals with
the Yosida-approximation.  In order to sharpen these convergence
results and prove strong convergence in $\WZ{\mathspace{H}}$, we
additionally need the following

\begin{assumption}
    \label{ass:AIsSubdifferential}
    The maximal monotone operator $A$ is given as a subdifferential of a
    proper, convex and lower semicontinuous function
    $\phi \colon  \mathspace{H} \rightarrow (-\infty,\infty]$, that is,
    $A = \partial\phi$.
\end{assumption}

\begin{theorem}[Continuity of the solution operator]\label{thm:SolOpcont}
    Let $\{ z_{n,0} \}_{n \in \mathbb{N}} \subset \mathspace{H}$ and
    $\sequence{\ell}{n} \subset \mathcal{U}(z_{n,0},D(A))$ be sequences
    such that $z_{n,0} \rightarrow z_0$ in $\mathspace{H}$,
    $\ell_n \rightharpoonup \ell$ in $\WZ{\mathspace{X}}$ and
    $\ell_n \rightarrow \ell$ in $\L{1}{\mathspace{X}}$.  Moreover,
    denote the solution of
    \begin{equation*}
        \boverdot{z}_n \in A(R\ell_n - Qz_n), \mediumspace z_n(0) = z_{n,0}
    \end{equation*}
    by $z_n \in H^1(0,T;\HH)$ (whose existence is guaranteed by
    \cref{thm:auxExistence}).  Then $z_n \rightharpoonup z$ in
    $\WZ{\mathspace{H}}$ and $z_n \rightarrow z$ in
    $\CO{\mathspace{H}}$.

    If additionally $\ell_n \rightarrow \ell$ in $\WZ{\mathspace{X}}$,
    $A$ fulfills \cref{ass:AIsSubdifferential}, and
    $\phi(R\ell_n(0) - Qz_{n,0}) \rightarrow \phi(R\ell(0) - Qz_0)$,
    then $z_n \rightarrow z$ in $\WZ{\mathspace{H}}$.
\end{theorem}

\begin{proof} Thanks to \cref{thm:auxExistence}, to be more precise~\eqref{eq:lsgestC}, $\{z_n\}$ is bounded in $C([0,T];\HH)$.  Since
    $A^0$ is bounded on bounded sets by assumption,~\eqref{eq:lsgestH1}
    then gives that $\{\boverdot{z}_n\}$ is bounded in
    $\LZ{\mathspace{H}}$.  Therefore, we can apply
    \cref{lem:weakConvergenceInW1RByBoundednessOfTheDerivative} with
    $A_n \defn  A$ for all $n \in \mathbb{N}$ to obtain
    $z_n \rightharpoonup z$ in $\WZ{\mathspace{H}}$ and
    $z_n \rightarrow z$ in $\CO{\mathspace{}H}$.

    If additionally $\ell_n \rightarrow \ell$ in $\WZ{\mathspace{X}}$,
    $A$ fulfills \cref{ass:AIsSubdifferential}, and
    $\phi(R\ell_n(0) - Qz_{n,0}) \rightarrow \phi(R\ell(0) - Qz_0)$, we
    can follow the lines of~\cite[Theorem~4.2 step 3)]{groger} to get
    \begin{align*}
        \limsup_{n \rightarrow \infty} \int_0^T (Q \boverdot{z}_n, \boverdot{z}_n)_\HH\, dt
        &= \limsup_{n \rightarrow \infty}
        -(R\boverdot{\ell}_n - Q\boverdot{z}_n, \boverdot{z}_n)_{\LZ{\mathspace{H}}}
        + (R\boverdot{\ell}, \boverdot{z})_{\LZ{\mathspace{H}}} \\
        & = \limsup_{n \rightarrow \infty} \phi(R\ell_n(0) - Qz_{n,0}) - \phi(R\ell_n(T) - Qz_n(T))
        + (R\boverdot{\ell}, \boverdot{z})_{\LZ{\mathspace{H}}} \\
        & \leq \phi(R\ell(0) - Qz_0) - \phi(R\ell(T) - Qz(T))
        + (R\boverdot{\ell}, \boverdot{z})_{\LZ{\mathspace{H}}} \\
        & = \int_0^T (Q \boverdot{z}, \boverdot{z})_\HH\, dt
    \end{align*}
    where the second and last equation follows from~\cite[Lemme~3.3]{brezis}.  Hence, by equipping $\HH$ with the
    equivalent norm $\sqrt{(Q \cdot, \cdot)_\HH}$, the strong convergence
    $z_n \rightarrow z$ in $\WZ{\mathspace{H}}$ follows from
    \cite[Proposition 3.32]{brezis2010functional}.
\end{proof}

\begin{theorem}[Convergence of the Yosida-approximation]
    \label{thm:continuityOfS}
    The statement of
    \cref{thm:SolOpcont}
    holds true when
    $z_n$
    is,
    for every $n \in \mathbb{N}$,
    the solution of
    \begin{align}
        \label{eq:generalYosidaAux}
        \boverdot{z}_n = A_{\lambda_n}(R\ell_n - Qz_n), \mediumspace z_n(0) = z_{n,0},
    \end{align}
    where
    $\sequence{\lambda}{n} \subset (0,\infty)$
    is a sequence converging to zero.
\end{theorem}

\begin{proof} According to
    \cref{lem:yosidaFulfillsYosidaApproximationAssumption}, the sequence
    of maximal monotone operators $A_n \defn  A_{\lambda_n}$ fulfills
    \cref{eq:yosidaApproximationAssumption} so that it only remains to
    prove that $\boverdot{z}_n$ is bounded in $\LZ{\mathspace{H}}$ to
    apply again
    \cref{lem:weakConvergenceInW1RByBoundednessOfTheDerivative}.  To
    this end, let $v_n \in \WZ{\mathspace{H}}$ be the solution of
    \begin{align*}
        \boverdot{v}_n \in A(R\ell_n - Qv_n), \mediumspace v_n(0) = z_{n,0}, 
    \end{align*}
    whose existence is guaranteed by \cref{thm:auxExistence} (note that
    $\ell_n \in \mathcal{U}(z_{n,0},D(A))$ by assumption).  Thanks to
    \cref{thm:SolOpcont}, it holds $v_n \rightharpoonup z$ in
    $\WZ{\mathspace{H}}$.  From \cref{prp:LInftyBoundAgainstDerivative},
    it follows
    $\norm{\boverdot{z}_n}{\LZ{\mathspace{H}}} \leq
    \norm{\boverdot{v}_n}{\LZ{\mathspace{H}}}$ and consequently,
    $\boverdot{z}_n$ is bounded in $\LZ{\mathspace{H}}$.  Thus,
    \cref{lem:weakConvergenceInW1RByBoundednessOfTheDerivative} yields
    $z_n \rightharpoonup z$ in $\WZ{\mathspace{H}}$ and
    $z_n \rightarrow z$ in $\CO{\mathspace{H}}$, as claimed.

    If additionally $\ell_n \rightarrow \ell$ in $\WZ{\mathspace{X}}$,
    $A$ fulfills \cref{ass:AIsSubdifferential}, and
    $\phi(R\ell_n(0) - Qz_{n,0}) \rightarrow \phi(R\ell(0) - Qz_0)$,
    then \cref{thm:SolOpcont} implies $v_n \rightarrow z$ in
    $\WZ{\mathspace{H}}$ so that
    \cite[Proposition 3.32]{brezis2010functional}  
    gives the
    strong convergence $z_n \rightarrow z$ in $\WZ{\mathspace{H}}$
    because of
    $\norm{\boverdot{z}_n}{\LZ{\mathspace{H}}} \leq
    \norm{\boverdot{v}_n}{\LZ{\mathspace{H}}}$ as seen above.  
\end{proof}

\begin{remark}
    The assertions of \cref{thm:SolOpcont} and \cref{thm:continuityOfS}
    are remarkable due to the following: As a first approach to prove
    the (strong) convergence of the states in $\WZ{\HH}$, one is tempted
    to follow the lines of the proofs of~\cite[Lemme~3.1]{brezis} and
    \cref{prp:LInftyBoundAgainstDerivative}, respectively. This would
    however require the strong convergence of the \emph{derivatives} of
    the given loads, which we want to avoid in order to enable less
    regular controls.  The detour via the Yosida-regularization in
    \cref{lem:weakConvergenceInW1RByBoundednessOfTheDerivative} allows
    to overcome this issue.
\end{remark}
\begin{remark}
    If we would allow for more regular controls, then we could weaken
    the assumptions on the maximal monotone operator $A$.  For instance,
    if $\ell \in H^2(0,T;\mathspace{X})$, then we can drop the
    assumptions that $D(A)$ is closed and $A^0$ is bounded on bounded
    sets.  In this case, one can use~\cite[Theorem~55.A]{zeidler3}
    instead of~\cite[Proposition~3.4]{brezis} in the proof of
    \cref{thm:auxExistence}.  The proof of~\cite[Theorem~55.A]{zeidler3}
    also gives the boundedness of $\boverdot{z}_n$ in
    $L^\infty(0,T;\HH)$ in this case. Thus,
    \cref{lem:weakConvergenceInW1RByBoundednessOfTheDerivative} is again
    applicable and we can argue similar as we did in the proof of
    \cref{thm:continuityOfS} to verify the previous convergence results.
    In this setting, we would not have any restrictions on $A$ (except
    monotonicity), but would need more regular loads, which is not
    favorable, as the latter serve as control variables in our
    optimization problem.  Moreover, the boundedness assumption on $A$
    is fulfilled for our concrete application problem in
    \cref{subsec:vonmises}. Therefore, we decided to choose the present
    setting and to impose the additional boundedness assumption on $A$.
\end{remark}
\begin{remark}
    It is to be noted that most of the above results can also be shown
    in more general Bochner-Sobolev spaces, that is, when loads are
    contained in $\W{r}{\mathspace{X}}$ and states in
    $\W{r}{\mathspace{H}}$ for some $r\in [1, \infty)$. However, since a
    Hilbert space setting is advantageous when it comes to the
    derivation of optimality conditions, we focus on the case $r=2$.
\end{remark}

Unfortunately, the Yosida-approximation is frequently not sufficient
for the de\-ri\-va\-tion of optimality conditions by means of the standard
adjoint approach, since the solution operator associated with~\eqref{eq:yosida} is in general still not G\^ateaux differentiable.
Therefore, we apply a second regularization turning the
Yosida-approximation of $A$ into a smooth operator. The properties
needed to ensure convergence of this second regularization are
investigated in the following

\begin{lemma}[Convergence of the Regularized Yosida-Approximation]
    \label{lem:approximationOfTheYosidaApproximation}
    Consider a sequence $\sequence{\lambda}{n}\subset (0, \infty)$ and a
    sequence of Lipschitz continuous operators
    $A_n \colon  \mathcal{H} \rightarrow \mathcal{H}$, $n\in \mathbb{N}$, such
    that
    \begin{equation}
        \lambda_n \searrow 0 \quad \text{and} \quad 
        \frac{1}{\lambda_n}\exp \Big(\frac{T\norm{Q}{L(\mathcal{H};\mathcal{H})}}{\lambda_n} \Big) 
        \,\sup_{h\in \HH} \norm{A_n(h) - A_{\lambda_n}(h)}{\mathcal{H}}\, \rightarrow 0.
    \end{equation}
    Let moreover $\sequence{\ell}{n} \subset \CO{\mathcal{X}}$ be given
    and denote by $z_n, z_{\lambda_n} \in \C{1}{\mathcal{H}}$ the
    solutions of
    \begin{alignat}{3}
        \boverdot{z}_n &= A_n(R\ell_n - Qz_n), & \quad  z_n(0) &= z_0, \label{eq:Aneq1}\\
        \text{and} \qquad \boverdot{z}_{\lambda_n} &=
        A_{\lambda_n}(R\ell_n - Qz_{\lambda_n}), & \quad z_{\lambda_n}(0)
        &= z_0. \label{eq:Alambdaneq}
    \end{alignat}
    Then $\norm{z_n - z_{\lambda_n}}{\C{1}{\mathcal{H}}} \rightarrow 0$.
\end{lemma}

\begin{proof} Again, thanks to the Lipschitz continuity of $A_n$ and
    $A_{\lambda_n}$, the existence and uniqueness of $z_n$ and
    $z_{n,\lambda}$ follows from Banach's contraction principle by
    classical arguments. Moreover, the continuity of $\ell_n$ carries
    over to the continuity of $\boverdot{z}_n$ and
    $\boverdot{z}_{n,\lambda}$.  Let us abbreviate
    $c_n \defn  \sup_{h\in \HH} \norm{A_n(h) -
    A_{\lambda_n}(h)}{\mathcal{H}}$.  Then, in light of~\eqref{eq:Aneq1} and~\eqref{eq:Alambdaneq}, we find
    \begin{align*}
        \norm{\boverdot{z}_n(t) - \boverdot{z}_{\lambda_n}(t)}{\mathcal{H}}
        \leq c_n + \frac{\norm{Q}{L(\mathcal{H};\mathcal{H})}}{\lambda_n}
        \norm{z_n(t) - z_{\lambda_n}(t)}{\mathcal{H}} \quad \forall\, t\in [0,T]
    \end{align*}
    so that Gronwall's inequality yields
    \begin{align*}
        \norm{\boverdot{z}_n(t) - \boverdot{z}_{\lambda_n}(t)}{\mathcal{H}}
        \leq \frac{\norm{Q}{L(\mathcal{H};\mathcal{H})}}{\lambda_n}
        \Big{(}
        T \exp \Big{(} \frac{\norm{Q}{L(\mathcal{H};\mathcal{H})}}{\lambda_n} \,T \Big{)} + 1
        \Big{)}
        c_n  \quad \forall\, t\in [0,T],
    \end{align*}
    which completes the proof.  
\end{proof}

\section{Existence and Approximation of Optimal
Controls}\label{sec:exopt}

Now we turn to the optimal control problem~\eqref{eq:optprob}. We
first address the existence of optimal solutions and afterwards
discuss the approximation of~\eqref{eq:optprob} in \cref{subsec:4.2}.

\subsection{Existence of Optimal Controls}
\label{subsec:3.2}

Based on \cref{thm:auxExistence}, we reduce the optimal control
problem~\eqref{eq:optprob} into a problem in the control variable
only. Recall that the control space $\XX_c$ embeds compactly in $\XX$.

\begin{definition}\label{def:reduced}
    Let $z_0 \in \mathspace{H}$ and $M \subset D(A)$. Due to
    \cref{thm:auxExistence}, there exists for every
    $\ell \in \WZ{\mathspace{X}_c} \cap \mathcal{U}(z_0;M)$ a solution
    $z \in \WZ{\mathspace{H}}$ of the state equation in
    \cref{eq:aux}. Consequently, we may define the solution operator
    \begin{equation*}
        \mathcal{S} \colon  \WZ{\mathspace{X}_c} \cap \mathcal{U}(z_0;M) \ni \ell \rightarrow z \in \WZ{\mathspace{H}}.	
    \end{equation*}
    This operator will be frequently called control-to-state map.
\end{definition}

With the definition above, problem \cref{eq:optprob} is equivalent to
the \emph{reduced problem}:
\begin{equation*}~\eqref{eq:optprob} \quad \Longleftrightarrow \quad 
    \left\{\;
        \begin{aligned}
            \min \quad & J(\mathcal{S}(\ell),\ell), \\
            \text{s.t.} \quad & \ell \in \WZ{\mathspace{X}_c} \cap
            \mathcal{U}(z_0;M).
        \end{aligned}
    \right.
\end{equation*}

Recall our standing assumptions on the objective, namely that
$J(z,\ell) = \Psi(z,\ell) + \Phi(\ell)$, where
$\Psi\colon  \WZ{\mathspace{W}} \times \WZ{\mathspace{X}_c} \to \R$ is
weakly lower semicontinuous and bounded from below and
$\Phi\colon  \WZ{\mathspace{X}_c} \to \R$ is weakly lower semicontinuous and
coercive.  These assumptions allow us to show the existence of
(globally) optimal solutions:

\begin{theorem}[Existence of Optimal Solutions]
    \label{thm:existenceOfAGlobalSolution}
    Let $z_0 \in \mathspace{H}$ and $M$ be a closed subset of
    $D(A)$. Then there exists a global solution of \cref{eq:optprob}.
\end{theorem}

\begin{proof} Based on \cref{thm:SolOpcont}, the proof follows the
    standard direct method of the calculus of variations.  First of all,
    since $\Psi$ is bounded from below and $\Phi$ is coercive, every
    infimal sequence of controls is bounded in $\WZ{\mathspace{X}_c}$
    and thus admits a weakly converging subsequence.  Due to the compact
    embedding of $\mathspace{X}_c$ in $\XX$, this sequence converges
    strongly in $\CO{\mathspace{X}}$ so that the weak limit belongs to
    $\UU(z_0;M)$, due the closeness of $M$.  Moreover, thanks to weak
    convergence in $\WZ{\XX}$ and strong convergence in $\CO{\XX}$,
    \cref{thm:SolOpcont} gives weak convergence of the associated states
    in $\WZ{\HH}$. The weak lower semicontinuity of $\Psi$ and $\Phi$
    together with $\HH \embed \mathspace{W}$ then implies the optimality
    of the weak limit.  
\end{proof}

Clearly, in view of the nonlinear state equation, one cannot expect
the optimal solution to be unique.  Note that, since $D(A)$ is closed
by our standing assumptions, the choice $M=D(A)$ is feasible.

\subsection{Convergence of Global Minimizers}
\label{subsec:4.2}

While the existence of optimal solutions for~\eqref{eq:optprob} can be
shown by well-established techniques as seen above, the derivation of
optimality conditions is all but standard because of the lack of
differentiability of the control-to-state map. We therefore apply a
regularization of $A$ built upon the Yosida-approximation in order to
obtain a smooth control-to-state mapping. In view of
\cref{lem:approximationOfTheYosidaApproximation}, this regularization
is assumed to satisfy the following

\begin{assumption}
    \label{ass:AnAssumption}
    Let $\sequence{A}{n}$ be a sequence of Lipschitz continuous
    operators from $\HH$ to $\HH$ such that, together with a sequence
    $\sequence{\lambda}{n} \subset (0, \infty)$, it holds
    \begin{equation}\label{eq:approxassu}
        \lambda_n \searrow 0 \quad \text{and} \quad 
        \frac{1}{\lambda_n}\exp \Big(\frac{T\norm{Q}{L(\mathcal{H};\mathcal{H})}}{\lambda_n} \Big) 
        \,\sup_{h\in \HH} \norm{A_n(h) - A_{\lambda_n}(h)}{\mathcal{H}}\, \rightarrow 0,
    \end{equation}
    i.e., the requirements in
    \cref{lem:approximationOfTheYosidaApproximation} are fulfilled.
\end{assumption}

In \cref{subsec:vonmises}, we show how to construct such a
regularization for a concrete application problem.  Given the
regularization of $A$, we define the corresponding optimal control
problem:
\begin{equation}\tag{\mbox{P$_n$}}\label{eq:optprobn}
    \left\{\quad
        \begin{aligned}
            \min \shortspace & J(z,\ell), \\
            \text{ s.t. } \shortspace & \boverdot{z} = A_n(R\ell - Qz),	\mediumspace z(0) = z_0, \\
            & (z,\ell) \in \WZ{\mathspace{H}} \times \big{(}
            \WZ{\mathspace{X}_c} \cap \mathcal{U}(z_0;M) \big{)}.
    \end{aligned}\right.
\end{equation}
Since $A_n$ is Lipschitz continuous, the equation
\begin{equation*}
    \boverdot{z} = A_n(R\ell - Qz), \quad z(0) = z_0
\end{equation*}
admits a unique solution for every $z_0\in \HH$ and every
$\ell \in \LZ{\mathspace{X}}$.  Similar to \cref{def:reduced}, we
denote the associated solution operator by
\begin{equation*}
    \mathcal{S}_n \colon  \LZ{\mathspace{X}} \to \WZ{\mathspace{H}}.	
\end{equation*}
Moreover, the solution operator associated with the
Yosida-approximation, i.e., the solution operator of
$\boverdot{z} = A_{\lambda_n}(R\ell - Qz)$, $z(0) = z_0$, is denoted
by
\begin{equation*}
    \mathcal{S}_{\lambda_n} \colon  \LZ{\mathspace{X}} \to \WZ{\mathspace{H}}.	
\end{equation*}

\begin{proposition}[Existence of Optimal Solutions of the Regularized Problems]\label{prop:exregopt}
    Let $n \in \mathbb{N}$, $z_0 \in \mathspace{H}$, and $M$ a closed
    subset of $D(A)$.  Then, under \cref{ass:AnAssumption}, there exists
    a global solution of \cref{eq:optprobn}.
\end{proposition}

\begin{proof} Let $\ell_1, \ell_2\in \LZ{\XX}$ be arbitrary and
    define $z_i \defn  \mathcal{S}_n(\ell_i)$, $i=1,2$.  Then, due to the
    Lipschitz continuity of $A_n$, we have for almost all $t \in [0, T]$
    \begin{equation}
        \begin{aligned}
            \norm{\boverdot{z}_1(t) - \boverdot{z}_2(t)}{\mathspace{H}} &=
            \norm{A_n(R\ell_1(t) - Qz_1(t))
            - A_n(R\ell_2(t) - Qz_2(t))}{\mathspace{H}} \\
            &\leq c \big( \norm{l_1(t) - l_2(t)}{\mathspace{X}} +
            \norm{z_1(t) - z_2(t)}{\mathspace{H}} \big),
        \end{aligned}
    \end{equation}
    which yields, thanks to Gronwall's inequality , the Lipschitz
    continuity of $\mathcal{S}_n$.  Using this together with the fact
    that $\mathspace{X}_c$ is compactly embedded into $\mathspace{X}$,
    one can argue as in the proof of
    \cref{thm:existenceOfAGlobalSolution} to obtain the existence of a
    global solution of \cref{eq:optprobn} for all $n \in \mathbb{N}$.  
\end{proof}

\begin{theorem}[Weak Approximation of Global Minimizers]
    \label{thm:regularizedOptimaConvergence}
    Let $z_0 \in \mathspace{H}$ and $M$ be a closed subset of $D(A)$.
    Suppose moreover that \cref{ass:AnAssumption} holds and let
    $\sequence{\overline{\ell}}{n}$ be a sequence of globally optimal
    controls of \cref{eq:optprobn}, $n\in\N$.  Then there exists a weak
    accumulation point and every weak accumulation point is a global
    solution of \cref{eq:optprob}.
\end{theorem}

\begin{proof} Due to $M\subset D(A)$,
    \cref{prp:LInftyBoundAgainstDerivative} gives
    $\mathcal{S}_{\lambda_n}(\overline\ell_1) \rightarrow
    \mathcal{S}(\overline\ell_1)$ in $\WZ{\mathspace{H}}$ so that
    \cref{lem:approximationOfTheYosidaApproximation} yields
    $\mathcal{S}_{n}(\overline{\ell}_1) \rightarrow
    \mathcal{S}(\overline\ell_1)$ in $\WZ{\mathspace{H}}$ and thus
    \begin{align*}
        \limsup_{n \rightarrow \infty}
        \Psi(\mathcal{S}_n(\overline{\ell}_n),\overline{\ell}_n)
        +
        \Phi(\overline\ell_n)
        =
        \limsup_{n \rightarrow \infty}
        J(\mathcal{S}_n(\overline{\ell}_n),\overline{\ell}_n) \leq
        \limsup_{n \rightarrow \infty}
        J(\mathcal{S}_n(\overline{\ell}_1),\overline{\ell}_1) =
        J(\mathcal{S}(\overline{\ell}_1),\overline{\ell}_1).
    \end{align*}
    Hence, by virtue of the boundedness of $\Psi$ from below and the
    radial unboundedness of $\Phi$, $\{\overline{\ell}_n\}$ is bounded
    and therefore admits a weak accumulation point in $\WZ{\XX_c}$.

    Let us now assume that a given subsequence of
    $\sequence{\overline{\ell}}{n}$, denoted by the same symbol for
    simplicity, converges weakly to $\tilde{\ell}$ in
    $\WZ{\mathspace{X}_c}$.  Since $\mathspace{X}_c$ is compactly
    embedded in $\mathspace{X}$, we obtain
    $\overline{\ell}_n \rightarrow \tilde{\ell}$ in $\CO{\mathspace{X}}$
    and consequently, $\tilde\ell \in \UU(z_0;M)$.  In addition, the
    strong convergence in $\CO{\mathspace{X}}$ in combination with
    \cref{thm:continuityOfS} and
    \cref{lem:approximationOfTheYosidaApproximation} yields weak
    convergence of the states, i.e.,
    $\mathcal{S}_n(\overline{\ell}_n) \weak \mathcal{S}(\tilde{\ell})$
    in $\WZ{\mathspace{H}}$ and thus also in $\WZ{\mathspace{W}}$.  Now,
    let $\overline{\ell}$ be a global solution of \cref{eq:optprob}. We
    can again use \cref{prp:LInftyBoundAgainstDerivative} and
    \cref{lem:approximationOfTheYosidaApproximation} to obtain
    $\mathcal{S}_n(\overline{\ell}) \rightarrow
    \mathcal{S}(\overline{\ell})$ in $\WZ{\mathspace{H}}$. This,
    together with the weak lower semicontinuity of $\Psi$ and $\Phi$,
    implies
    \begin{equation}\label{eq:convobj}
        \begin{aligned}
            J(\mathcal{S}(\tilde\ell), \tilde{\ell}) &=  \Psi(\mathcal{S}(\tilde{\ell}),\tilde{\ell}) + \Phi(\tilde\ell)
            \leq \liminf_{n \rightarrow \infty}  \Psi(\mathcal{S}_n(\overline{\ell}_n),\overline{\ell}_n) + \Phi(\overline\ell_n) \\
            &\leq \limsup_{n \rightarrow \infty}
            J(\mathcal{S}_n(\overline{\ell}_n),\overline{\ell}_n) \leq
            \limsup_{n \rightarrow \infty}
            J(\mathcal{S}_n(\overline{\ell}),\overline{\ell}) =
            J(\mathcal{S}(\overline{\ell}),\overline{\ell}),
        \end{aligned}
    \end{equation}
    giving in turn the optimality of the weak limit.  
\end{proof}

\begin{corollary}[Strong Approximation of Global Minimizers]\label{cor:strongapprox}
    In addition to \cref{ass:AnAssumption}, assume that
    $\Phi\colon  \WZ{\XX_c} \to \R$ is such that, if a sequence
    $\{\ell_n\}_{n\in \N}$ satisfies $\ell_n \weak \ell$ in $\WZ{\XX_c}$
    and $\Phi(\ell_n) \to \Phi(\ell)$, then $\ell_n \to \ell$ in
    $\WZ{\XX_c}$. Then every weak accumulation point of a sequence of
    globally optimal controls of~\eqref{eq:optprobn} is also a strong
    one.

    Moreover, if in addition, at least one of the following holds
    \begin{itemize}
        \item \cref{ass:AIsSubdifferential} is satisfied, that is
            $A = \partial\phi$, and $\phi$ is continuous on $M$ or

        \item $\Psi\colon  \WZ{\WW} \times \WZ{\XX_c} \to \R$ is such that, if
            sequences $\{z_n\}_{n\in \N}$ and $\{\ell_n\}_{n\in \N}$ satisfy
            $z_n \weak z$ in $\WZ{\HH}$ and $\ell_n \rightarrow \ell$ in
            $\WZ{\XX_c}$ and $\Psi(z_n, \ell_n) \to \Psi(z, \ell)$, then
            $z_n \to z$ in $\WZ{\HH}$,
    \end{itemize}
    then the associated sequence of state also converges strongly in
    $\WZ{\HH}$.
\end{corollary}

\begin{proof} Consider an arbitrary accumulation point $\tilde\ell$
    of a sequence of global minimizers of~\eqref{eq:optprobn}, i.e.,
    $\overline{\ell}_n \weak \tilde\ell$ in $\WZ{\XX_c}$. From the
    previous proof, we know that then~\eqref{eq:convobj} holds, giving
    in turn
    \begin{equation*}
        \Psi(\mathcal{S}_n(\overline{\ell}_n),\overline{\ell}_n) + \Phi(\overline\ell_n) 
        \to \Psi(\mathcal{S}(\tilde{\ell}),\tilde{\ell}) + \Phi(\tilde\ell).
    \end{equation*}
    Since
    $\mathcal{S}_n(\overline{\ell}_n) \weak \mathcal{S}(\tilde\ell)$, as
    seen in the previous proof, and both, $\Psi$ and $\Phi$, are weakly
    lower semicontinuous by assumption, this implies
    $\Phi(\overline{\ell}_n) \to \Phi(\tilde\ell)$ and
    $\Psi(\mathcal{S}_n(\overline{\ell}_n),\overline{\ell}_n) \to
    \Psi(\mathcal{S}(\tilde{\ell}),\tilde{\ell})$.  The hypothesis on
    $\Phi$ thus yields $\overline{\ell}_n \to \tilde\ell$ in
    $\WZ{\XX_c}$ so that $\tilde\ell$ is indeed a strong accumulation
    point as claimed.

    Due to $\XX_c \embed \XX$, the strong convergence carries over to
    $\WZ{\XX}$ and therefore, we deduce from \cref{thm:continuityOfS}
    that
    $\mathcal{S}_{\lambda_n}(\overline{\ell}_n) \to
    \mathcal{S}(\tilde\ell)$ in $\WZ{\HH}$, provided that
    \cref{ass:AIsSubdifferential} is fulfilled and
    $\phi(R\overline\ell_n(0) - Qz_0) \rightarrow \phi(R\tilde\ell(0) -
    Qz_0)$ holds.  If the additional requirements on $\Psi$ are
    fulfilled, we also obtain the strong convergence
    $\mathcal{S}_{\lambda_n}(\overline{\ell}_n) \to
    \mathcal{S}(\tilde\ell)$ in $\WZ{\HH}$, since we already showed
    $\overline{\ell}_n \to \tilde\ell$ in $\WZ{\XX_c}$.  Thus, in both
    cases, \cref{lem:approximationOfTheYosidaApproximation} gives
    $\mathcal{S}_n(\overline{\ell}_n) \to \mathcal{S}(\tilde\ell)$ in
    $\WZ{\HH}$, which is the second assertion.  
\end{proof}

\begin{example}\label{exa:concreteJ}
    Let us assume that $\XX_c$ is a Hilbert space. Then a possible
    objective functional fulfilling the requirements on $\Phi$ in
    \cref{cor:strongapprox} reads as follows:
    \begin{equation*}
        J(z, \ell) = \Psi(z, \ell) + \frac{\alpha}{2}\, \|\ell\|_{\WZ{\XX_c}}^2,
    \end{equation*}
    i.e., $\Phi(\ell) \defn  \alpha/2 \, \|\ell\|_{\WZ{\XX_c}}^2$. Herein,
    $\Psi\colon  \WZ{\HH} \times \WZ{\XX_c} \to \R$ is again lower
    semicontinuous and bounded from below and $\alpha > 0$ is a given
    constant.  Since $\WZ{\XX_c}$ is a Hilbert space, too, weak
    convergence and norm convergence give strong convergence and
    consequently, this specific choice of $\Phi$ fulfills the condition in
    \cref{cor:strongapprox}.
\end{example}

\begin{remark}[Approximation of Local Minimizers]
    By standard localization arguments, the above convergence analysis
    can be adapted to approximate local minimizers.  Following the lines
    of, for instance, \cite{castro02}, one can show that, under the assumptions of
    \cref{cor:strongapprox}, every strict local minimum of
    \cref{eq:optprob} can be approximated by a sequence of local minima
    of \cref{eq:optprobn}.  A local minimizer $\overline{\ell}$ of
    \cref{eq:optprob}, which is not necessarily strict, can be
    approximated by replacing the objective in~\eqref{eq:optprobn} by
    $\overline{J}(z,l) \defn  J(z,l) + \norm{\ell -
    \overline{\ell}}{\WZ{\mathspace{X}_c}}$, which is of course only
    of theoretical interest, cf.~e.g.~\cite{barbu84}.  Since these results
    and their proofs are standard, we omitted them.
\end{remark}

Now that we answered the question of approximation of optimal controls
via regularization, we turn to the regularized problems and derive
optimality conditions for these in the next two sections.

\section{First-Order Optimality Conditions}
\label{sec:5}

In the following, we consider a single element of the sequence of
regularized problems. The associated regularized operator is denoted
by $A_s$ so that the regularized optimal control problems reads as
follows:
\begin{equation}\tag{\mbox{P$_s$}}\label{eq:optprobs}
    \left\{\quad
        \begin{aligned}
            \min \shortspace & J(z,\ell), \\
            \text{ s.t. } \shortspace & \boverdot{z} = A_s(R\ell - Qz),	\mediumspace z(0) = z_0, \\
            & (z,\ell) \in \WZ{\mathspace{H}} \times \big{(}
            \WZ{\mathspace{X}_c} \cap \mathcal{U}(z_0;M) \big{)}.
    \end{aligned}\right.
\end{equation}
Beside our standing assumption and the Lipschitz continuity required
in \cref{ass:AnAssumption}, we need the following additional
assumptions for the derivation of first-order necessary optimality
conditions for~\eqref{eq:optprobs}.  Recall the continuous embeddings
$\YY \embed \ZZ \embed \HH$ from \cref{sec:2}.

\begin{assumption}\label{ass:AsAndJFrechet}\
    \begin{itemize}
        \item[(i)]
            $J \colon  \WZ{\mathspace{W}} \times \WZ{\mathspace{X}_c} \rightarrow
            \mathbb{R}$ is Fr\'{e}chet differentiable.
        \item[(ii)] $A_s \colon  \mathspace{Y} \rightarrow \mathspace{Y}$ is
            Lipschitz continuous and Fr\'{e}chet differentiable from $\YY$ to
            $\ZZ$.  Moreover, $A_s'(y)$ can be extended to elements of
            $L(\mathspace{Z};\mathspace{Z})$ and
            $L(\mathspace{H};\mathspace{H})$, respectively, denoted by the
            same symbol.  There exists a constant $C$ such that these
            extensions satisfy
            $\norm{A_s'(y) z}{\mathspace{Z}} \leq C \,\norm{z}{\mathspace{Z}}$
            and
            $\norm{A_s'(y) h}{\mathspace{H}} \leq C \norm{h}{\mathspace{H}}$
            for all $y \in \mathspace{Y}$, $z \in \mathspace{Z}$, and
            $h \in \mathspace{H}$.
    \end{itemize}
\end{assumption}

\begin{remark}\label{rem:normgap}
    It is well known that a norm gap is often indispensable to ensure
    Fr\'echet differentiability.  This is also the case in our
    application example in \cref{subsec:vonmises}. This is the reason
    for considering two different spaces $\YY$ and $\ZZ$ in context of
    the Fr\'echet differentiability of $A_s$ in
    \cref{ass:AsAndJFrechet}.
\end{remark}

We start the derivation of optimality conditions for~\eqref{eq:optprobs} with the Fr\'echet-de\-ri\-va\-tive of the associated
control-to-state mapping.

\subsection{Differentiability of the Regularized Control-to-State
Mapping}
\label{subsec:5.1}

As $A_s \colon  \YY \to \YY$ is supposed to be Lipschitz continuous and
$R$ and $Q$ are not only linear and continuous as operators with
values in $\HH$, but also in $\YY$ according to our standing
assumptions, Banach's fixed point theorem immediately implies that the
state equation in~\eqref{eq:optprobs}, i.e.,
\begin{equation}\label{eq:auxs}
    \boverdot{z} = A_s(R\ell - Qz), \quad z(0) = z_0,
\end{equation}
admits a unique solution $z \in \WZ{\mathspace{Y}}$ for every right
hand side $\ell \in \LZ{\mathspace{X}}$, provided that $z_0 \in \YY$.
Therefore, similar to above, we can define the associated solution
operator
$\mathcal{S}_s \colon  \LZ{\mathspace{X}} \rightarrow \WZ{\mathspace{Y}}$
(for fixed $z_0 \in \mathspace{Y}$).  We will frequently consider this
operator with different domains, e.g.~$\WZ{\mathspace{X}}$, and
ranges, in particular $\WZ{\mathspace{Z}}$. With a little abuse of
notation, these operators are denoted by the same symbol.

\begin{lemma}[Lipschitz Continuity of $\mathcal{S}_s$]\label{prp:solutionOperatorLipschitz}
    The solution operator $\mathcal{S}_s$ is globally Lipschitz
    continuous from $\LZ{\mathspace{X}}$ to $\WZ{\mathspace{Y}}$.
\end{lemma}

\begin{proof} This can be proven completely analogously to the
    Lipschitz continuity of $\mathcal{S}_n$ from $\LZ{\mathspace{X}}$ to
    $\WZ{\mathspace{Y}}$ in \cref{prop:exregopt}.  
\end{proof}

\begin{lemma}
    \label{lem:frechetEquationExistence}
    Assume that \cref{ass:AsAndJFrechet}(ii) is fulfilled and let
    $y \in \LZ{\mathspace{Y}}$ and $w \in \LZ{\mathspace{Z}}$ be given.
    Then there exists a unique solution $\eta \in \WZ{\mathspace{Z}}$ of
    \begin{equation}\label{eq:lineq}
        \boverdot{\eta} = A_s'(y) (w - Q\eta), \quad \eta(0) = 0.	
    \end{equation}
\end{lemma}

\begin{proof} Let us define
    \begin{align*}
        B \colon  [0, T] \times \mathspace{Z} \rightarrow \mathspace{Z}, \quad (t,\eta)
        \mapsto  A_s'(y(t)) (w(t) - Q\eta )
    \end{align*}
    so that~\eqref{eq:lineq} becomes
    $\boverdot{\eta}(t) = B(t, \eta(t))$ a.e.\ in $(0,T)$,
    $\eta(0) = 0$.  Now, given $\eta \in \LZ{\mathspace{Z}}$,
    $[0, T] \ni t \mapsto B(t,\eta(t)) \in \mathspace{Z}$ is Bochner
    measurable as a pointwise limit of Bochner measurable
    functions. Furthermore, \cref{ass:AsAndJFrechet}(ii) implies for
    almost all $t \in [0, T]$ and all $\eta_1, \eta_2 \in \ZZ$ that
    $\|B(t, \eta_1) - B(t, \eta_2)\|_{\ZZ} \leq C\, \|\eta_1 -
    \eta_2\|_\ZZ$. Therefore, we can apply Banach's fixed point argument
    to the integral equation associated with~\eqref{eq:lineq}, which
    gives the assertion.  
\end{proof}

\begin{theorem}[Fr\'echet differentiability of the regularized
    solution operator]
    \label{thm:SsFrechetDifferentiability}
    Under \cref{ass:AsAndJFrechet}(ii), the solution operator
    $\mathcal{S}_s$ is Fr\'{e}chet differentiable from
    $\WZ{\mathspace{X}}$ to $\WZ{\mathspace{Z}}$.  Its directional
    derivative at $\ell \in \WZ{\mathspace{X}}$ in direction
    $h \in \WZ{\mathspace{X}}$ is given by the unique solution of
    \begin{align}
        \label{eq:etaEquation}
        \boverdot{\eta} = A_s'(R\ell - Qz) (Rh - Q \eta), \quad \eta(0) = 0,
    \end{align}
    where $z \defn  \mathcal{S}_s(\ell) \in \WZ{\mathspace{Y}}$.  Moreover,
    there exists a constant $C$ such that
    $\norm{\mathcal{S}_s'(\ell) h}{\WZ{\mathspace{Z}}} \leq C
    \norm{h}{\LZ{\mathspace{X}}}$ holds for all
    $\ell,h \in \WZ{\mathspace{X}}$.
\end{theorem}

\begin{proof} Let $\ell,h \in \WZ{\mathspace{X}}$ be arbitrary and
    abbreviate $z_h \defn  \mathcal{S}_s(\ell + h)$.  Thanks to
    \cref{lem:frechetEquationExistence}, there exists a unique solution
    $\eta \in \WZ{\mathspace{Z}}$ of \cref{eq:etaEquation}. Clearly, the
    solution operator of \cref{eq:etaEquation} is linear with respect to
    $h$. Moreover, \cref{ass:AsAndJFrechet}(ii) implies for almost all
    $t \in [0,T]$ that
    \begin{equation*}
        \norm{\boverdot{\eta}(t)}{\mathspace{Z}} \leq C \big{(} \norm{h(t)}{\mathspace{X}}
        + \norm{\eta(t)}{\mathspace{Z}} \big{)},
    \end{equation*}
    so that Gronwall's inequality gives
    $\norm{\eta}{\WZ{\mathspace{Z}}} \leq C
    \norm{h}{\LZ{\mathspace{X}}}$, i.e., the continuity of the solution
    operator of \cref{eq:etaEquation}.  This also proves the asserted
    inequality (after having proved that
    $\eta = \mathcal{S}_s'(\ell) h$, which we do next).

    It remains to verify the remainder term property. For this purpose,
    let us denote the remainder term of $A_s$ by $r_1$, i.e.,
    \begin{equation*}
        A_s(y + \zeta) = A_s(y) + A_s'(y)\zeta + r_1(y;\zeta) \quad \text{with} \quad 
        \frac{\|r_1(y;\zeta)\|_\ZZ}{\|\zeta\|_\YY} \to 0 \text{ as } \zeta \to 0 \text{ in }\YY.
    \end{equation*}		
    Moreover, we abbreviate
    \begin{equation*}
        y \defn  R\ell - Qz \in \WZ{\YY} \quad \text{and} \quad \zeta \defn  Rh - Q(z_h - z) \in \WZ{\YY}.
    \end{equation*}    	
    Then, in view of the definition of $z$, $z_h$, and $\eta$ (as
    solution of~\eqref{eq:etaEquation}), we find for almost all
    $t \in [0,T]$
  \begin{align*}
    \norm{\boverdot{z}_h(t) - \boverdot{z}(t) - \boverdot{\eta}(t)}{\mathspace{Z}}
    &= \norm{A_s(y(t) + \zeta(t)) - A_s(y(t)) - A_s'(y(t)) (\zeta(t) + Q(z_h(t) - z(t) - \eta(t)))}{\mathspace{Z}} \\
    &\leq \|A_s'(y(t)) Q(z_h(t) - z(t) - \eta(t))\|_{\ZZ} + \| r_1(y(t);\zeta(t))\|_{\ZZ}.
  \end{align*}
    Hence, \cref{ass:AsAndJFrechet}(ii) and Gronwall's inequality yield
    \begin{equation}\label{eq:remS}
        \norm{z_h - z - \eta}{\WZ{\mathspace{Z}}} \leq C \,\norm{r_1(y;\zeta)}{\LZ{\mathspace{Z}}}.
    \end{equation}
    (note that $r_1(y;\zeta) \in \LZ{\ZZ}$ by its definition as
    remainder term).  Furthermore, thanks to
    \cref{prp:solutionOperatorLipschitz} and the definition of $\zeta$,
    we obtain
    \begin{equation}\label{eq:zetaest}
        \|\zeta\|_{\WZ{\YY}} \leq C \, \|h\|_{\WZ{\XX}}
    \end{equation}
    such that $h \to 0$ in $\WZ{\XX}$ implies $\zeta \to 0$ in
    $\WZ{\YY}$. The continuous embedding $\WZ{\YY}\embed \CO{\YY}$ and
    the remainder term property of $r_1$ thus give for almost all
    $t\in (0,T)$ that
    \begin{equation}\label{eq:rempointw}
        \frac{\|r_1(y(t);\zeta(t))\|_{\ZZ}}{\|h\|_{\WZ{\XX}}} 
        \leq C\,\frac{\|r_1(y(t);\zeta(t))\|_{\ZZ}}{\|\zeta(t)\|_{\YY}} \, \frac{\|\zeta\|_{\WZ{\YY}}}{\|h\|_{\WZ{\XX}}}  \to 0
    \end{equation}
    as $h \to 0$ in $\WZ{\XX}$.  Moreover, the Lipschitz continuity of
    $A_s\colon  \YY \to \YY$ together with~\ref{ass:AsAndJFrechet}(ii),
    $\YY\embed \ZZ$, and~\eqref{eq:zetaest} yield for almost all
    $t\in (0,T)$ that
    \begin{equation*}
        \frac{\|r_1(y(t);\zeta(t))\|_{\ZZ}}{\|h\|_{\WZ{\XX}}} 
        = \frac{\|(A_s(y + \zeta) - A_s(y) - A_s'(y)\zeta)(t)\|_{\ZZ}}{\|h\|_{\WZ{\XX}}} 
        \leq C\, \frac{\|\zeta(t)\|_{\YY}}{\|h\|_{\WZ{\XX}}} \leq C.
    \end{equation*}
    In combination with~\eqref{eq:rempointw} and Lebesgue's dominated
    convergence theorem, this yields
    \begin{equation*}
        \frac{\|r_1(y;\zeta)\|_{\LZ{\ZZ}}}{\norm{h}{\WZ{\mathspace{X}}}} \to 0
    \end{equation*}
    as $h \rightarrow 0$ in $\WZ{\mathspace{X}}$, which, in view
    of~\eqref{eq:remS} finishes the proof.  
\end{proof}

\begin{remark}
    It is to be noted that we did not employ the implicit function
    theorem to show the differentiability of $\mathcal{S}_s$.  The
    reason is that $H\colon  z\mapsto \boverdot z - A_s(R\ell - Q z)$ is
    Fr\'echet differentiable from $\WZ{\YY}$ to $\LZ{\ZZ}$, but the
    derivative $H'(z)$ is not continuously invertible in these spaces,
    cf.\ \cref{lem:frechetEquationExistence}. On the other hand, $H$ is
    not differentiable from $\WZ{\YY}$ to $\LZ{\YY}$ (due to the
    differentiability properties of $A_s$, see \cref{rem:normgap}),
    which would be the right spaces for the existence result from
    \cref{lem:frechetEquationExistence}.  The same observation for a
    more abstract setting was already made in~\cite{wachsmuth2}.
\end{remark}

\subsection{Adjoint Equation}
\label{subsec:5.2}

Now that we know that the (regularized) control-to-state map is
G\^ateaux differentiable, we can apply the standard adjoint approach
to derive first-order necessary optimality conditions in form of a
Karush-Kuhn-Tucker (KKT) system.  To keep the discussion concise, we
restrict our analysis to the case without further control constraints.
To be more precise, we require the following:

\begin{assumption}\label{assu:controlconstr}
    Let $z_0 \in \mathspace{Y}$ such that $-Qz_0 \in D(A)$.  The set $M$
    in the definition of the set of admissible controls is given by the
    singleton $M = \{ -Qz_0 \}$ such that
    \begin{equation*}
        \mathcal{U} \defn  \mathcal{U}\bigl(z_0;\{ -Qz_0 \}\bigr) =
        \bigl\{ \ell \in \WZ{\mathspace{X}} \colon  \ell(0) \in \ker R
        \bigr\}. 
    \end{equation*}
\end{assumption}

Note that $\UU$ is a linear subspace of $\WZ{\XX}$.

\begin{remark}[Additional Control Constraints]\label{rem:ctrlconstr}
    One could allow for additional control constraints in our analysis,
    even more complex ones than the ones covered by $\UU(z_0;M)$ such as
    for instance box constraints over the whole time interval or
    vanishing initial and final loading, i.e., $\ell(0) = \ell(T) = 0$,
    which is certainly meaningful for many practically relevant
    problems.  However, since the differentiability of the
    control-to-state map is the essential issue in the derivation of
    optimality conditions and additional (convex and closed) control
    constraints can be incorporated by standard argument, we decided to
    leave them out in order to keep the discussion as concise as
    possible.

    However, without any further assumptions, the existence of solutions
    to the unregularized state equation~\eqref{eq:aux} cannot be
    guaranteed. To be more precise, one needs that
    $R\ell(0) - Q z_0 \in D(A)$, see \cref{thm:auxExistence}, which
    holds in case of $\UU$, provided that $-Q z_0 \in D(A)$. This is the
    reason for considering the set $\WZ{\XX_C} \cap \UU$ as set of
    admissible controls in the rest of the paper.

    Note moreover that, if the operator $R$ is injective (which is the
    case in \cref{sec:7}), then
    $\UU = \{ \ell \in \WZ{\mathspace{X}} \colon  \ell(0) = 0 \}$.
\end{remark}

The chain rule immediately gives that the reduced objective defined by
\begin{equation}\label{eq:redobj}
    F \colon  \WZ{\mathspace{X}_c}
    \to\mathbb{R}, \quad \ell \mapsto J(\mathcal{S}_s(\ell),\ell)
\end{equation}
is Fr\'echet differentiable, too. Thus, by standard arguments, one
derives the following

\begin{lemma}[Purely Primal Necessary Optimality Conditions]
    \label{lem:generalFirstOrderOptimalityCondition}
    Let \cref{ass:AsAndJFrechet} and \cref{assu:controlconstr} hold.
    Then, if a control $\overline{\ell} \in \WZ{\XX_c} \cap \UU$ with
    associated state $\overline{z} = \mathcal{S}_s(\overline\ell)$ is
    locally optimal for \cref{eq:optprobs}, then
    \begin{equation}
        \label{eq:objectiveFunctionDerivative}
        F'(\overline{\ell}) h
        = J'_z(\overline{z},\overline{\ell}) \mathcal{S}'_s(\overline{\ell}) h 
        + J'_\ell(\overline{z},\overline{\ell}) h = 0,
    \end{equation}
    for all $h \in \WZ{\mathspace{X}_c} \cap \mathcal{U}$.
\end{lemma}

Next, we reformulate~\eqref{eq:objectiveFunctionDerivative} in terms
of a KKT-system by introducing an adjoint equation, which formally
reads
\begin{equation}
    \label{eq:prototypeAdjointEquation}
    \boverdot{\varphi} =  Q A'_s(y)^* \varphi + v, \quad \varphi(T) = 0.
\end{equation}
Depending on the regularity of the right hand side $v$, we define
different notions of solutions:

\begin{definition}
    Let $y \in \LZ{\mathspace{Y}}$ and $v \in \WZ{\mathspace{H}}^*$ be
    given.  A function $\varphi \in \LZ{\mathcal{H}}$ is called
    \emph{weak solution} of~\eqref{eq:prototypeAdjointEquation}, if
    \begin{equation}
        \label{eq:prototypeAdjointEquationDefinition}
        -\scalarproduct{\varphi}{\boverdot{\eta}}{\LZ{\mathspace{H}}}
        = \scalarproduct{\varphi}{A'_s(y) Q \eta}{\LZ{\mathspace{H}}}
        + v(\eta)
    \end{equation}
    holds for all $\eta \in \WZ{\mathspace{H}}$ with $\eta(0) = 0$.

    If $v$ takes the form
    \begin{equation}\label{eq:splitup}
        v(\eta) = (v_1, \eta)_{\LZ{\HH}} + (v_2, \eta(T))_\HH
    \end{equation}		
    with some $v_1 \in \LZ{\HH}$ and $v_2 \in \HH$, then we call
    $\varphi \in \WZ{\HH}$ \emph{strong solution} of~\eqref{eq:prototypeAdjointEquation}, if, for almost all
    $t\in (0,T)$,
    \begin{equation}\label{eq:adjointstrong}
        \boverdot{\varphi}(t) = \big(Q A'_s(y)^*\varphi\big)(t) + v_1(t) \quad \text{in } \HH, 
        \quad \varphi(T) = - v_2 \quad \text{in } \HH.
    \end{equation}
\end{definition}

In the following, we will---as usual---identify $v\in \LZ{\HH}$ with
an element of $\WZ{\HH}^*$ via $(v, \cdot\, )_{\LZ{\HH}}$ and denote
this element with a slight abuse of notation by the same symbol.

\begin{lemma}\label{lem:adex}
    Let $y \in \LZ{\mathspace{Y}}$ and $v \in \WZ{\mathspace{H}}^*$.
    Then there is a unique weak solution of
    \cref{eq:prototypeAdjointEquation}, which is given by
    $\varphi \defn  -v \circ \mathcal{S}_y \in \LZ{\mathspace{H}}^* =
    \LZ{\mathspace{H}}$, where
    $\mathcal{S}_y \colon  \LZ{\mathspace{H}} \rightarrow \WZ{\mathspace{H}}$
    is the solution operator of
    \begin{align}
        \label{eq:specialEta}
        \boverdot{\eta} = -A'_s(y) Q \eta + w, \quad \eta(0) = 0,
    \end{align}
    that is, $\mathcal{S}_y(w) = \eta$.

    Moreover, if $v$ is of the form~\eqref{eq:splitup}, then there
    exists a unique strong solution of~\eqref{eq:prototypeAdjointEquation}, and the weak and the strong
    solution coincide.
\end{lemma}
\begin{proof} At first note that the existence of a solution of
    \cref{eq:specialEta} can be proven exactly as in
    \cref{lem:frechetEquationExistence}.  Let
    $\eta \in \WZ{\mathspace{H}}$ with $\eta(0) = 0$ be arbitrary and
    define
    $w \defn  \boverdot{\eta} + A'_s(y) Q \eta \in \LZ{\mathspace{H}}$,
    hence, $\eta = \mathcal{S}_y(w)$. By the definition of $w$ and
    $\varphi$, it follows that
    \begin{equation*}
        (\boverdot{\eta} + A'_s(y) Q \eta, \varphi)_{\LZ{\HH}} =  (\varphi, w)_{\LZ{\HH}} = - v(\mathcal{S}_y(w)) = - v(\eta),
    \end{equation*}    	
    i.e., \cref{eq:prototypeAdjointEquationDefinition} holds. Since
    $\eta$ was arbitrary, we see that $\varphi$ is a weak solution of
    \cref{eq:prototypeAdjointEquation}.

    To prove uniqueness, let $\tilde{\varphi} \in \LZ{\mathspace{H}}$ be
    another weak solution.  Then, we choose an arbitrary
    $w \in \LZ{\mathspace{H}}$ and set $\eta \defn  \mathcal{S}_y(w)$ to see
    that
    \begin{equation*}
        \scalarproduct{\varphi}{w}{\LZ{\mathspace{H}}}
        = -v(\eta) = \scalarproduct{\tilde{\varphi}}{\boverdot\eta+  A'_s(y) Q \eta}{\LZ{\mathspace{H}}} 
        = \scalarproduct{\tilde{\varphi}}{w}{\LZ{\mathspace{H}}},
    \end{equation*}
    and therefore $\varphi = \tilde{\varphi}$.

    Now we turn to the strong solution and suppose that $v$ is as given
    in~\eqref{eq:splitup}.  Existence and uniqueness of a strong
    solution can again be shown by means of Banach's fixed point
    theorem. To this end, let us consider the affine-linear operator
    \begin{equation*}
        B\colon  [0,T] \times \ZZ \to \ZZ, \quad B(t,\varphi) =  Q A'_s(y(t))^* \varphi + v_1(t).
    \end{equation*}
    Since $\HH$ is separable by our standing assumptions, we can apply~\cite[Chap.~IV, Thm.~1.4]{gajewski} to obtain that, for every
    $\varphi \in \LZ{\HH}$, the mapping $(0,T) \mapsto B(t, \varphi(t))$
    is Bochner measureable. Moreover, since
    $\|A_s'(y)^*\|_{L(\HH;\HH)} = \|A_s'(y)\|_{L(\HH;\HH)}$,
    \cref{ass:AsAndJFrechet}(ii) yields that $B$ is also Lipschitz
    continuous w.r.t.\ the second variable for almost all $t\in
    (0,T)$. Therefore, similarly to the proof of
    \cref{lem:frechetEquationExistence}, one can apply Banach's fixed
    point theorem to the integral equation associated with~\eqref{eq:adjointstrong} to establish the existence of a unique
    strong solution.

    Finally, if we test~\eqref{eq:adjointstrong} with an arbitrary
    $\eta \in \WZ{\HH}$ with $\eta(0) = 0$ and integrate by parts, then
    we see that every strong solution is also a weak solution. Since the
    latter one is unique, as seen above, we deduce that weak and strong
    solution coincide.  
\end{proof}

\begin{theorem}[KKT-Conditions for~\eqref{eq:optprobs}]
    \label{thm:firstOrderOptimalityCondition}
    Assume that \cref{ass:AsAndJFrechet} and \cref{assu:controlconstr}
    hold and let
    $\overline{\ell} \in \WZ{\mathspace{X}_c} \cap \mathcal{U}$ be a
    locally optimal control for \cref{eq:optprobs} with associated state
    $\overline{z} = \mathcal{S}_s(\overline{\ell})$.  Then there exists
    a unique adjoint state $\varphi \in \LZ{\mathspace{H}}$ such that
    the following optimality system is fulfilled
    \begin{subequations}\label{eq:firstOrderOptimalitySystem}
        \begin{alignat}{3}
            & \boverdot{\overline{z}} = A_s(R\overline{\ell} - Q\overline{z}), & \quad & \overline{z}(0) = z_0, \\
            & \left\{\begin{aligned}
                    & -\scalarproduct{\varphi}{\boverdot{\eta}}{\LZ{\mathspace{H}}} \\
                    &\quad = \scalarproduct{\varphi}{A'_s(R\overline{\ell} -
                    Q\overline{z}) Q \eta}{\LZ{\mathspace{H}}} +
                    J'_z(\overline{z},\overline{\ell})\eta
            \end{aligned}\right.
            & &
            \begin{aligned}
                \forall \, \eta \in \WZ{\HH}\colon  \quad &\\
                \eta(0) = 0&
            \end{aligned} \label{eq:adjeq}\\[1ex]
            & \scalarproduct{\varphi}{A_s'(R\overline{\ell} - Q\overline{z})
            Rh}{\L{2}{\mathspace{H}}} =
            J'_\ell(\overline{z},\overline{\ell}) h & & \forall \, h \in
            \WZ{\mathspace{X}_c} \cap \mathcal{U}. \label{eq:gradeq}
        \end{alignat}
    \end{subequations}
    If $J$ enjoys extra regularity, namely
    \begin{equation}\label{eq:regpsi}
        J(z, \ell) = \Psi_1(z, \ell) + \Psi_2(z(T), \ell(T)) + \Phi(\ell) 
    \end{equation}        
    with two Fr\'echet differentiable functionals
    $\Psi_1\colon  \LZ{\HH} \times \WZ{\XX_c}\to \R$ and
    $\Psi_2 \colon  \HH \times \XX_c \to \R$, then
    $\varphi \in \WZ{\mathspace{H}}$ is a strong solution of
    \begin{equation}\label{eq:adjstrong}
        \boverdot{\varphi}(t) = \big(Q A'_s(R\overline{\ell} - Q\overline{z})^*\varphi\big)(t) + \ddp{\Psi_1}{z}(\overline{z},\overline{\ell}), 
        \qquad \varphi(T) = - \ddp{\Psi_2}{z}(\overline{z}(T), \overline{\ell}(T)).
    \end{equation}
\end{theorem}

\begin{remark}
    The exemplary objective functionals in \cref{sec:7} are precisely of
    the form in~\eqref{eq:regpsi}.
\end{remark}

\begin{proof}[Proof of \cref{thm:firstOrderOptimalityCondition}] Since
    $ J'_z(\overline{z},\overline{\ell}) \in \WZ{\WW}^* \embed
    \WZ{\HH}^*$, \cref{lem:adex} gives the existence of a unique
    solution of~\eqref{eq:adjeq}.  Now, let
    $h \in \WZ{\mathspace{X}_c} \cap \mathcal{U}$ be arbitrary and
    define
    $\eta \defn  \mathcal{S}'_s(\overline{\ell}) h \in \WZ{\mathspace{Z}}
    \subset \WZ{\mathspace{H}}$.  The weak form of the adjoint equation
    then implies
    \begin{equation}
        \label{eq:adjointEquality}
        \begin{aligned}
            \scalarproduct{\varphi}{A_s'(R\overline{\ell} - Q\overline{z})
            Rh }{\LZ{\mathspace{H}}}
            = \scalarproduct{\varphi}{\boverdot{\eta} + A_s'(R\overline{\ell} - Q\overline{z}) Q\eta}{\LZ{\mathspace{H}}}
            = -J'_z(\mathcal{S}_s(\overline{\ell}),\overline{\ell})\eta.
        \end{aligned}
    \end{equation}
    This together with \cref{lem:generalFirstOrderOptimalityCondition}
    shows that $(\overline{z},\overline{\ell},\varphi)$ fulfills the
    optimality system~\eqref{eq:firstOrderOptimalitySystem}.  If $J$ is
    of the form in~\eqref{eq:regpsi}, then \cref{lem:adex} implies that
    the weak solution of the adjoint equation is in fact a strong
    solution and solves~\eqref{eq:adjstrong}.  
\end{proof}

\begin{corollary}
    \label{cor:derivativeZeroEquivalence}
    Let \cref{ass:AsAndJFrechet} and \cref{assu:controlconstr} hold.
    Then $\bar\ell \in \WZ{\mathspace{X}_c} \cap \mathcal{U}$ with
    associated state $\overline{z} = \mathcal{S}_s(\overline{\ell})$
    fulfills~\eqref{eq:objectiveFunctionDerivative} if and only if there
    exists an adjoint state $\varphi \in \LZ{\mathspace{H}}$ such that
    $(\overline{z},\overline{\ell},\varphi)$ satisfies the optimality
    system~\eqref{eq:firstOrderOptimalitySystem}.
\end{corollary}

\begin{proof} The proof of \cref{thm:firstOrderOptimalityCondition}
    already shows that~\eqref{eq:objectiveFunctionDerivative} implies
    the optimality system in~\eqref{eq:firstOrderOptimalitySystem}.

    To prove the reverse implication, assume that
    $(\overline{z},\overline\ell,\varphi)$ fulfills the optimality
    system \cref{eq:firstOrderOptimalitySystem}.  Then choose an
    arbitrary $h \in \WZ{\mathspace{H}}$, define
    $\eta \defn  \mathcal{S}_s'(\overline\ell) h$, and use the fact that
    $\varphi$ is the weak solution of~\eqref{eq:adjeq} to obtain
    \cref{eq:adjointEquality}.  This together with~\eqref{eq:gradeq}
    finally give \cref{eq:objectiveFunctionDerivative}.  
\end{proof}

\begin{example}\label{ex:ode}
    Under suitable additional assumptions, it is possible to further
    simplify the gradient equation~\eqref{eq:gradeq}.  For this purpose
    assume that $R$ is injective (so that
    $\UU = \{\ell \in H^1(0,T,\XX) \colon  \ell(0) = 0\}$), $\XX_c$ is a
    Hilbert space, and
    \begin{equation}\label{eq:regpsi2}
        J(z, \ell) = \Psi_1(z, \ell) + \Psi_2(z(T), \ell(T)) + \frac{\gamma}{2}\,\|\boverdot{\ell}\|_{\LZ{\XX_c}}^2, 
    \end{equation}
    where $\Psi_1\colon  \WZ{\mathcal{W}} \times L^2(0,T;\XX_c) \to \R$ and
    $\Psi_2\colon  \mathcal{W} \times \XX_c \to \R$ are
    Fr\'echet differentiable and $\gamma > 0$. This type of objective
    will also appear in the application problem in \cref{sec:7}. Then~\eqref{eq:gradeq} becomes
    \begin{equation}\label{eq:odeweak}
        \begin{aligned}[t]
            \gamma (\partial_t \overline{\ell},
            \partial_t\overline{h})_{\LZ{\XX_c}}
            & - \int_0^T \dual{R^*A_s'(R\overline{\ell} - Q\overline{z})^* \varphi}{h}_{\XX^*, \XX} \, d t \\
            & + \int_0^T \ddp{\Psi_1}{\ell}(\overline{z},
            \overline{\ell})h\,dt + \ddp{\Psi_2}{\ell}(\overline{z}(T),
            \overline{\ell}(T))h(T)
            = 0 \\
            & \qquad\qquad\qquad\qquad\qquad \forall\, h \in
            \WZ{\mathspace{X}_c} \text{ with } h(0) = 0,
        \end{aligned}
    \end{equation}
    where we identified
    $\partial_\ell\Psi_1(\overline{z}, \overline{\ell}) \in
    (\LZ{\XX_c})^* = \LZ{\XX_c}$.  Note that $\XX_c$ as a Hilbert space
    satisfies the Radon-Nikod\'{y}m-property.  Since $\XX_c \embed \XX$, we
    may identify $R^*A_s'(R\overline{\ell} - Q\overline{z})^* \varphi$
    with an element of $\LZ{\XX_c}$, too, which we denote by the same
    symbol.  Then, if we choose $h(t) = \psi(t) \, \xi$ with
    $\psi \in C^\infty_c(0,T)$ and $\xi \in \XX_c$ arbitrary, we obtain
    \begin{equation*}
        \Big( - \int_0^T \Big[\gamma\,\partial_t \psi\, \partial_t\overline{\ell} 
                + R^*A_s'(R\overline{\ell} - Q\overline{z})^* \varphi \, \psi 
        - \ddp{\Psi_1}{\ell}(\overline{z}, \overline{\ell})\psi\Big] d t , \xi\Big)_{\XX_c} = 0.
    \end{equation*}
    Now, since $\xi \in \XX_c$ was arbitrary, we find that the second
    distributional time derivative of $\overline{\ell}$ is a regular
    distribution in $\LZ{\XX_c}$, i.e.,
    $\overline{\ell} \in H^2(0,T;\XX_c)$, satisfying for almost all
    $t\in (0,T)$
    \begin{equation}\label{eq:odestrong}
        \gamma\, \partial_t^2 \overline\ell(t) + R^* A_s'\big(R\overline\ell(t) - Q\overline z(t)\big)^* \varphi(t) = 
        \ddp{\Psi_1}{\ell}(\overline{z}, \overline{\ell})(t) \quad \text{in }\XX_c.
    \end{equation}
    Since $\XX_c$ is supposed to be a Hilbert space, we can apply
    integration by parts to~\eqref{eq:odeweak}.  Together with
    $\overline{\ell} \in \UU = \{\ell \in \WZ{\XX} : \ell(0) = 0\}$ and~\eqref{eq:odestrong}, this implies the following boundary
    conditions:
    \begin{equation}\label{eq:boundary}
        \overline\ell(0) = 0, \quad \gamma\,\partial_t\overline{\ell}(T) = - \ddp{\Psi_2}{\ell}(\overline{z}(T), \overline{\ell}(T)),
    \end{equation}
    where we again identified
    $\partial_\ell \Psi_2(\overline{z}(T), \overline{\ell}(T)) \in
    \XX_c^*$ with its Riesz representative.  In summary, we have thus
    seen that the gradient equation in~\eqref{eq:gradeq} becomes an
    operator boundary value problem in $\XX_c$, namely~\eqref{eq:odestrong}--\eqref{eq:boundary}.
\end{example}

\section{Second-Order Sufficient Conditions}
\label{sec:6}

The next section is devoted to the derivation of second-order
sufficient optimality conditions for the regularized problem~\eqref{eq:optprobs}.  As it was the case for the first-order
conditions, the main part concerns the differentiability properties of
the control-to-state map $\mathcal{S}_s$ and the reduced objective, to
be more precise to show that these are twice continuously
Fr\'echet differentiable.  For this purpose, we need the following
sharpened assumptions on the objective and the regularized operator
$A_s$:

\begin{assumption}\label{ass:AsAndJTwiceFrechet}\
    \begin{itemize}
        \item[(i)]
            $J \colon  \WZ{\mathspace{W}} \times \WZ{\mathspace{X}_c} \rightarrow
            \mathbb{R}$ is twice continuously Fr\'{e}chet differentiable.
        \item[(ii)] The Fr\'echet-derivative $A_s'$ is Lipschitz continuous
            from $\YY$ to $L(\ZZ;\ZZ)$. Moreover, for every $y\in \YY$,
            $A_s'(y)$ can be extended to an element of
            $L(\mathspace{W};\mathspace{W})$.  The mapping arising in this way
            is Lipschitz continuous from $\YY$ to $L(\WW;\WW)$.  Furthermore,
            there is a constant $C> 0$ such that
            $\norm{A_s'(y) w }{\mathspace{W}} \leq C \norm{w}{\mathspace{W}}$
            hold for all $y \in \mathspace{Y}$ and all $w \in \mathspace{W}$.
        \item[(iii)] $A_s'$ is Fr\'echet differentiable from $\YY$ to
            $L(\ZZ;\WW)$. For all $y\in\YY$, its derivative $A_s''(y)$ can be
            extended to an element of
            $L(\mathspace{Z};L(\mathspace{Z};\mathspace{W}))$ and the mapping
            $y\mapsto A''_s(y)$ is continuous in these spaces. Moreover, there
            exists a constant $C$ such that
            $\norm{A_s''(y) [z_1, z_2]}{\mathspace{W}} \leq C
            \norm{z_1}{\mathspace{Z}} \norm{z_2}{\mathspace{Z}}$ for all
            $y \in \mathspace{Y}$ and all $z_1, z_2 \in \mathspace{Z}$.
    \end{itemize}
\end{assumption}

\begin{remark}\label{rem:normgap2}
    We point out that a second norm gap arises in
    \cref{ass:AsAndJTwiceFrechet}, since $A_s'$ is only
    Fr\'echet differentiable as an operator with values in $\WW$ and not
    in $\ZZ\embed \WW$. This assumption is again motivated by the
    application problem in \cref{sec:7}. The example given there
    demonstrates that such as second norm gap is indeed necessary in
    general, since, given a concrete application, one cannot expect
    $A_s$ to be twice Fr\'echet differentiable in $\YY$, and even not as
    an operator from $\YY$ to $\ZZ$.
\end{remark}

The following proposition addresses the second derivative of the
solution operator under the above assumptions.  Its proof is in
principle completely along the lines of the proof of
\cref{thm:SsFrechetDifferentiability} on the first derivative of
$\mathcal{S}$. We therefore postpone it to \cref{sec:secondderiv}.

\begin{proposition}[Second Derivative of the Solution Operator]
    \label{prp:SpTwiceFrechetDifferentiable}
    Under \cref{ass:AsAndJFrechet}(ii) and
    \cref{ass:AsAndJTwiceFrechet}(ii) \& (iii), the solution operator
    $\mathcal{S}_s \colon  \WZ{\mathspace{X}} \rightarrow \WZ{\mathspace{W}}$
    is twice Fr\'{e}chet differentiable. Given
    $\ell, h_1, h_2 \in \WZ{\mathspace{X}}$, its second derivative
    $\mathcal{S}''_s(\ell)[h_1 , h_2] \in \WZ{\mathspace{W}}$ is given
    by the unique solution of
    \begin{align}
        \label{eq:xiEquation}
        \boverdot{\xi} &= A_s''(R\ell - Qz) [Rh_1 - Q\eta_1, Rh_2 - Q\eta_2] - A_s'(R\ell - Qz) Q \xi, \quad \xi(0) = 0,
    \end{align}
    where $z \defn  \mathcal{S}_s(\ell) \in \WZ{\mathspace{Y}}$ and
    $\eta_{i} \defn  \mathcal{S}'_s(\ell) h_i \in \WZ{\mathspace{Z}}$,
    $i = 1,2$.

    Moreover, there exists a constant $C$ such that
    \begin{equation}\label{eq:S2est}
        \norm{\mathcal{S}''_s(\ell)[h_1,  h_2]}{\WZ{\mathspace{W}}}
        \leq C\norm{h_1}{\WZ{\mathspace{X}}} \norm{h_2}{\WZ{\mathspace{X}}}
    \end{equation}    	
    for all $\ell, h_1, h_2 \in \WZ{\mathspace{X}}$.
\end{proposition}

\begin{lemma}
    \label{lem:Sp''Inequality}
    Assume that \cref{ass:AsAndJFrechet} (ii) and
    \cref{ass:AsAndJTwiceFrechet} (ii) and (iii) are fulfilled.  Then
    there exists a constant $C$ such that
    \begin{equation*}
        \begin{aligned}
            \norm{\mathcal{S}_s''(\ell_1)
            &- \mathcal{S}_s''(\ell_2)}{L(\WZ{\mathspace{X}};L(\WZ{\mathspace{X}};\WZ{\mathspace{W}}))} \\
            &\leq C \big( \norm{A_s''(R\ell_1 - Qz_1) - A_s''(R\ell_2 -
                Qz_2)}{\LZ{L(\mathspace{Z};L(\mathspace{Z};\mathspace{W}))}} +
            \norm{\ell_1 - \ell_2}{\WZ{\mathspace{X}}}\big)
        \end{aligned}
    \end{equation*}
    holds for all $\ell_1,\ell_2 \in \WZ{\mathspace{X}}$, where
    $z_i \defn  \mathcal{S}_s(\ell_i)$, $i=1,2$.
\end{lemma}

\begin{proof} Let $\ell_1, \ell_2, h_1,h_2 \in \WZ{\mathspace{X}}$
    be arbitrary. We again abbreviate $z_i \defn  \mathcal{S}_s(\ell_i)$,
    $\eta_{i,j} \defn  \mathcal{S}'_s(\ell_i) h_j$,
    $\xi_{i} \defn  \mathcal{S}''_s(\ell_i)[h_1, h_2]$, and
    $y_i \defn  R\ell_i - Q z_i$ for $i,j \in \{ 1,2 \}$.  By the equation
    for $\mathcal{S}_s''$, we obtain for almost all $t \in [0,T]$
    \begin{align*}
        \boverdot{\xi}_{1} - \boverdot{\xi}_{2} 
        &= A_s''(y_1) [Rh_1 - Q\eta_{1,1}, Rh_2 - Q\eta_{1,2}] - A_s'(y_1) Q \xi_1\\
        &\quad  - A_s''(y_2) [Rh_1 - Q\eta_{2,1}, Rh_2 - Q\eta_{2,2}] - A_s'(y_2) Q \xi_2\\
        & = \big( A_s''(y_1)(Rh_1 - Q\eta_{1,1}) - A_s''(y_2) (Rh_1 - Q\eta_{2,1})\big) (Rh_2 - Q\eta_{1,2})\\
        &\quad  + A_s''(y_2) [Rh_1 - Q\eta_{2,1}, Q (\eta_{2,2} - \eta_{1,2})]
        + \big( A_s'(y_2) - A_s'(y_1) \big)Q \xi_1 + A_s'(y_2) Q (\xi_2 - \xi_1).
    \end{align*}
    With the help of
    \begin{multline*}
        A_s''(y_1) (Rh_1 - Q\eta_{1,1}) - A_s''(y_2) (Rh_1 - Q\eta_{2,1})\\
        = \big( A_s''(y_1) - A_s''(y_2) \big) (Rh_1 - Q\eta_{1,1}) +
        A_s''(y_2) Q(\eta_{2,1} - \eta_{1,1}),
    \end{multline*}
    and Gronwall's inequality, we thus arrive at
    \begin{align*}
        &\norm{\xi_1 - \xi_2}{\WZ{\mathspace{W}}} \\
        &\leq C \Big[\norm{Rh_1 - Q\eta_{2,1}}{\WZ{\mathspace{Z}}}\norm{\eta_{1,2} - \eta_{2,2}}{\WZ{\mathspace{Z}}}
            + \norm{y_1 - y_2}{\WZ{\mathspace{Y}}}\norm{\xi_1}{\WZ{\mathspace{W}}}\\
            & \qquad\; 
            + \Big(\norm{A_s''(y_1) - A_s''(y_2)}{L^2(0,T;L(\mathspace{Z};L(\mathspace{Z};\mathspace{W})))}
                \norm{Rh_1 - Q\eta_{1,1}}{\WZ{\mathspace{Z}}}  \\
            & \hspace*{30ex} + \norm{\eta_{1,1} - \eta_{2,1}}{\WZ{\mathspace{Z}}}  \Big) 
        \norm{Rh_2 - Q\eta_{1,2}}{\WZ{\mathspace{Z}}} \Big] \\[-1ex]
        &\leq 
        C \Big{(} \norm{A_s''(y_1)  - A_s''(y_2)}{L^2(0,T;L(\mathspace{Z};L(\mathspace{Z};\mathspace{W})))}
        + \norm{\ell_1 - \ell_2}{\WZ{\mathspace{X}}} \Big{)} \norm{h_1}{\WZ{\mathspace{X}}}
        \norm{h_2}{\WZ{\mathspace{X}}},
    \end{align*}
    where we used the estimate in \cref{thm:SsFrechetDifferentiability},~\eqref{eq:S2est}, and the Lipschitz continuity of $\mathcal{S}_s'$
    by \cref{S_p'Lipschitz} in the Appendix.  
\end{proof}

If $A_s''$ were Lipschitz continuous from $\YY$ to
$L(\ZZ;L(\ZZ;\WW))$, then \cref{lem:Sp''Inequality} would immediately
imply the Lipschitz continuity of $\mathcal{S}_s''$. However, to
obtain the continuity of the second derivative, this additional
assumption is not necessary as the following theorem shows:

\begin{theorem}[Second-Order Continuous Fr\'echet Differentiability of
    the Solution Operator]
    \label{thm:SsC2}
    Suppose that \cref{ass:AsAndJFrechet}(ii) and
    \cref{ass:AsAndJTwiceFrechet}(ii) \& (iii) are fulfilled. Then
    $\mathcal{S}_s \colon  \WZ{\mathspace{X}} \rightarrow \WZ{\mathspace{W}}$
    is twice continuously Fr\'{e}chet differentiable.  Its second
    derivative at $\ell\in \WZ{\XX}$ in directions
    $h_1, h_2 \in \WZ{\mathspace{X}}$ is given by the unique solution of
    \cref{eq:xiEquation}.
\end{theorem}

\begin{proof} Thanks to \cref{prp:SpTwiceFrechetDifferentiable}, we
    only have to show that
    $\mathcal{S}_s''$
    is continuous from
    $\WZ{\mathspace{X}}$
    to
    $L(\WZ{\mathspace{X}};L(\WZ{\mathspace{X}};\WZ{\mathspace{W}}))$.
    For this let
    $\sequence{\ell}{n} \subset \WZ{\mathspace{X}}$ and
    $\ell \in \WZ{\mathspace{X}}$ be given such that
    $\ell_n \rightarrow \ell$ in $\WZ{\mathspace{X}}$ so that in
    particular $\ell_n \rightarrow \ell$ in $\CO{\mathspace{X}}$. Then,
    \cref{prp:solutionOperatorLipschitz} implies
    $z_n \defn  \mathcal{S}_s(\ell_n) \rightarrow \mathcal{S}_s(\ell) \nfed  z$
    in $\CO{\mathspace{Y}}$. With this convergence results at hand, we
    can apply \cref{lem:compactSequenceSets} with $M = [0,T]$,
    $N = \mathspace{Y}$, $G_n = R\ell_n - Qz_n$ and $G = R\ell - Qz$ to
    see that
    \begin{equation*}
        U \defn  \Big{(} \bigcup_{n = 1}^\infty (R\ell_n - Qz_n)([0,T]) \Big{)} \cup
        \Big{(} (R\ell - Qz)([0,T]) \Big{)}
    \end{equation*}
    is compact. Therefore, thanks to the continuity assumption in
    \cref{ass:AsAndJTwiceFrechet}(iii),
    \linebreak  
    $A_s'' \colon  \mathspace{Y} \rightarrow
    L(\mathspace{Z};L(\mathspace{Z};\mathspace{W}))$ is uniformly
    continuous on $U$. Consequently, $A_s''(R\ell_n - Qz_n)$ converges
    to $A_s''(R\ell - Qz)$ in
    $\CO{L(\mathspace{Z};L(\mathspace{Z};\mathspace{W}))}$, which,
    together with \cref{lem:Sp''Inequality}, yields the assertion.  
\end{proof}

\begin{remark}
    It is to be noted that the regularized state equation~\eqref{eq:auxs} and the equations corresponding to the derivatives
    of $\mathcal{S}_s$, i.e.,~\eqref{eq:etaEquation} and~\eqref{eq:xiEquation}, provide more regular solutions under the
    hypotheses of \cref{ass:AsAndJFrechet}(ii) and
    \cref{ass:AsAndJTwiceFrechet}(ii) \& (iii).  Indeed, if
    $\ell,h_1,h_2 \in \WZ{\mathspace{X}}$, then the solutions of all
    three equations can be shown to be continuously differentiable in
    time with values in the respective spaces ($\YY$, $\ZZ$, and $\WW$,
    respectively).  Moreover, the time derivatives of $z$ and $\eta$ are
    absolutely continuous and the same would hold for $\xi$, if $A_s''$
    were Lipschitz continuous. However, we did not exploit this
    additional regularity, since the original unregularized problem~\eqref{eq:optprob} does not provide this property in general.
\end{remark}

With the above differentiability result at hand, it is now standard to
derive the following:

\begin{theorem}[Second-Order Sufficient Optimality Conditions for~\eqref{eq:optprobs}]
    \label{thm:SecondOrderOptimalityCondition}
    Assume that \cref{ass:AsAndJFrechet}, \cref{assu:controlconstr}, and
    \cref{ass:AsAndJTwiceFrechet} hold.  Let
    $(\overline{z},\overline{\ell}, \varphi) \in \WZ{\mathspace{Y}}
    \times (\WZ{\mathspace{X}_c} \cap \mathcal{U}) \times
    \LZ{\mathspace{H}}$ be a solution of the optimality system~\eqref{eq:firstOrderOptimalitySystem}.  Moreover, suppose that there
    is a $\delta > 0$ such that
    \begin{equation}\tag{SSC}\label{eq:ssc}
        F''(\overline{\ell}) h^2 \geq \delta \norm{h}{\WZ{\mathspace{X}_c}}^2
    \end{equation}
    for all $h \in \WZ{\mathspace{X}_c} \cap \mathcal{U}$, where $F$ is
    the reduced objective from~\eqref{eq:redobj}.  Then
    $(\overline{z},\overline{\ell})$ is locally optimal for
    \cref{eq:optprobs} and there exist $\varepsilon > 0$ and $\tau > 0$
    such that the following \emph{quadratic growth condition}
    \begin{equation}\label{eq:qgc}
        F(\ell) \geq F(\overline{\ell}) + \tau \norm{\ell - \overline{\ell}}{\WZ{\mathspace{X}_c}}^2
    \end{equation}
    holds for all $\ell \in \WZ{\mathspace{X}_c} \cap \mathcal{U}$ with
    $\|\ell - \overline{\ell}\|_{\WZ{\XX_c}} \leq \varepsilon$.
\end{theorem}

\begin{proof} Thanks to the assumptions on $J$ and \cref{thm:SsC2},
    the chain rule implies that the reduced objective function
    $F(\cdot) = J(\mathcal{S}_s(\cdot),\cdot) \colon  \WZ{\mathspace{X}_c}
    \rightarrow \mathbb{R}$ is twice continuously
    Fr\'{e}chet differentiable and, according to
    \cref{cor:derivativeZeroEquivalence}, the equation in~\eqref{eq:objectiveFunctionDerivative} holds for all
    $h \in \WZ{\mathspace{X}_c} \cap \mathcal{U}$. Since $\UU$ is a
    linear subspace, the claim then follows from standard arguments, see
    e.g.~\cite[Satz~4.23]{troltzsch}.  
\end{proof}

\begin{remark}
    As already mentioned in \cref{rem:ctrlconstr}, one could also
    account for additional control constraints. In this case, a critical
    cone would arise in the second-order conditions,
    cf.~e.g.~the survey article \cite{castro15}.
\end{remark}

Using the adjoint equation, the second derivative of the reduced
objective in~\eqref{eq:ssc} can be reformulated as follows:

\begin{corollary}\label{cor:ssc}
    Assume in addition to the hypotheses of
    \cref{ass:AsAndJTwiceFrechet}(iii) that
    $\|A''_s(y)[z_1, z_2]\|_\HH \leq C\, \|z_1\|_\ZZ \|z_2\|_{\ZZ}$ for
    all $y\in \YY$ and $z_1, z_2\in \ZZ$, i.e., the last inequality in
    \cref{ass:AsAndJTwiceFrechet} holds in $\HH$ instead of the weaker
    space $\WW$.  Then it holds for all $\ell, h \in \WZ{\HH}$ that
    \begin{equation*}
        F''(\ell)h^2 =  \Psi''(z, \ell)(\eta, h)^2 + \Phi''(\ell)h^2 - \big(\varphi, A_s''(R\ell - Q z)(R h - Q\eta)^2\big)_{\LZ{\mathspace{H}}},
    \end{equation*}
    where $z = \mathcal{S}_s(\ell)$, $\eta = \mathcal{S}_s'(\ell)h$, and
    $\varphi$ solves the adjoint equation in~\eqref{eq:adjeq}.
\end{corollary}

\begin{proof} Let us again abbreviate $y = R\ell - Q z$.  According
    to the chain rule, the second derivative of the reduced objective is
    given by
    \begin{equation*}
        \begin{aligned}
            F''(\ell)h^2 &= \tddp{^2}{\ell^2} J(z, \ell)h^2 + \tddp{^2}{z^2}
            J(z, \ell)\eta^2
            + 2\,\tddp{^2}{\ell\partial z} J(z, \ell)[h, \eta] + \tddp{}{z}J(z, \ell) \xi \\
            &=\Psi''(z,\ell)(\eta,h)^2 + \Phi''(\ell)h^2 + \tddp{}{z}
            J(z,\ell) \xi
        \end{aligned}
    \end{equation*}
    with $z = \mathcal{S}_s(\ell)$, $\eta = \mathcal{S}_s'(\ell)h$, and
    $\xi = \mathcal{S}_s''(\ell)h^2$.  Now, since $A''_s(y)$ is a
    bilinear form on $\HH$ by assumption, we obtain that
    $\xi \in \WZ{\HH}$.  Therefore, we are allowed to test the adjoint
    equation in~\eqref{eq:adjeq} (in its weak form) with $\xi$, which
    results in
    \begin{equation*}
        \begin{aligned}
            \tddp{}{z} J(z,\ell) \xi = - (\varphi, \boverdot{\xi} + A'_s(y)
            Q \xi)_{\LZ{\mathspace{H}}} = - (\varphi, A_s''(y)(R h -
            Q\eta)^2)_{\LZ{\mathspace{H}}},
        \end{aligned}
    \end{equation*}
    where we used the precise form of $\mathcal{S}_s''(\ell)$
    in~\eqref{eq:xiEquation} for the last identity.  
\end{proof}

\section{Application to Optimal Control of Homogenized
Elastoplasticity}
\label{sec:7}

In the upcoming sections, we apply the analysis from the previous
sections to an optimal control problem governed by a system of
equations that arise as homogenization limit in elastoplasticity and
was derived in~\cite[Theorem.~2.2]{schweizerHomogenization}. It
describes the evolution of plastic deformation in a material with
periodic microstructure and formally (i.e., in its strong form) reads
as follows:
\begin{subequations}\label{eq:homogenizationEquations}
    \begin{alignat}{2}
        -\nabla_x \cdot \pi \Sigma &= f & \quad &\text{ in } \Omega, \\
        \Sigma &= \mathbb{C} (\nabla_x^s u + \nabla_y^sv - Bz) &&\text{ in } \Omega \times Y, \\
        -\nabla_y \cdot \Sigma &= 0  &&\text{ in } \Omega \times Y, \\
        \boverdot{z} &\in A(B^\top \Sigma - \mathbb{B}z) &&\text{ in } \Omega \times Y, \label{eq:evolstrong}\\
        u &= 0  &&\text{ on } \Gamma_D, \\
        \nu \cdot \pi \Sigma  &= g &&\text{ on } \Gamma_N, \\
        z(0) &= z_0 && \text{ in } \Omega \times Y.
    \end{alignat}
\end{subequations}
Herein, $\Omega \subset \R^d$, $d=2,3$, is a given domain occupied by
the body under consideration, while $Y = [0,1]^d$ is the unit cell.
The boundary of $\Omega$ consists of two disjoint parts, the Dirichlet
boundary $\Gamma_D$ and the Neumann boundary $\Gamma_N$.  Furthermore,
$u \colon  (0,T) \times \Omega \to \R^d$ is the displacement on the macro
level, while $v \colon  (0,T) \times \Omega \times Y \to \R^d$ is the
displacement reflecting the micro structure.  The stress tensor is
denoted by $\Sigma \colon  (0,T) \times \Omega \times Y \to \Rs$ and
$z\colon  (0,T) \times \Omega \times Y \to \mathbb{V}$ is the internal
variable describing changes in the material behavior under plastic
deformation (such as hardening), where $\mathbb{V}$ is a finite
dimensional Banach space.  Moreover,
$\nabla_x^s \defn  \frac{1}{2}(\nabla_x + \nabla_x^\top)$ is the
linearized strain in $\Omega$ and $\nabla^s_y$ is defined
analogously. The elasticity tensor
$\mathbb{C} \colon  \Omega \times Y \to \LL(\Rs)$ and the hardening
parameter $\mathbb{B} \colon  \Omega \times Y \to \LL(\mathbb{V})$ are given
linear and coercive mappings and, by
$B \colon  \Omega \times Y \rightarrow \LL(\mathbb{V};\Rs)$, one recovers
the plastic strain from the internal variables $z$. The evolution of
the internal variables is determined by a maximal monotone operator
$A\colon  \mathbb{V} \to 2^{\mathbb{V}}$.  In \cref{subsec:vonmises} below,
we present a concrete example for such an operator, namely the case of
linear kinematic hardening with von Mises yield condition. Finally,
$z_0$ is a given initial state and $\pi$ is the averaging over the
unit cell, i.e.,
\begin{equation}\label{eq:defpi}
    \pi \colon  \Sigma \mapsto \fint_Y \Sigma (\cdot,y)\,dy \defn  \frac{1}{|Y|} \int_Y \Sigma (\cdot,y)\,dy.
\end{equation}
The precise assumptions on these data as well as the precise notion of
solutions to~\eqref{eq:homogenizationEquations} are given below.

The volume force $f\colon (0,T) \times \Omega \to \R^d$ and the
boundary loads $g\colon (0,T) \times \Gamma_N \to \R^d$, serve as
control variables. In the following, we will frequently write $\ell$
for the tuple $(f,g)$.  Possible objectives could include a desired
displacement or stress distribution at end time, i.e.,
\begin{equation*}
    \begin{aligned}
        J(u, \Sigma, \ell) \defn  \frac{\alpha}{2}\int_\Omega |u(T) - u_d|^2
        \, d x + \frac{\beta}{2} \int_\Omega |(\pi\Sigma)(T) - \sigma_d|^2
        \, d x + \Phi(\ell),
    \end{aligned}
\end{equation*}
where $u_d\colon  \Omega \to \R^d$ and $\sigma_d\colon  \Omega \to \Rs$ are given
desired displacement and stress field, respectively,
$\alpha, \beta \geq 0 $, and $\Phi$ is a regularization term depending
on the choice of the control space that will be specified below, see
\cref{rem:objex}.

\subsection{Homogenized Plasticity -- Notation and Standing
Assumptions}

Before discussing the optimal control problem, we first have to
introduce the precise notion of solution for homogenized
elastoplasticity system in~\eqref{eq:homogenizationEquations}. For
this purpose, we need several assumptions and definitions. We start
with the following

\begin{assumption}[Hypotheses on the data in~\eqref{eq:homogenizationEquations}]\label{assu:data}\
    \begin{itemize}
        \item \emph{Regularity of the domain:} The domain
            $\Omega\subset\R^d$, $d\in\{2,3\}$, is bounded with Lipschitz
            boundary $\Gamma$. The boundary consists of two disjoint
            measurable parts $\Gamma_N$ and $\Gamma_D$ such that
            $\Gamma=\Gamma_N \cup \Gamma_D$. While $\Gamma_N$ is a relatively
            open subset, $\Gamma_D$ is a relatively closed subset of $\Gamma$
            with positive measure.  In addition, the set
            $\Omega \cup \Gamma_N$ is regular in the sense of Gr\"oger,
            cf.~\cite{Gro89}.
        \item \emph{Assumptions on the coefficients:} The elasticity tensor
            and the hardening parameter satisfy
            $\mathbb{C} \in L^\infty(\Omega \times Y; \LL(\Rs))$ and
            $\Bb\in L^\infty(\Omega \times Y; \LL(\mathbb{V}))$ and are
            symmetric and uniformly coercive, i.e., there exist constants
            $\underline{c}>0$ and $\underline{b} > 0$ such that
            \begin{equation*}
                \begin{aligned}
                    \mathbb{C} (x,y) \sigma \colon  \sigma & \geq \underline{c}\,
                    \|\sigma\|_{\R^{d\times d}}^2
                    & & \forall\, \sigma \in \Rs, & & \text{f.a.a.~} (x,y) \in \Omega \times Y,  \\
                    \Bb(x,y) \zeta \colon  \zeta & \geq \underline{b}\,
                    \|\zeta\|_{\mathbb{V}}^2 & & \forall\, \zeta \in \mathbb{V}, &
                    & \text{f.a.a.~} (x,y) \in \Omega \times Y .
                \end{aligned}
            \end{equation*}
            In addition,
            $B \in L^\infty(\Omega \times Y; \LL(\mathbb{V}; \Rs))$ is a given
            linear mapping.
    \end{itemize}
\end{assumption}

Next, we define the function spaces for the various variables in~\eqref{eq:homogenizationEquations}:

\begin{definition}[Function spaces]
    Let $s \in [1,\infty)$. For the quantities in~\eqref{eq:homogenizationEquations}, we define the following spaces:
    \begin{itemize}
        \item space for the macro displacement $u$:
            \begin{equation*}
                U^s \defn  W^{1,s}_D(\Omega;\R^d) \defn  
                \overline{\bigl\{\psi|_\Omega \colon  \psi \in C^\infty_0(\R^d;\R^d),\;
                        \operatorname{supp}(\psi)\cap \Gamma_D =
                \emptyset\bigr\}}^{W^{1,s}(\Omega;\R^d)} 
            \end{equation*}
        \item space for the internal variable $z$:
            \begin{equation*}
                Z^s \defn  L^s\bigl(\Omega \times Y;\mathbb{V}\bigr)
            \end{equation*}
        \item stress space for $\Sigma$:
            \begin{equation*}
                S^s \defn  L^s\bigl(\Omega \times Y; \Rs\bigr).
            \end{equation*}
        \item space for the micro displacement $v$:
            \begin{equation*}
                V^s \defn  L^{s}\bigl(\Omega;
                W^{1,s}_{\per,\perp}(Y;\R^d)\bigr). 
            \end{equation*}
    \end{itemize}
    For the latter, we denote by $C^\infty_{\per}(Y;\R^d)$ the space of
    $C^\infty(\R^d;\R^d)$ functions which are $Y$-periodic, identified
    with their restriction on $Y$, and define $W^{1,s}_\per(Y;\R^d)$ to be
    the closure of $C^\infty_\per(Y;\R^d)$ with respect to the
    $W^{1,s}(Y;\R^d)$ norm.  Further, $W^{1,s}_\perp(Y;\R^d)$ is the
    closed subspace of $W^{1,s}(Y;\R^d)$ consisting of functions of
    mean~$0$, and
    \begin{equation*}
        W^{1,s}_{\per,\perp}(Y;\R^d) = W^{1,s}_\per(Y;\R^d) \cap
        W^{1,s}_\perp(Y;\R^d).
    \end{equation*}
    We set the norm on $V^s$ to be
    \begin{equation*}
        \norm{v}{V^s} \defn  \norm{v}{L^{s}(\Omega \times Y;\R^d)} +
        \norm{\nabla_y^s v}{L^{s}(\Omega \times Y;\R^{d\times d})},
    \end{equation*}
    with which $V^s$ becomes a Banach space and for the case $s=2$ a
    Hilbert space with the obvious scalar product.
\end{definition}

\begin{assumption}[Maximal monotone operator]\label{assu:maxmon}
    The maximal monotone operator $A$ from the evolution law in~\eqref{eq:evolstrong} is a set-valued map in the Hilbert space
    $\HH = Z^2$, i.e., $A\colon  Z^2 \to 2^{Z^2}$.  It is assumed to satisfy
    our standing assumptions from \cref{sec:2}.  Moreover, we assume
    that there is a sequence of operators $\{A_n\}$ from $Z^2$ to $Z^2$
    satisfying \cref{ass:AnAssumption}.
\end{assumption}

In \cref{subsec:vonmises} below, we will investigate the maximal
monotone operator arising in the case of linear kinematic hardening
with von Mises yield condition and show how to construct the
approximating sequence of smooth operators for this particular case.
With the above definitions at hand, we are now in the position to
define our precise notion of solutions to~\eqref{eq:homogenizationEquations}:

\begin{definition}[Weak solutions]
    \label{def:definitionOfASolutionOfhomogenizationEquations}
    Let $\ell \in \WZ{(U^{2})^*}$ and $z_0 \in Z^2$. Then we say that a
    tuple
    \begin{equation*}
        (u,v,z,\Sigma ) \in \WZ{U^s} \times \WZ{V^2} \times \WZ{Z^2} \times \WZ{S^2}
    \end{equation*}
    is a solution of \cref{eq:homogenizationEquations}, if, for almost
    all $t\in (0,T)$, there holds
    \begin{subequations}\label{eq:homweak}
        \begin{align}
            \int_\Omega (\pi \Sigma (t))(x) \cdot \nabla_x^s \varphi(x)\,dx &= \dualpair{\ell(t)}{\varphi}{}
            &&\forall\, \varphi \in U^{2}, \label{eq:momOm}\\
            \Sigma(t)  &= \mathbb{C} (\pi_r^{-1}\nabla_x^s u(t) + \nabla_y^sv(t) - Bz(t))
            &&\text{in } S^2, \label{eq:constlaw}\\
            \int_{\Omega \times Y} \Sigma (t,x,y) \cdot \nabla_y^s \psi(x,y)\,dy &= 0
            &&\forall\, \psi \in V^{2}, \label{eq:momY}\\
            \boverdot{z}(t) &\in A(B^\top \Sigma (t) - \mathbb{B}z(t)) \quad &&\text{in } Z^2, \label{eq:plastevol}\\
            z(0) &= z_0  &&\text{in } Z^2, \label{eq:plastini}
        \end{align}   
    \end{subequations}
    where $\pi \colon  S^2 \to L^2(\Omega;\Rs)$ is the average mapping from~\eqref{eq:defpi} and
    \begin{equation*}
        \pi^{-1}_r \colon  L^2(\Omega;\Rs) \ni \varepsilon \mapsto 
        \Big( \Omega \times Y \ni (x,y) \mapsto \varepsilon(x) \in \Rs \Big) \in S^2.
    \end{equation*}
    In the following, we will frequently consider $\pi_r^{-1}$ in
    different domains and ranges, for simplicity denoted by the same
    symbol.
\end{definition}

\subsection{Reduction of the System}

In the following, we reduce the system~\eqref{eq:homweak} to an
equation in the internal variable $z$ only and it will turn out that
this equation has exactly the form of our general equation~\eqref{eq:aux}.
To this end we proceed analog to~\cite[Chapter~4]{groger}.
For this purpose, let us define the following
operators:

\begin{definition}
    Let $s\in [1, \infty)$. Then we define
    \begin{equation*}
        \symnabla_{(x,y)}\colon  U^s \times V^s \to S^s, \quad 
        \symnabla_{(x,y)}(u,v) \defn  \pi_r^{-1}\symnabla_x u + \symnabla_y v.
    \end{equation*}
    For its adjoint, we write
    \begin{equation*}
        \begin{aligned}
            & \Div_{(x,y)}\colon  S^{s'} \to  (U^{s})^* \times (V^{s})^*, \\
            & \dual{\Div_{(x,y)}\sigma}{(\varphi, \psi)}
            \begin{aligned}[t]
                & \defn  - \dual{{\symnabla_{(x,y)}}^*\sigma}{(\varphi, \psi)}
                = -\int_{\Omega\times Y} \sigma(x,y) \colon (\symnabla_x
                \varphi(x) + \symnabla_y \psi(x,y)) \,d(x,y).
            \end{aligned}
        \end{aligned}
    \end{equation*}
    With a slight abuse of notation, we denote these operators for
    different values of $s$ always by the same symbol.
\end{definition}

\begin{lemma}
    \label{lem:w1sExistenceHomogenization}
    Let \cref{assu:data} be fulfilled.  Then there is an index
    $\bar s > 2$ such that, for every $s\in [\bar s', \bar s]$ and every
    $(\mathfrak{f}, \mathfrak{g}) \in (U^{s'})^* \times (V^{s'})^*$,
    there exists a unique solution $(u,v) \in U^s \times V^s$ of
    \begin{equation}\label{eq:linsyshom}
        -\Div_{(x,y)} \big( \mathbb{C} \symnabla_{(x,y)} (u,v)\big) = (\mathfrak{f}, \mathfrak{g}) \quad \text{in } (U^{s'})^* \times (V^{s'})^*
    \end{equation}
    and there is a constant $C_s>0$, independent of $\mathfrak{f}$ and
    $\mathfrak{g}$, such that
    \begin{equation*}
        \|(u, v)\|_{U^s \times V^s} \leq C_s \big( \|\mathfrak{f}\|_{(U^{s'})^*} + \|\mathfrak{g}\|_{(V^{s'})^*} \big).
    \end{equation*}
\end{lemma}

\begin{proof} The claim is equivalent to
    $-\Div_{(x,y)} \mathbb{C} \symnabla_{(x,y)}$ being a topological
    isomorphism between $U^s\times V^s$ and its dual space
    $(U^s)^* \times (V^s)^*$ for $s \in [\bar s',\bar s]$.  We start
    with the case $s=2$. For this, the left hand side of~\eqref{eq:linsyshom}
    gives rise to a bilinear form $\mathfrak{b}$ on the Hilbert space
    $U^2 \times V^2$:
    \begin{equation*}
        \mathfrak{b}\bigl((u,v),(\varphi,\psi)\bigr) \defn 
        \big(\mathbb{C} \symnabla_{(x,y)} (u, v), \symnabla_{(x,y)}
        (\varphi, \psi)\big)_{S^2} 
    \end{equation*}
    Clearly, $\mathfrak{b}$ is bounded. Due to Poincar\'e's inequality
    for functions with zero mean value, which implies
    \begin{align*}
        \mathfrak{b}\bigl((u,v),(u,v)\bigr) &\geq \underline{c}\,
        |Y| \int_\Omega |\symnabla_x u(x)|^2\,dx 
        + \underline{c} \int_{\Omega \times Y} | \symnabla_y v(x,y) |^2 \,d(x,y) \\
        &\geq C \big(\norm{u}{H^1_D(\Omega;\R^d)} +
        \norm{v}{V^2(\Omega \times Y;\R^d)} \big)^2 
        \quad \forall\, (u,v) \in U^2 \times V^2,
    \end{align*}
    it is also coercive so that the claim and isomorphism property for
    $-\Div_{(x,y)} \mathbb{C} \symnabla_{(x,y)}$ for $s=2$ follows from the
    Lax-Milgram lemma. 

    We next extrapolate this isomorphism property to $U^s \times V^s$
    for $s$ around $2$ using the fundamental stability theorem by
    {\v{S}}ne{\u{i}}berg~\cite{S74}. (See also the more accessible and
    extensive~\cite[Appendix~A]{ABES19}.) More precisely, we show that
    the spaces $U^s \times V^s$ and their duals form complex
    interpolation scales in $s$. Then the stability theorem shows that
    the set of scale parameters $s$ such that
    $-\Div_{(x,y)} \mathbb{C} \symnabla_{(x,y)}$ is a topological isomorphism
    between $U^s \times V^s$ and its dual space is open. Since the set includes
    $2$, as seen above, this then implies the claim.

    To establish the interpolation scale, it is enough to consider the
    primal case, since the dual interpolation scale is inherited from
    the primal one by duality properties of the complex interpolation
    functor~\cite[Theorem~1.11.3]{Triebel:1978}. So, we show that
    \begin{equation*}
        U^{s_\theta} \times V^{s_\theta} = \bigl[U^{s_0}\times V^{s_0},U^{s_1}\times
        V^{s_1}\bigr]_{\theta} \quad \text{for} \quad \frac1s =
        \frac{1-\theta}{s_0} + \frac\theta{s_1}
    \end{equation*}
    for all $s_0,s_1 \in (1,\infty)$ and $\theta \in (0,1)$. It is
    moreover sufficient to consider each component in the interpolation
    separately.

    For the $U^s = W^{1,s}_D(\Omega;\R^d)$ spaces, the interpolation
    scale property is well known by now in the setting of
    \cref{assu:data} and even much more general ones; we refer
    to~\cite{BecEge19}. The result for $V^s$ is proven by
    reducing the problem to the $W^{1,s}_{\per,\perp}(Y;\R^d)$ spaces
    and showing that these are complemented subspaces of
    $W^{1,s}(Y;\R^d)$ and thus inherit the latter's interpolation
    properties. This is done in the appendix,
    \cref{thm:interpolation-vs}, and finishes the proof.
\end{proof}

\begin{remark}\label{rem:sneiberg}
    In general, one cannot expect $\bar s$ to be
    significantly larger than 2, due to both the irregular coefficient
    tensors and the mixed boundary conditions, see
    e.g.~\cite{ERS07,Sha68,Mey63}.  This issue will become crucial in
    the discussion of second-order necessary optimality conditions in~\ref{subsec:opsys} below.
\end{remark}

Now we are in the position to reduce~\eqref{eq:homweak} to an equation
in the variable $z$ only.  For this purpose, we need the following

\begin{definition}[$Q$ and $R$ for the case of homogenized plasticity]\label{def:QR}
    Let $s\in [\bar s',\bar s]$ be given.  By
    \cref{lem:w1sExistenceHomogenization}, the solution operator
    associated with~\eqref{eq:linsyshom}, denoted by
    \begin{equation*}
        \GG \defn  \big(-\Div_{(x,y)} \mathbb{C} \symnabla_{(x,y)}\big)^{-1} \colon  (U^{s'})^* \times (V^{s'})^* \to U^s \times V^s,
    \end{equation*}
    is well defined, linear and bounded.  The components of $\GG$ are
    abbreviated by
    \begin{equation*}
        \GG_u \defn  (1,0) \,\GG\colon   (U^{s'})^* \times (V^{s'})^* \to U^s, \quad
        \GG_v \defn  (0,1) \, \GG\colon   (U^{s'})^* \times (V^{s'})^* \to V^s.
    \end{equation*}
    Based on this solution operator, we moreover define
    \begin{equation*}
        T\colon  Z^s \ni z \mapsto B^\top \mathbb{C} \symnabla_{(x,y)} \GG(-\Div_{(x,y)}(\mathbb{C} B z)) \in Z^s.
    \end{equation*}
    Now, we have everything at hand to define the mappings $R$ and $Q$
    from our general equation~\eqref{eq:aux} for the special case of
    homogenized plasticity:
    \begin{align}
        & R\colon  (U^{s'})^*\ni \ell \mapsto B^\top \mathbb{C} \symnabla_{(x,y)} \GG (\ell,0)\in Z^s, \label{eq:Rhom}\\
        & Q \colon  Z^s \ni z \mapsto (B^\top \mathbb{C} B + \mathbb{B} - T )z \in Z^s.   \label{eq:Qhom}
    \end{align}
    Again, with a slight abuse of notation, we denote all of the above
    operators for different values of $s\in [\bar s', \bar s]$ always by
    the same symbol.
\end{definition}

The reason for defining the operators $Q$ and $R$ in the way we did in
\cref{def:QR} is the following: Owing to
\cref{lem:w1sExistenceHomogenization}, given $z \in Z^2$, one can
solve~\eqref{eq:momOm}--\eqref{eq:momY} for $u$, $v$, and $\Sigma$ so
that the tuple $(u, v, \Sigma)\in U^2\times V^2\times S^2$ is uniquely
determined by $z$.  Even more, using the operators from \cref{def:QR},
we see that the solution of~\eqref{eq:momOm}--\eqref{eq:momY} for
given $z$ is
\begin{align}
    (u, v) &= \GG(-\Div_{(x,y)}(\mathbb{C} B z)) + (\ell, 0)), \label{eq:uv}\\
    \Sigma 
    &= \mathbb{C} \big[\symnabla_{(x,y)}\GG\big(-\Div_{(x,y)}(\mathbb{C} B z)) + (\ell, 0)\big) - Bz\big]. \label{eq:Sigma}
\end{align}
Inserting the last equation in~\eqref{eq:plastevol} and employing the
definition of $Q$ and $R$ in~\eqref{eq:Qhom} and~\eqref{eq:Rhom} then
yields
\begin{equation*}
    \boverdot{z} \in A(B^\top \Sigma - \mathbb{B}z) = A(R\ell - Q z), 
\end{equation*}
i.e., exactly an evolution equation of the general form in~\eqref{eq:aux}.  This shows that the system~\eqref{eq:homweak} of
homogenized elastoplasticity can equivalently be rewritten as an
abstract operator evolution equation of the form~\eqref{eq:aux}.

For the differentiability properties needed in sections~\ref{sec:5}
and~\ref{sec:6}, a norm gap is required such that it is no longer
sufficient to consider just the Hilbert space $\HH = Z^2$. In
accordance with the definitions of $R$ and $Q$, we therefore define
the spaces in the abstract setting in our concrete application problem
as follows:

\begin{definition}[Spaces in case of homogenized plasticity]\label{def:spaces}
    The spaces $\YY$, $\ZZ$, $\HH$, and $\WW$ from \cref{sec:2} are set
    to
    \begin{equation}\label{eq:spacesvonmises}
        \YY \defn  Z^{s_1}
        \embed \ZZ \defn  Z^{s_2}
        \embed \HH = Z^{2}
        \embed \WW \defn  Z^{s_3}
    \end{equation}
    with $s_1 \geq s_2\geq 2 \geq s_3$. The integrability indices $s_1$,
    $s_2$, and $s_3$ depend crucially on the differentiability
    properties of the regularized version of $A$ and will be specified
    for a concrete realization of $A$ in \cref{subsec:vonmises}
    below. Moreover, we choose
    \begin{equation*}
        \XX \defn  (U^{s_1'})^* = W^{1,s_1'}_D(\Omega;\R^d)^* \nfed  W^{-1,s_1}_D(\Omega;\R^d).
    \end{equation*}
    Furthermore, the control space is given by
    \begin{equation}\label{eq:sobolev}
        \XX_c \defn  L^p(\Omega;\R^d) \times L^r(\Gamma_N;\R^d)
        \quad \text{with} \quad p > \tfrac{d s_1}{d + s_1}
        \quad \text{and} \quad r > \tfrac{(d-1) s_1}{d}.
    \end{equation}
    Due to $s_1 \geq 2$, $\XX_c$ is reflexive and embeds compactly in
    $\XX$ by Sobolev embedding and trace theorems.
    Therefore, all our standing assumptions on the spaces in
    \cref{sec:2} are fulfilled.

    Of course, elements in $\XX_c$ are identified with those in $\XX$ by
    \begin{equation*}
        \dual{(f,g)}{u}_{(U^{s_1})^*, U^{s_1}} \defn  \int_\Omega f\,u\,dx + \int_{\Gamma_N} g\,u \, ds,\quad 
        (f,g) \in \XX_c, \, u \in U^{s_1}.
    \end{equation*}
\end{definition}

In order to apply our general theory to the present setting, we need
the following assumption on the regularity of the linear equation~\eqref{eq:linsyshom}. As we will see in subsections~\ref{subsec:opsys}
and~\ref{subsec:vonmises} below, this assumption may become fairly
restrictive, if one aims to establish second-order sufficient
optimality conditions, since, in this case, $s_1$ and the conjugate
index $s_3'$ may be rather large.

\begin{assumption}[Critical regularity condition]\label{assu:sneiberg}
    The index $\bar s$ from \cref{lem:w1sExistenceHomogenization}
    satisfies $\bar s \geq \max\{s_1, s_3'\}$, where $s_1$ and $s_3$ are
    the integrability indices from~\eqref{eq:spacesvonmises}.
\end{assumption}

\begin{proposition}
    Under \cref{assu:sneiberg} and with the spaces defined in
    \cref{def:spaces}, the operators $R$ and $Q$ from~\eqref{eq:Rhom}
    and~\eqref{eq:Qhom}, respectively, satisfy the standing assumptions
    from \cref{sec:2}, that is, $R$ is linear and bounded from
    $(U^{s_1'})^*$ to $Z^{s_1}$ and $Q$ is a linear and bounded operator
    from $Z^s$ to $Z^s$ for all $s\in [s_3, s_1]$ and, considered as an
    operator in $Z^2$, coercive and self-adjoint.
\end{proposition}

\begin{proof}
    The required mapping properties of $Q$ and $R$ directly follow from
    their construction in \cref{def:QR} in combination with
    \cref{lem:w1sExistenceHomogenization} and \cref{assu:sneiberg},
    respectively.  It remains to show that $Q$ is coercive and
    self-adjoint. Since $\mathbb{B}$ is symmetric and coercive according
    to \cref{assu:data}, it is sufficient to prove that the operator
    $B^\top \mathbb{C} B - T \colon  Z^2 \rightarrow Z^2$ is symmetric and
    positive.  To prove the symmetry, first observe that
    $B^\top \mathbb{C} B$ is symmetric by the symmetry of $\mathbb{C} $.
    The symmetry of $\mathbb{C} $ moreover implies that
    $\GG\colon  (U^2)^* \times (V^2)^* \to U^2 \times V^2$, i.e., the solution
    operator of~\eqref{eq:linsyshom}, is self-adjoint. Therefore, the
    construction of $T$ in \cref{def:QR} implies for all
    $z_1, z_2 \in Z^2$ that
    \begin{equation*}
        \begin{aligned}
            (T z_1, z_2)_{Z^2}
            &= \dual{-\Div_{(x,y)} (\mathbb{C} B z_2)}{\GG(-\Div_{(x,y)}(\mathbb{C} B z_1))}\\
            &= \dual{\GG(-\Div_{(x,y)} (\mathbb{C} B z_2))}{-\Div_{(x,y)}(\mathbb{C} B
            z_1)} = (z_1, T z_2)_{Z^2}
        \end{aligned}
    \end{equation*}
    so that $T$ is also symmetric.  To show the positivity of
    $B^\top \mathbb{C} B - T$, let $z\in Z^2$ be arbitrary.  To shorten the
    notation, we abbreviate $(u_z, v_z) \defn  \GG(-\Div_{(x,y)}(\mathbb{C} B
    z))$. Then, by testing the equation for $(u_z, v_z)$, i.e.,~\eqref{eq:linsyshom} with
    $(\mathfrak{f}, \mathfrak{g}) = -\Div_{(x,y)}(\mathbb{C} B z)$, with
    $(-u_z, -v_z)$, we arrive at
    \begin{equation*}
        \big(\mathbb{C} (Bz - \symnabla_{(x,y)} (u_z, v_z), -\symnabla_{(x,y)} (u_z, v_z)\big)_{S^2} = 0.
    \end{equation*}         
    Since, by construction,
    $T z = B^\top \mathbb{C} \symnabla_{(x,y)} (u_z, v_z)$, the coercivity of
    $\mathbb{C} $ therefore implies
    \begin{equation*}
        \begin{aligned}
            \scalarproduct{(B^\top \mathbb{C} B - T)z}{z}{Z^2}
            &= \big(\mathbb{C} (B z - \symnabla_{(x,y)} (u_z, v_z), Bz\big)_{S^2} \\
            &= \big(\mathbb{C} (B z - \symnabla_{(x,y)} (u_z, v_z), Bz -
            \symnabla_{(x,y)} (u_z, v_z)\big)_{S^2}\geq 0.
        \end{aligned}
    \end{equation*}
    As $z$ was arbitrary, this proves the positivity.
\end{proof}

We point out that the whole analysis in sections~\ref{sec:3} and~\ref{sec:exopt} is carried out in the Hilbert space $\HH = Z^2$.
Therefore, for the mere existence and approximation results from these
two sections, the critical regularity condition in
\cref{assu:sneiberg} is \emph{automatically fulfilled} by setting
$s_1 = s_2 = s_3 = 2$ (so that $\YY = \ZZ = \WW = \HH = Z^2$). Note
that, in this case, the Lax-Milgram lemma guarantees the assertion of
\cref{assu:sneiberg} without any further regularity assumptions, see
the proof of \cref{lem:w1sExistenceHomogenization}.  The additional
crucial regularity assumption only comes into play, when first- and
second-order optimality conditions are investigated, see
\cref{rem:lessassus}. In \cref{subsec:vonmises} below, we will
elaborate in detail, where the critical \cref{assu:sneiberg} is needed
to ensure the required differentiability properties of the regularized
control-to-state map for the example of a specific yield condition.

We collect our findings so far in the following

\begin{theorem}[Homogenized plasticity as abstract evolution VI]\label{thm:approxhom}
    Under the Assumptions~\ref{assu:data} and~\ref{assu:maxmon}, the
    system of homogenized elastoplasticty in its weak form in~\eqref{eq:homweak} is equivalent to an abstract operator
    differential equation of the form
    \begin{equation}\label{eq:odehom}
        \boverdot{z} \in A(R\ell - Qz), \mediumspace z(0) = z_0,
    \end{equation}
    with $Q$ and $R$ as defined in~\eqref{eq:Rhom} and~\eqref{eq:Qhom}
    in the following sense: If $(u, v, z, \Sigma)$ solves~\eqref{eq:homweak}, then $z$ is a solution~\eqref{eq:odehom}, and
    vice versa, if $z$ solves~\eqref{eq:odehom}, then $z$ together with
    $(u,v)$ and $\Sigma$ as defined in~\eqref{eq:uv} and~\eqref{eq:Sigma}, respectively, form a solution of~\eqref{eq:homweak}.

    In addition, $Q$ and $R$ satisfy the standing assumptions from
    \cref{sec:2} (provided that the function spaces are chosen according
    to \cref{def:spaces}). Therefore, the existence and approximation
    results of Sections~\ref{sec:3} and~\ref{sec:exopt} hold for~\eqref{eq:odehom}, in particular:
    \begin{itemize}
        \item For every $\ell \in \UU(z_0, D(A))$, there is a unique
            solution $(u,v,z, \Sigma)$ of the weak system of homogenized
            plasticity in~\eqref{eq:homweak}, cf.~\cref{thm:auxExistence}.
        \item Optimal control problems governed by the weak system of
            homogenized elastoplasticity admit globally optimal solutions,
            provided that the standing assumptions on the objective are
            fulfilled, cf.~\cref{thm:existenceOfAGlobalSolution}.
        \item The approximation results of
            \cref{thm:regularizedOptimaConvergence} and
            \cref{cor:strongapprox} apply in case of homogenized
            elastoplasticity.
    \end{itemize}
\end{theorem}

\begin{remark}
    Existence and uniqueness of solutions to~\eqref{eq:homweak} was
    already established in~\cite{schweizerHomogenization}.
\end{remark}

\begin{remark}\label{rem:objex}
    The example for objective functionals mentioned above, i.e.,
    \begin{equation}\label{eq:objex}
        \begin{aligned}
            J(u, \Sigma,f,g) \defn
            \tfrac{\alpha}{2}\int_\Omega |u(T) - u_d|^2 \,dx +
            \tfrac{\beta}{2}\int_\Omega |(\pi\Sigma)(T) - \sigma_d|^2 \,dx +
            \tfrac{\gamma}{2} \|(f,g)\|_{H^1(0,T;\XX_c)}^2
        \end{aligned}
    \end{equation}
    with $u_D \in L^2(\Omega;\R^d)$ and $\sigma_D \in L^2(\Omega;\Rs)$
    and $\alpha, \beta \geq 0$, $\gamma> 0$, satisfies the standing
    assumptions on the objective functional, as we will see in the
    following.  In this case, the functional $\Psi\colon  (z,\ell) \to \R$ in
    the general setting consists of the two integrals at end point $T$.
    Let us consider the first one containing the displacement $u$.
    Since the latter is given by the first component of the solution
    operator of~\eqref{eq:uv}, which maps
    $H^1(0,T;Z^2) \times H^1(0,T;\XX_c)$ to
    $H^1(0,T;U^{2}) \embed C([0,T];L^2(\Omega;\R^d))$, this integral is
    well defined. (Note that the operators in~\eqref{eq:uv} just act
    pointwise in time and the time regularity of $z$ and $\ell$ carries
    over to $u$ and $v$.)  Clearly, this solution operator is linear and
    bounded.  In case of the second integral involving $\Sigma$, one can
    argue completely analogously based on the solution operator of~\eqref{eq:Sigma}. Thus, $\Psi$ is convex and continuous, hence
    weakly lower semicontinuous, and in addition, bounded from below by
    zero.  Moreover, the purely control part of the objective is given
    by $\Phi(f,g) = \frac{\gamma}{2} \|(f,g)\|_{H^1(0,T;\XX_c)}^2$ and
    therefore clearly weakly lower semicontinuous and coercive as
    required.  Thus all standing assumptions are fulfilled as
    claimed. Of course, various other objective functionals are possible
    as well, such as tracking type objectives over the whole
    space-time-cylinder, but to keep the discussion concise, we just
    mention the example above.
\end{remark}

\subsection{Optimality System}\label{subsec:opsys}

In the following section, we establish necessary and sufficient
optimality conditions for the optimal control of regularized
homogenized elastoplasticity. To be more precise, we consider a single
element of the sequence of regularizations of the maximal monotone
operator $A$ from \cref{assu:maxmon}, which we again denote by $A_s$,
and apply the general theory from \cref{sec:5} and \cref{sec:6}. In
view of the norm gap needed for the differentiability of $A_s$, we
will consider $A_s$ in different domains and ranges (denoted by the
same symbol) and assume that $A_s$ maps $\YY = Z^{s_1}$ to
itself. Accordingly, we treat the regularized version of the state
equation (with $A_s$ instead of $A$) in the same manner, i.e., with
integrability index $s_1$ instead of $2$, see~\eqref{eq:ocps} below.
To keep the discussion concise, we moreover assume in all what follows
that $s_1 \geq 2$ is such that $p = r =2$ satisfy the conditions in~\eqref{eq:sobolev}. For $d= \dim(\Omega)=3$, this implies $s_1 < 3$
and, in case without boundary control, i.e., $g\equiv 0$, $s_1 < 6$ is
sufficient. This will become important in the discussion of
second-order sufficient conditions, as we will see below. Motivated by~\eqref{eq:objex}, we consider an optimal control problem of the form
\begin{equation}\label{eq:ocps}
    \left\{ \quad   
        \begin{aligned}
            \min \quad &
            \begin{aligned}[t]
                J(u, \Sigma, f, g) & \defn  F_1(u, \Sigma) + F_2(u(T), \Sigma(T)) \\
                & \qquad+
                \frac{\gamma}{2}\Big(\|\boverdot{f}\|_{\LZ{L^2(\Omega;\R^d)}}^2
                + \|\boverdot{g}\|_{\LZ{L^2(\Gamma_N;\R^d)}}^2\Big)
            \end{aligned}\\
            \text{s.t.} \quad &
            (f, g) \in H^1(0,T;L^2(\Omega;\R^d)\times L^2(\Gamma_N;\R^d)), \\
            & (u, v, z, \Sigma) \in H^1(0,T;U^{s_1} \times V^{s_1} \times Z^{s_1} \times S^{s_1}),\\
            \text{and} \quad &
            \begin{aligned}[t]
                -\Div_{(x,y)} \Sigma &= ((f, g),0),\\
                \Sigma &= \mathbb{C} \big( \symnabla_{(x,y)} (u, v) - B z\big) ,\\
                \boverdot{z} &= A_s(B^\top \Sigma - \Bb z), \quad z(0) = z_0,
            \end{aligned} \\
            & \ell(0) = 0.
        \end{aligned}
    \right.
\end{equation}
Since, in many applications, displacement and stress on the macro
level are of special interest, especially at end time, we focus on
objectives with this particular structure with continuously
Fr\'echet differentiable mappings
\begin{equation}\label{eq:objuSigma}
    F_1 \colon  L^2(0,T;U^2) \times L^2(0,T;S^2) \to \R, \quad F_2\colon  U^2 \times S^2 \to \R.
\end{equation}
In order to apply our general theory, we not only have to reduce the
state system to an equation of the form~\eqref{eq:odehom}, but also
have to reduce the objective.  For this purpose, let us denote the
solution operators of~\eqref{eq:uv} and~\eqref{eq:Sigma} by
$\mfu \colon  (\ell, z) \mapsto u$ and $\mfS\colon  (\ell, z) \mapsto \Sigma$.  To
shorten the notation, we will consider $\mfu$ and $\mfS$ with
different domains and ranges, e.g.\
$\mfu\colon  (U^{s'})^* \times Z^s \to U^s$ and
$\mfu\colon  L^2(0,T;(U^{s'})^*) \times L^2(0,T;Z^s) \to L^2(0,T;U^s)$ with
$s\in [\bar s', \bar s]$ and analogously for $\mfS$. Note again that
the time regularity of $z$ and $\ell$ directly carries over to the
time regularity of $u$ and $\Sigma$.  Given these operators, we define
\begin{equation*}
    \begin{aligned}
        &\Psi_1\colon  L^2(0,T;Z^2) \times L^2(0,T;(U^2)^*) \to \R, & & \Psi_1(z, \ell) \defn  F_1(\mfu(z, \ell), \mfS(z, \ell)),\\
        &\Psi_2\colon  Z^2 \times (U^2)^* \to \R, & & \Psi_2(z, \ell) \defn 
        F_2(\mfu(z, \ell), \mfS(z, \ell)),
    \end{aligned}
\end{equation*}
so that the objective in~\eqref{eq:ocps} becomes
\begin{equation}
    J(z, \ell) = \Psi_1(z,\ell) + \Psi_2(z(T), \ell(T)) + \tfrac{\gamma}{2} \|(\boverdot f, \boverdot g)\|_{L^2(0,T;\XX_c)}^2, 
\end{equation}
i.e., exactly an objective of the form in~\eqref{eq:regpsi} and~\eqref{eq:regpsi2}, respectively.  Since $\mfu$ and $\mfS$ are linear
and bounded and $F_1$ and $F_2$ are assumed to be continuously
Fr\'echet differentiable, the chain rule implies the differentiability
of $\Psi_1$ and $\Psi_2$ so that \cref{ass:AsAndJFrechet}(i) is met,
if we set $s_3 = 2$ so that $\WW = Z^2$.  To apply the results of
\cref{sec:5} in order to establish an optimality system for~\eqref{eq:ocps}, we additionally need that $A_s$ satisfies
\cref{ass:AsAndJFrechet}(ii), which is ensured by the following
\begin{assumption}\label{assu:Asmooth1}
    We set $s_2 = s_3 = 2$, i.e., $\WW = \ZZ = Z^2$, and assume that
    $A_s$ fulfills \cref{ass:AsAndJFrechet}(ii) with $\YY = Z^{s_1}$,
    $s_1 \geq 2$, i.e., in particular that $A_s$ is
    Fr\'echet differentiable from $\ZZ^{s_1}$ to $\ZZ^2$.
\end{assumption}

In light of \cref{lem:w1sExistenceHomogenization},
\cref{assu:Asmooth1} does not impose any restriction for practical
realizations of $A_s$, as we will see in \cref{subsec:vonmises} below.
Given this assumption, \cref{thm:firstOrderOptimalityCondition} and
\cref{ex:ode} imply for a locally optimal solution
$(\overline \ell) = (\overline f, \overline g)$ with associated
optimal internal variable $\overline z$:
\begin{subequations}
    \begin{align}
        & \boverdot{\overline{z}}(t) = A_s(R\overline{\ell} - Q\overline{z})(t), \quad \overline{z}(0) = z_0, \\
        &\boverdot{\varphi}(t) = \big(Q A'_s(R\overline{\ell} - Q\overline{z})^*\varphi\big)(t) + \ddp{\Psi_1}{z}(\overline{z},\overline{\ell})(t), 
        \quad \varphi(T) = -\ddp{\Psi_2}{z}(\overline{z}(T), \overline{\ell}(T)),\\
        & \left\{ \quad
            \begin{aligned}
                \gamma\, \partial_t^2 \overline\ell(t) + R^*
                A_s'\big(R\overline\ell(t) - Q\overline z(t)\big)^* \varphi(t)
                &= \ddp{\Psi_1}{\ell}(\overline{z}, \overline{\ell})(t), \\
                \overline\ell(0) = 0, \quad \partial_t\overline{\ell}(T) &= -
                \ddp{\Psi_2}{\ell}(\overline{z}(T), \overline{\ell}(T)).
            \end{aligned}
        \right.
    \end{align}
\end{subequations}
Then, owing to the precise structure of $R$, $Q$, $\Psi_1$, and
$\Psi_2$, this leads us to the following

\begin{theorem}[KKT-system for optimal control of homogenized plasticity]\label{thm:nochom}
    Let \cref{assu:data} be satisfied and assume that $A_s$ fulfills
    \cref{assu:Asmooth1}.  Suppose moreover that the regularity
    condition in \cref{assu:sneiberg} is satisfied, i.e.,
    $\bar s \geq s_1$.  Then, if
    $(\overline f, \overline g) \in H^1(0,T;L^2(\Omega)) \times
    H^1(0,T;L^2(\Gamma_N))$ is locally optimal for~\eqref{eq:ocps} with
    associated state
    $(\overline{u}, \overline{v}, \overline{z}, \overline{\Sigma}) \in
    H^1(0,T;U^{s_1}) \times H^1(0,T;V^{s_1}) \times H^1(0,T;Z^{s_1})
    \times H^1(0,T;S^{s_1})$, then there exists an adjoint state
    \begin{equation*}
        \begin{gathered}
            (w, q, \varphi, \Upsilon)
            \in H^1(0,T;U^{2}) \times H^1(0,T;V^{2}) \times H^1(0,T;Z^{2}) \times H^1(0,T;S^{2}),\\
            (w_T, q_T, \Upsilon_T) \in U^2\times V^2\times S^2
        \end{gathered}
    \end{equation*}
    such that the following \emph{optimality system} is satisfied:
    \begin{subequations}\label{eq:opsyshom}
        \begin{align}
            \intertext{State equation:}
            -\Div_{(x,y)} \overline \Sigma &= ((\overline f, \overline g),0), \label{eq:mom}\\
            \overline\Sigma &= \mathbb{C} \big( \symnabla_{(x,y)} (\overline{u}, \overline{v}) - B \overline z\big) ,\\
            \boverdot{\overline{z}} &= A_s(B^\top \overline \Sigma - \Bb\overline{z}), \quad \overline{z}(0) = z_0, \\
            \intertext{Adjoint equation:}
            -\Div_{(x,y)} \Upsilon &= 
            \big(\tddp{}{u}F_1(\overline u, \overline \Sigma), 0\big) 
            - \Div_{(x,y)} \big( \mathbb{C} B A_s'(B^\top \overline \Sigma - \Bb \overline z)^*\varphi\big),\\
            \Upsilon &= \mathbb{C} \big(\symnabla_{(x,y)} (w,q) - \tddp{}{\Sigma}F_1(\overline{u}, \overline{\Sigma})\big),\\
            \boverdot\varphi &= (B^\top \mathbb{C} B + \Bb) A_s'(B^\top \overline \Sigma - \Bb \overline z)^*\varphi + B^\top \Upsilon,
            \;\; \varphi(T) = - B^\top \Upsilon_T,\\
            -\Div_{(x,y)} \Upsilon_T &= \big(\tddp{}{u}F_2(\overline u(T), \overline \Sigma(T)), 0\big), \\
            \Upsilon_T &= \mathbb{C} \big(\symnabla_{(x,y)} (w_T,q_T) - \tddp{}{\Sigma}F_2(\overline{u}(T), \overline{\Sigma}(T))\big),\\
            \intertext{Gradient equation:}
            \gamma \,\partial_t^2 \overline f + w &= 0, 
            \qquad \overline f(0) = 0, \quad \gamma\,\partial_t \overline{f}(T) + w_T = 0, \label{eq:gradeqvol}\\
            \gamma \,\partial_t^2 \overline g + w &= 0, 
            \qquad \overline g(0) = 0, \quad \gamma\,\partial_t \overline{g}(T) + w_T = 0. \label{eq:gradeqbdry}
        \end{align}
    \end{subequations}
\end{theorem}

\begin{remark}
    A passage to the limit w.r.t.\ the regularization in order to obtain
    an optimality system for the original optimal control problem
    involving the maximal monotone operator $A$ would of course be of
    particular interest.  The results of~\cite{Wac16} however indicate
    that the optimality conditions obtained in this way are in general
    rather weak. In~\cite{Wac16}, an optimal control problem governed by
    quasi-static elastoplasticity (without homogenization) is
    considered, which provides substantial similarities to~\eqref{eq:ocps}. This system could also be treated by means of a
    reduction to the internal variable similar to our procedure for~\eqref{eq:homogenizationEquations}.  In~\cite{Wac16} however, a time
    discretization followed by a regularization was employed for the
    derivation of first-order optimality conditions. The reason for the
    comparatively weak optimality conditions obtained for the original
    (non-smooth) problem is the poor regularity of the dual variables in
    the limit, in particular the adjoint state.  We expect a similar
    behavior in case of~\eqref{eq:opsyshom}, when the regularization is
    driven to zero.  This however is subject to future research.
\end{remark}

Next, we turn to second-order sufficient optimality conditions. Now,
$\Psi_1$, $\Psi_2$, and $A_s$ have to fulfill
\cref{ass:AsAndJTwiceFrechet}.  For this purpose, we require the
following

\begin{assumption}\label{assu:twice}
    The mappings $F_1$ and $F_2$ from~\eqref{eq:objuSigma} are twice
    Fr\'echet diff\-er\-entiable. Moreover, $A_s$ satisfies
    \cref{ass:AsAndJFrechet} (ii) and \cref{ass:AsAndJTwiceFrechet}(ii)
    and (iii) with $\WW = \HH = Z^2$ (i.e., $s_3 = 2$), $\ZZ = Z^{s_2}$,
    $s_2 \geq 2$, and $\YY = Z^{s_1}$, $s_1 \geq s_2$.
\end{assumption}

In contrast to \cref{assu:Asmooth1}, this assumption is very
restrictive. If we assume that $A_s$ arises as a Nemyzki operator from
a nonlinear function $A_s\colon  \mathbb{V} \to \mathbb{V}$ (for simplicity
denoted by the same symbol), then the last condition in
\cref{ass:AsAndJTwiceFrechet}, i.e.,
\begin{equation}\label{eq:cruineq}
    \|A_s''(z)[z_1, z_2]\|_{\WW} \leq C\, \|z_1\|_{Z^{s_2}} \|z_2\|_{Z^{s_2}} \quad \forall\, z\in Z^{s_1}, \; z_1, z_2 \in Z^{s_2}
\end{equation}
with $\WW = Z^2$, may only hold---even in case
$A_s''(z) \in L^\infty(\Omega\times Y; \LL(\mathbb{V},
\mathbb{V}))$---provided that $s_2 \geq 2 \,s_3 = 4$.  In order to
have that $A_s$ is Fr\'echet differentiable from $Z^{s_1}$ to
$Z^{s_2}$, we therefore need $s_1 > 4$ such that \cref{assu:sneiberg}
indeed becomes very restrictive, see \cref{rem:sneiberg} and
\cref{rem:groesser4} below.  Moreover, as described at the beginning
of this subsection, if one sets $r=2$, i.e., considers to boundary
loads in $H^1(0,T;L^2(\Gamma_N;\R^d))$, then, in view of~\eqref{eq:sobolev}, $s_1 < 3$ has to hold (at least in three spatial
dimensions) so that our second-order analysis cannot be applied in
case of boundary controls (at least if $r=2$ and $d=3$). Therefore, we
restrict to volume forces, i.e., controls in
$H^1(0,T;L^2(\Omega;\R^d))$ in what follows.

Then, given all these assumptions, we can apply
\cref{thm:SecondOrderOptimalityCondition} and \cref{cor:ssc} to~\eqref{eq:ocps} to obtain the following

\begin{theorem}[Second-order sufficient conditions for optimal control of homogenized plasticity]
    Let Assumptions~\ref{assu:data} and~\ref{assu:twice} hold and
    suppose that $\bar s \geq s_1$ so that \cref{assu:sneiberg} is
    fulfilled.  Assume moreover that
    $\overline{f} \in H^1(0,T;L^2(\Omega;\R^d))$ together with its
    associated state
    $(\overline u, \overline v, \overline z, \overline \Sigma)$ and an
    adjoint state $(w, q, \varphi, \Upsilon, w_T, q_T, \Upsilon_T)$
    satisfies the optimality system~\eqref{eq:mom}--\eqref{eq:gradeqvol}
    (without~\eqref{eq:gradeqbdry} because boundary controls are
    omitted) and, in addition, that there exists $\delta > 0$ such that
    \begin{equation*}
        \begin{aligned}
            & \tddp{^2}{u^2} F_1(\overline u, \overline \Sigma) \eta_u^2
            + \tddp{^2}{\Sigma^2} F_1(\overline u, \overline \Sigma) \eta_\Sigma^2 \\
            &\; + \tddp{^2}{u^2} F_2(\overline u(T), \overline \Sigma(T))
            \eta_u(T)^2 + \tddp{^2}{\Sigma^2} F_2(\overline u(T), \overline
            \Sigma(T)) \eta_\Sigma(T)^2
            + \gamma \,\|\boverdot h\|_{L^2(0,T;L^2(\Omega;\R^d))}^2 \\
            & \qquad\qquad + \int_0^T\int_\Omega \big(\varphi , A_s''(B^\top
                \overline\Sigma - \Bb \overline z)(B^\top \eta_\Sigma - \Bb
            \eta_z)^2\big)_{\mathbb{V}} \,dx\,dt \geq \delta\,
            \|h\|_{H^1(0,T; L^2(\Omega;\R^d))}^2
        \end{aligned}
    \end{equation*}
    holds for all $h\in H^1(0,T;L^2(\Omega;\R^d))$ with $h(0) = 0$,
    where
    $(\eta_u, \eta_v, \eta_z, \eta_\Sigma) \in H^1(0,T;U^2 \times V^2
    \times Z^2\times S^2)$ is the solution of the linearized state
    system associated with $h$, i.e.,
    \begin{equation*}
        \begin{aligned}
            -\Div_{(x,y)} \eta_\Sigma &= (h,0), \\
            \eta_\Sigma &= \mathbb{C} \big( \symnabla_{(x,y)} (\eta_u,\eta_v) - B \eta_z\big) ,\\
            \boverdot{\eta_z} &= A_s'(B^\top \overline \Sigma -
            \Bb\overline{z})(B^\top \eta_\Sigma - \Bb\eta_z), \quad
            \eta_z(0) = 0.
        \end{aligned}
    \end{equation*}     
    Then $\overline{f}$ is a strict local minimizer of~\eqref{eq:ocps}
    and satisfies the quadratic growth condition~\eqref{eq:qgc}.
\end{theorem}

\begin{remark}\label{rem:groesser3}
    As indicated above, \cref{assu:twice} in combination with
    \cref{assu:sneiberg} is very restrictive.  One can however weaken
    these assumptions, if the objective provides certain properties. Let
    us for instance consider an objective of the form
    \begin{equation}
        J(u, f) \defn  
        F_3(u) + \tfrac{\gamma}{2} \, \|\boverdot f\|_{L^2(0,T;L^2(\Omega;\R^d))}^2
    \end{equation}
    with a twice Fr\'echet differentiable functional
    \begin{equation*}
        F_3 \colon  H^1(0,T;L^2(\Omega;\R^d)) \to \R.
    \end{equation*}
    In this case, it is sufficient to choose $s_3$ such that
    $\mfu\colon  (z, \ell) \mapsto u$ maps $\WW \times \XX_c$ with
    $\WW = Z^{s_3}$ to $W^{1,p}(\Omega;\R^d) \embed L^2(\Omega;\R^d)$,
    i.e., $p \geq 6/5$ for $d=3$.  Moreover, as seen above, in order to
    have that the Nemyzki operator $A_s$ fulfills~\eqref{eq:cruineq}, we
    need $s_1 > 2 s_3$. Thus, we are tempted to choose $s_3$ as small as
    possible. However, the crucial, limiting condition is the regularity
    assumption in \cref{assu:sneiberg} involving the conjugate exponent,
    i.e., $\bar s \geq \max\{s_1, s_3'\}$, and this leads to the
    following equilibration of $s_1$ and $s_3$ in the case $d=3$:
    $s_1 > 3$, $s_3 = 3/2$ (such that $s_3' = 3$).  Then, in view of~\eqref{eq:uv}, $\mfu$ maps $Z^{s_3}$ to $W^{1,s_3}(\Omega;\R^d)$,
    which is continuously embedded in $L^2(\Omega;\R^d)$ for $d\leq 3$
    as desired.  In the next subsection, we will present an example for
    a Nemyzki operator $A_s$ fulfilling all assumptions for $s_1$
    arbitrarily close to $3$. But nonetheless, $\bar s > 3$ in
    \cref{assu:sneiberg} is still a rather restrictive assumption and
    will not be satisfied in general (depending on the regularity of
    $\mathbb{C} $ and the boundary of $\Omega$). This shows that the
    second-order analysis for problems of type~\eqref{eq:optprob}
    (resp.~its regularized counterparts, to be precise) is in general a
    delicate issue, mainly due to the quasi-linear structure of the
    state equation.
\end{remark}

\subsection{A Concrete Flow Rule}\label{subsec:vonmises}

In the following, we will discuss a concrete realization of the
maximal monotone operator $A$ and its regularization, respectively, in
order to demonstrate how the Assumptions~\ref{assu:maxmon},~\ref{assu:Asmooth1}, and~\ref{assu:twice} can be satisfied in practice
and how restrictive they are, in particular
Assumption~\ref{assu:twice}.

We consider the case of linear kinematic hardening with von Mises
yield condition, cf.~\cite{HanReddy1999} for a detailed description of this
model.  In this case, the finite dimensional space for the internal
variable is given by $\mathbb{V} = \Rs$ and $B\colon  \Rs\to \Rs$ is the
identity so that $Z^2 = S^2$. Moreover, $A$ is the convex
subdifferential of the indicator functional $I_\KK$ of following set
of admissible stresses
\begin{equation*}
    \KK \defn  \{ \tau \in S^s \colon  |\tau^D(x,y)| \leq \sigma_0 \quad \text{f.a.a.~} (x,y) \in \Omega \times Y\},
\end{equation*}
where $\tau^D \defn  \tau - \frac{1}{d} \operatorname{tr}(\tau)I$ is the
deviator of $\tau \in \Rs$ and $\sigma_0$ denotes the initial
uni-axial yield stress, a given material parameter.  The domain of
$A = \partial I_\KK$ is trivially $\KK$, which is closed and
convex. Furthermore, it is easily seen that $A^0(\tau) = 0$ for all
$\tau \in D(A) = \KK$ so that all of our standing assumptions are
fulfilled in this case.  Note moreover that $A$ satisfies
\cref{ass:AIsSubdifferential} so that the second approximation result
on the convergence of the optimal states in \cref{cor:strongapprox}
applies in this case.  For the Yosida-approximation of
$\partial I_\KK$, one obtains
\begin{equation}\label{eq:proj}
    A_\lambda = \frac{1}{\lambda}(I - \pi_{\mathcal{K}}) 
    = \frac{1}{\lambda} \max \Big\{ 0, 1 - \frac{\sigma_0}{|\tau^D|} \Big\} \tau^D,
\end{equation}
cf.~\cite{hermey11}, where $\pi_\KK$ denotes the projection on $\KK$ in
$Z^2$, i.e.,
$\pi_\KK(\sigma) \defn  \argmin_{\tau \in \KK} \|\tau - \sigma\|_{Z^2}^2$.
Herein, with a slight abuse of notation, we denote the Nemyzki
operator in $L^\infty(\Omega \times Y)$ associated with the pointwise
maximum, i.e., $\R \ni r \mapsto \max\{0,r\} \in \R$, by the same
symbol.  In addition, we set $\max\{0, 1 - \sigma_0/r\} \defn  0$, if
$r = 0$.

The precise form of $A_\lambda$ in~\eqref{eq:proj} immediately
suggests the following regularization of the Yosida approximation:
\begin{equation*}
    A_{\lambda, \epsilon}(\tau) \defn  \frac{1}{\lambda} \maxs_\epsilon \Big( 1 - \frac{\sigma_0}{|\tau^D|} \Big) \tau^D,
\end{equation*}
where $\maxs_\epsilon$ is a regularized version of the max-function,
depending on a regularization parameter $\epsilon > 0$.  To be more
precise,
$\maxs_\epsilon\colon  L^\infty(\Omega \times Y) \to L^\infty(\Omega \times
Y)$ is the Nemyzki operator associated with a real valued function
(again denoted by the same symbol) with the following properties:
\begin{enumerate}
    \item For every $\epsilon > 0$, there holds
        $\maxs_\epsilon \in C^2(\R)$,
    \item $\maxs_\epsilon(r) = \max\{0, r\}$ for $|r| \geq \frac{1}{2}$
        and every $0 < \epsilon \leq 1/2$,
    \item
        $| \maxs_\epsilon(r) - \max\{0, r\} | \leq \mathcal{O}(\epsilon)$
        for all $r \in \mathbb{R}$.
\end{enumerate}

\begin{example}
    An example for a function satisfying the above conditions is
    \begin{equation*}
        \maxs_\epsilon(r) \defn  
        \begin{cases}
            \max\{0,r\}, & |r| \geq \epsilon, \\
            \tfrac{1}{16\epsilon^3} (r+\epsilon)^3 (3\epsilon-r), & |r| <
            \epsilon.
        \end{cases}
    \end{equation*}
    One easily verifies that $\maxs_\varepsilon$ is twice continuously
    differentiable and that
    $|\maxs_\varepsilon(r) = \max\{0,r\}| \leq \frac{3}{16} \epsilon$.
\end{example}

\begin{lemma}
    Let $\sequence{\lambda}{n} \subset \R^+$ and
    $\sequence{\epsilon}{n}\subset \R^+$ be null sequences satisfying
    \begin{equation}\label{eq:epslambda}
        \epsilon_n = o\Big(\lambda_n^2 \exp\big(-\tfrac{T\|Q\|}{\lambda_n}\big)\Big), 
    \end{equation}
    and define $A_n \defn  A_{\lambda_n, \epsilon_n}\colon  Z^2 \to Z^2$. Then,
    the sequence $\sequence{A}{n}$ satisfies
    \cref{ass:AnAssumption}. Thus, \cref{assu:maxmon} is fulfilled in
    this case so that the approximation results from
    \cref{thm:approxhom} apply.
\end{lemma}

\begin{proof}
    Based on our assumptions on $\maxs_\epsilon$, we find for every
    $\tau \in Z^2$ and all $n\in \N$ such that $\epsilon_n \leq 1/2$
    that
    \begin{equation*}
        \begin{aligned}
            & \norm{A_n(\tau) - A_{\lambda_n}(\tau)}{Z^2}^2 \\
            &\quad = \frac{1}{\lambda_n^2} \int_\Omega
            \Big|\maxs_{\epsilon_n} \big( 1 - \tfrac{\sigma_0}{|\tau^D|}
                \big) - \max \big\{ 0, 1 - \tfrac{\sigma_0}{|\tau^D|}
            \big\}\Big|^2 |\tau^D|^2 \,dx \leq
            C\,\frac{\epsilon_n^2}{\lambda_n^2}.
        \end{aligned}
    \end{equation*}
    The coupling of $\epsilon_n$ and $\lambda_n$ in~\eqref{eq:epslambda}
    then implies that~\eqref{eq:approxassu} is fulfilled.
\end{proof}

\begin{remark}
    We point out that we neither claim that the coupling of $\lambda$
    and $\epsilon$ in~\eqref{eq:epslambda} is optimal nor that our
    regularization approach is the most efficient one for this specific
    flow rule.
\end{remark}

Let us now fix $n \in \mathbb{N}$ and set $\lambda_s \defn  \lambda_n$,
$\maxs_s \defn  \maxs_{\epsilon_n}$, and $A_s \defn  A_n$.  As before, we will
denote the Nemyzki operators generated by $\maxs_s$ and its
derivatives by the same symbol.  The following result can be proven as
in~\cite[Prop.~2.11]{herzog2} by using~\cite[Theorem~7]{GKT92}:

\begin{lemma}\label{lem:Asdiff1}
    Let $s>r\geq 1$ be arbitrary. Then $A_s$ is continuously
    Fr\'echet differentiable from $Z^s$ to $Z^r$ and its directional
    derivative at $\tau \in Z^s$ in direction $h \in Z^r$ is given by
    \begin{equation*}
        A_s'(\tau)h = \frac{1}{\lambda_s} \maxs_s'\Big(1 - \frac{\sigma_0}{|\tau^D|}\Big) \frac{\sigma_0}{|\tau^D|^3}
        (\tau^D \colon h^D) \tau^D + \frac{1}{\lambda_s} \maxs_s\Big(1 - \frac{\sigma_0}{|\tau^D|}\Big) h^D.
    \end{equation*}
\end{lemma}

As a consequence of this result, we obtain the following

\begin{corollary}
    \cref{assu:Asmooth1} is fulfilled by setting $s_1 \defn  \bar s$, where
    $\bar s>2$ is the exponent, whose existence is guaranteed by
    \cref{lem:w1sExistenceHomogenization}. Thus, in case of linear
    kinematic hardening with von Mises yield condition and the
    regularization introduced above, the optimality condition in
    \cref{thm:nochom} are indeed necessary for local optimality without
    any further assumptions (except our standing \cref{assu:data}).
\end{corollary}

\begin{proof}
    We have to verify \cref{ass:AsAndJFrechet}(ii) for
    $\YY = Z^{\bar s}$ and $\ZZ = \HH = Z^2$.  The
    Fr\'echet differentiability from $Z^{\bar s}$ to $Z^2$ follows from
    \cref{lem:Asdiff1}. Moreover, the (global) Lipschitz continuity of
    $A_s$ in $Z^{\bar s}$ can be deduced from the smoothness of
    $\maxs_s$ and the condition $\maxs_s(r) = \max\{0,r\}$ for all
    $|r| \geq 1/2$. The latter condition also guarantees that
    $\|A_s'(y)h\|_{Z^2} \leq C\,\|h\|_{Z^2}$ for all $y\in Z^{\bar s}$
    and all $h\in Z^2$. This completes the proof.
\end{proof}

Furthermore, following the lines of~\cite{wachsmuth2} and~\cite[Theorem~9]{GKT92}, one proves the following

\begin{lemma}\label{lem:Asdiff2}
    For every $s > 2$ and $1 \geq r < s/2$, $A_s$ is twice
    Fr\'echet differentiable and its second derivative at $\tau \in Z^s$
    in directions $h_1, h_2 \in Z^r$ is given by
    \begin{equation*}
        \begin{aligned}
            A_s''(\tau)[h_1, h_2 ]
            &= \tfrac{\gamma} {\lambda_s |\tau^D|^3} \Big[
                \begin{aligned}[t]
                    & \maxs_s''\big(1 - \tfrac{\gamma}{|\tau^D|} \big)
                    \tfrac{\gamma}{|\tau^D|^3}
                    (\tau^D \colon h_1^D) (\tau^D \colon h_2^D) \tau^D\\
                    &+ \maxs_s'\big(1 - \tfrac{\gamma}{|\tau^D|} \big) \Big(
                        \begin{aligned}[t]
                            -\tfrac{3}{|\tau^D|^2} (\tau^D \colon h_1^D) (\tau^D \colon
                            h_2^D)\tau^D
                            + (h_1^D \colon h_2^D)  \tau^D &\\[-1mm]
                            + (\tau^D \colon h_1^D) h_2^D + (\tau^D \colon h_2^D) h_1^D
                    & \Big)\Big].
                \end{aligned}
            \end{aligned}
        \end{aligned}
    \end{equation*}
\end{lemma}

\begin{corollary}
    The conditions on $A_s$ in \cref{assu:twice} are satisfied, if the
    index $\bar s$ from \cref{lem:w1sExistenceHomogenization} fulfills
    $\bar s > 4$.
\end{corollary}

\begin{proof}
    If we set $s_1 = \bar s > 4$, $s_2 \in ]4, s_1[$, and $s_3 = 2$,
    then \cref{lem:Asdiff2} implies the differentiability conditions in
    \cref{ass:AsAndJTwiceFrechet}(iii) with $\YY = Z^{s_1}$,
    $\ZZ=Z^{s_2}$, and $\WW = Z^{s_3}$.  The Lipschitz continuity of
    $A_s'$ from $Z^{s_1}$ to $L(Z^{s_2})$ as well as the estimate
    $\|A'_s(y)w\|_{Z^2} \leq C\, \|w\|_{Z^2}$ follow from the condition
    $\maxs_s(r) = \max\{0,r\}$ for all $|r| \geq 1/2$. This condition
    also ensures that
    $\|A_s''(y)[z_1, z_2]\|_{Z^2} \leq C\, \|z_1\|_{Z^{4}}
    \|z_2\|_{Z^{4}}$, which in turn implies the last condition in
    \cref{ass:AsAndJTwiceFrechet}(iii) thanks to $s_2 > 4$.
\end{proof}

\begin{remark}\label{rem:groesser4}
    As indicated in \cref{rem:sneiberg}, the assumption $\bar s > 4$ is
    very restrictive. However, if $\WW = Z^2$, then any Nemyzki operator
    is only twice Fr\'echet differentiable from $\YY$ to $\WW$, if
    $\YY = Z^s$ with $s>4$, see e.g.~\cite{GKT92} and the references
    therein. In light of this observation, the above regularization is
    rather well-behaved. As explained in \cref{rem:groesser3}, one can
    reduce the value of $s_3$, if only the $L^2$-norm of the
    displacement appears in the objective.  However, one still needs
    $\bar s > 3$ in this case, which is not guaranteed by
    \cref{lem:w1sExistenceHomogenization} in general.  But again, one
    can show that any $\bar s > 3$ is sufficient for our concrete
    example, no matter how close it is to 3.  This concrete realization
    of $A_s$ thus allows for the weakest possible regularity
    assumptions.
\end{remark}


\appendix

\section{Second Derivative of the Solution
Operator}\label{sec:secondderiv}

Before we are in the position to show that $\mathcal{S}$ is twice
Fr\'echet differentiable, we need the following result on the
Lipschitz continuity of the first derivative, which is also needed
in the proof of \cref{lem:Sp''Inequality}.

\begin{lemma}
    \label{S_p'Lipschitz}
    Assume that \cref{ass:AsAndJFrechet}(ii) and
    \cref{ass:AsAndJTwiceFrechet}(ii) are fulfilled.  Then
    $\mathcal{S}_s'$ is Lipschitz continuous from $\WZ{\mathspace{X}}$
    to $L(\WZ{\mathspace{X}};\WZ{\mathspace{Z}})$.
\end{lemma}

\begin{proof}[Proof of \cref{prp:SpTwiceFrechetDifferentiable}] Let
    $\ell_1, \ell_2, h \in \WZ{\mathspace{X}}$ be arbitrary and
    abbreviate
    \begin{equation*}
        z_i \defn  \mathcal{S}_s(\ell_i), \quad \eta_i \defn  \mathcal{S}_s'(\ell_i) h, \quad \text {and} \quad 
        y_i \defn  R\ell_i - Q z_i, \quad i=1,2.    
    \end{equation*}
    Using the first Lipschitz-assumption in
    \cref{ass:AsAndJTwiceFrechet}(ii), we deduce for almost all
    $t \in [0, T]$ that
    \begin{align*}
        \norm{\boverdot{\eta}_1(t) - \boverdot{\eta}_2(t)}{\mathspace{Z}}
        &= \norm{\big(A_s'(y_1(t)) - A_s'(y_2(t))\big)(Rh(t) - Q \eta_1(t))
        + A_s'(y_2(t)) Q (\eta_1(t) - \eta_2(t))}{\mathspace{Z}} \\
        &\leq C\big( \norm{y_1(t) - y_2(t)}{Y}	\norm{Rh(t) - Q \eta_1(t)}{\mathspace{Z}} 
        +  \norm{\eta_1(t) - \eta_2(t) \rangle}{\mathspace{Z}}\big).
    \end{align*}
    Gronwall's inequality and the definition of $y_1$ and $y_2$ then
    yield
    \begin{align*}
        \norm{\eta_1 - \eta_2}{\WZ{\mathspace{Z}}}
        &\leq C \norm{R(\ell_1 - \ell_2) - Q(z_1 - z_2)}{\LZ{Y}} \norm{Rh - Q \eta_1}{\WZ{\mathspace{Z}}} \\
        &\leq C \norm{\ell_1 - \ell_2}{\LZ{\mathspace{X}}} \norm{h}{\WZ{\mathspace{X}}},
    \end{align*}
    where we used \cref{prp:solutionOperatorLipschitz} and the
    estimate in \cref{thm:SsFrechetDifferentiability}.  
\end{proof}

Now, we are ready to prove that the solution operator is twice
Fr\'echet-dif\-fer\-en\-ti\-a\-ble.  The proof is based on an estimate of the
remainder term and thus similar to the one of
\cref{thm:SsFrechetDifferentiability}.

\begin{proof} The proof is similar to the one of
    \cref{thm:SsFrechetDifferentiability}.  Let
    $\ell, h_1, h_2 \in \WZ{\mathspace{X}}$ be arbitrary and define
    $z \defn  \mathcal{S}_s(\ell)$, $z_1 \defn  \mathcal{S}_s(\ell + h_1)$,
    $\eta_{i} \defn  \mathcal{S}'_s(\ell) h_i \in \WZ{\mathspace{Z}}$,
    $i \in \{ 1,2 \}$, and
    $\eta_{1,2} \defn  \mathcal{S}'_s(\ell + h_1) h_2$.

    We first address the existence of solutions to~\eqref{eq:xiEquation}. We argue similarly to
    \cref{lem:frechetEquationExistence} and set
    \begin{align*}
        w \colon  [0, T] \rightarrow \mathspace{W}, \quad t \mapsto A_s''(R\ell(t) - Qz(t))[Rh_1(t) - Q\eta_{1}(t), Rh_2(t) - Q\eta_2(t)].
    \end{align*}
    From the estimate in \cref{ass:AsAndJTwiceFrechet}(iii) it follows
    that
    \begin{align*}
        \norm{w(t)}{\mathspace{W}}
        \leq C \norm{Rh_1(t) - Q\eta_{1}(t)}{\mathspace{Z}} \norm{Rh_2(t) - Q\eta_2(t)}{\mathspace{Z}},
    \end{align*}
    and, since the limit of Bochner measurable functions is Bochner
    measurable, we obtain $w \in \LZ{\mathspace{W}}$.  Since $A_s'(y)$
    is assumed to be bounded in $\WW$ by~\ref{ass:AsAndJTwiceFrechet}(ii), we can now follow the proof of
    \cref{lem:frechetEquationExistence} (with $\WW$ instead of $\ZZ$)
    to deduce the existence of a unique solution
    $\xi \in \WZ{\mathspace{W}}$ of \cref{eq:xiEquation}.  The
    (bi-)linearity of the associated solution operator w.r.t.~$h_1$
    and $h_2$ is straightforward to see.  For its continuity, we
    calculate
    \begin{align*}
        \norm{\boverdot{\xi}(t)}{\mathspace{W}}
        &\leq C \norm{Rh_1(t) - Q\eta_{1}(t)}{\mathspace{Z}} \norm{Rh_2(t) - Q\eta_2(t)}{\mathspace{Z}}
        + C \norm{\xi(t)}{\mathspace{W}}
    \end{align*}
    so that Gronwall's inequality and the estimate in
    \cref{thm:SsFrechetDifferentiability} give
    \begin{equation*}
        \norm{\xi}{\WZ{\mathspace{W}}} 
        \leq C \norm{h_1}{\WZ{\mathspace{X}}} \norm{h_2}{\WZ{\mathspace{X}}}.
    \end{equation*}
    This shows also~\eqref{eq:S2est} (after having proved that
    $\xi = \mathcal{S}''_s(\ell)[h_1, h_2]$).

    It only remains to prove the remainder term property. To this end,
    we define
    \begin{equation*}
        y \defn  R\ell - Qz, \quad \zeta \defn   R h_1 - Q(z_1 - z).
    \end{equation*}    	
    Then, the equations for $\eta_{1,2}$, $\eta_2$, and $\xi$ lead to
    \begin{align*}
        \boverdot{\eta}_{1,2} - \boverdot{\eta}_{2} - \boverdot{\xi}
        & = \big( A_s'(y + \zeta) - A_s'(y) \big) (Rh_2 - Q\eta_{1,2}) \\
        & \quad - A_s''(y)[ Rh_1 - Q\eta_{1}, Rh_2 - Q\eta_2]
        - A_s'(y) Q(\eta_{1,2} - \eta_2 - \xi) \\
        & = A_s''(y)[\zeta, Rh_2 - Q\eta_{1,2}] + r_2(y;\zeta)( Rh_2 - Q\eta_{1,2}) \\
        & \quad - A_s''(y) [ Rh_1 - Q\eta_{1}, Rh_2 - Q\eta_2]  - A_s'(y) Q(\eta_{1,2} - \eta_2 - \xi) \\
        & = A_s''(y)[\zeta, Q(\eta_{2} - \eta_{1,2})]  - A_s''(y) [Q(z_1 - z - \eta_{1}), Rh_2 - Q\eta_2] \\
        & \quad + r_2(y;\zeta)(Rh_2 - Q\eta_{1,2})  - A_s'(y) Q(\eta_{1,2} - \eta_2 - \xi) ,
    \end{align*}
    where
    $r_2(y;\zeta) \defn  A_s'(y + \zeta) - A_s'(y) - A_s''(y)\zeta \in
    \LZ{L(\ZZ;\WW)}$ denotes the corresponding remainder term. The
    estimate in \cref{ass:AsAndJTwiceFrechet}(iii) thus implies
    \begin{align*}
        & \norm{\boverdot{\eta}_{1,2}(t) - \boverdot{\eta}_{2}(t) - \boverdot{\xi}(t)}{\mathspace{W}} \\
        &\; \leq C \big(\norm{\zeta(t))}{\mathspace{Z}} \norm{\eta_{2}(t) - \eta_{1,2}(t)}{\mathspace{Z}} 
            + \norm{z_1(t) - z(t) - \eta_{1}(t)}{\mathspace{Z}} \norm{Rh_2(t) - Q\eta_2(t)}{\mathspace{Z}} \\
            & \qquad\quad + \norm{r_2(y(t), \zeta(t))}{L(\mathspace{Z};\mathspace{W})}
            \norm{Rh_2(t) - Q\eta_{1,2}(t)}{\mathspace{Z}} 
        + \norm{\eta_{1,2}(t) - \eta_2(t) - \xi(t)}{\mathspace{W}} \big)
    \end{align*}
    for almost all $t \in [0,T]$ such that Gronwall's inequality
    yields
    \begin{align*}
      & \norm{\eta_{1,2} - \eta_2 - \xi}{\WZ{\mathspace{W}}} \\
      & \leq C \big( \norm{Rh_1 - Q(z_1 - z)}{L^\infty(0,T;\ZZ)}
      \norm{\eta_{2} - \eta_{1,2}}{\LZ{\mathspace{Z}}}
      + \norm{z_1 - z - \eta_{1}}{L^\infty(0,T;\ZZ)}
      \norm{Rh_2 - Q\eta_2}{\LZ{\mathspace{Z}}} \\
      &\qquad\qquad + \norm{r_2(y; \zeta)}{\LZ{L(\mathspace{Z};\mathspace{W})}} 
      \norm{Rh_2 - Q\eta_{1,2}}{\WZ{\mathspace{Z}}} \big) \\
      &\leq C \norm{h_2}{\WZ{\mathspace{X}}}  
      \big(\norm{h_1}{\WZ{\mathspace{X}}}^2 + \norm{z_1 - z - \eta_{1}}{\WZ{\mathspace{Z}}}
      + \norm{r_2(y; \zeta)}{\LZ{L(\mathspace{Z};\mathspace{W})}} \big),
    \end{align*}
    where we used \cref{prp:solutionOperatorLipschitz},
    \cref{S_p'Lipschitz} and the estimate in
    \cref{thm:SsFrechetDifferentiability}. Denoting the solution
    operator of \cref{eq:xiEquation} already by
    $\mathcal{S}''_s(\ell)[h_1, h_2]$, we have thus shown
    \begin{multline*}
        \norm{\mathcal{S}'_s(\ell + h_1) - \mathcal{S}'_s(\ell) - \mathcal{S}''_s(\ell) h_1}{L(\WZ{\mathspace{X}};\WZ{\mathspace{W}})} \\
        \leq C \big( \norm{h_1}{\WZ{\mathspace{X}}}^2 +
            \norm{\mathcal{S}_s(\ell + h_1) - \mathcal{S}_s(\ell) - \mathcal{S}_s'(\ell)h_1}{\WZ{\mathspace{Z}}}
            + \norm{r_2(y; \zeta)}{\LZ{L(\mathspace{Z};\mathspace{W})}}
        \big).
    \end{multline*}
    Therefore, thanks to the Fr\'{e}chet differentiability of
    $\mathcal{S}_s \colon  \WZ{\mathspace{X}} \rightarrow
    \WZ{\mathspace{Z}}$, it only remains to show that
    \begin{align}
        \label{6.4}
        \frac{\norm{r_2(y; \zeta)}{\LZ{L(\mathspace{Z};\mathspace{W})}}}{\norm{h_1}{\WZ{\mathspace{X}}}} \rightarrow 0,
    \end{align}
    as $0 \neq h_1 \rightarrow 0$ in $\WZ{\mathspace{X}}$.  To this
    end, we note that the embedding $\WZ{\YY} \embed \CO{\YY}$ and
    \cref{prp:solutionOperatorLipschitz} yield for all $t\in [0,T]$
    \begin{align}
        \label{6.2}
        \frac{\norm{\zeta(t)}{\mathspace{Y}}}{\norm{h_1}{\WZ{\mathspace{X}}}}
        \leq C\, \frac{\norm{\zeta}{\WZ{\mathspace{Y}}}}{\norm{h_1}{\WZ{\mathspace{X}}}}
        = C\,\frac{\norm{Rh_1 - Q(z_1 - z)}{\WZ{\mathspace{Y}}}}{\norm{h_1}{\WZ{\mathspace{X}}}} \leq C
    \end{align}
    Hence, thanks to the Fr\'{e}chet differentiability of
    $A_s' \colon  \mathspace{Y} \rightarrow L(\mathspace{Z};\mathspace{W})$,
    we have for almost all $t \in [0, T]$
    \begin{align*}
        \frac{\norm{r_2(y; \zeta)(t)}{L(\mathspace{Z};\mathspace{W})}}{\norm{h_1}{\WZ{\mathspace{X}}}}
        \leq C \frac{\norm{r_2(y; \zeta)(t)}{L(\mathspace{Z};\mathspace{W})}}{\norm{\zeta(t)}{\mathspace{Y}}} \rightarrow 0
    \end{align*}
    as $0 \neq h_1 \rightarrow 0$ in
    $\WZ{\mathspace{X}}$. Furthermore, using the Lipschitz continuity
    of
    $A_s' \colon  \mathspace{Y} \rightarrow L(\mathspace{Z};\mathspace{Z})$,
    the estimate for $A_s''$ in \cref{ass:AsAndJTwiceFrechet}(iii) and
    again \cref{6.2}, we deduce
    \begin{equation*}
        \begin{aligned}
            \frac{\norm{r_2(y; \zeta)(t)}{L(\mathspace{Z};\mathspace{W})}}{\norm{h_1}{\WZ{\mathspace{X}}}}
        = \frac{\norm{A_s'(y(t) + \zeta(t)) - A_s'(y(t)) -
            A_s''(y(t))
            \zeta(t)}{L(\mathspace{Z};\mathspace{W})}}{\norm{h_1}{\WZ{\mathspace{X}}}}
        \leq C
            \frac{\norm{\zeta(t)}{\mathspace{Y}}}{\norm{h_1}{\WZ{\mathspace{X}}}}
            \leq C
        \end{aligned}
    \end{equation*}
    for almost all $t \in [0,T]$. The convergence in \cref{6.4} now
    follows from Lebesgue's dominated convergence theorem.  
\end{proof}


\section{Interpolation for the $V^s$ spaces}

We prove that the spaces
$V^s = L^s(\Omega;W^{1,s}_{\per,\perp}(Y;\R^d))$ defined in
\cref{sec:7} form a complex interpolation scale in $s$. This result
is a cornerstone in the proof of
\cref{lem:w1sExistenceHomogenization}.

\begin{lemma}
    \label{lem:interpolation-periodic}
    Let $\theta \in (0,1)$ and $s_0,s_1 \in (1,\infty)$ and set $\frac1s =
    \frac{1-\theta}{s_0} + \frac\theta{s_1}$. Then
    \begin{equation}
        \label{eq:per-interp}
        \Bigl[W^{1,s_0}_{\per}(Y;\R^d),
        W^{1,s_1}_{\per}(Y;\R^d)\Bigr]_\theta = W^{1,s}_{\per}(Y;\R^d).
    \end{equation}
\end{lemma}

\begin{proof}
    The proof relies on the complemented subspace interpolation
    theorem~\cite[Theorem~1.17.1]{Triebel:1978} which essentially says
    that one can transfer interpolation properties to complemented
    subspaces provided there exists a common projection onto these
    subspaces on all involved spaces.

    In this spirit, we first consider a larger regular domain
    $Y^\# \supset Y$ which includes a finite open covering of $Y$, and,
    for all $1 < r < \infty$, identify $W^{1,r}_\per(Y;\R^d)$
    isomorphically with the closed subspace
    $W^{1,r}_{\per,Y}(Y^\#;\R^d)$ of $W^{1,r}(Y^\#;\R^d)$ consisting of
    periodic extensions of $W^{1,r}_\per(Y;\R^d)$-functions. This is
    possible since the periodic extension of a $W^{1,r}_\per(Y;\R^d)$
    function will preserve the Sobolev
    regularity~\cite[Proposition~3.50]{CD99}.

    We next argue that there exists a projection $P_\per$ mapping 
    $W^{1,r}(Y^\#;\R^d)$ onto $W^{1,r}_{\per,Y}(Y^\#;\R^d)$. (We will not
    give a detailed proof of this since the details are somewhat tedious
    and lengthy.)  To this end, we first wrap $u \in W^{1,r}(Y^\#;\R^d)$
    around the torus $\mathbb{T}^d$ in a smooth manner using a fixed
    smooth partition of unity on $\mathbb{T}^d$ derived from the open
    covering of $Y$, and then pull it back. One checks that this indeed
    yields a function $P_\per u$ which is periodic on $Y$. Moreover,
    $P_\per$ is in fact a continuous linear operator on
    $W^{1,r}(Y^\#;\R^d)$, which in addition acts as the identity on the
    periodic extensions of $C^\infty_\per(Y;\R^d)$. This implies that
    $P_\per$ is indeed the searched-for projection of
    $W^{1,r}(Y^\#;\R^d)$ onto $W^{1,r}_{\per,Y}(Y^\#;\R^d)$.

    The  complemented subspace interpolation
    theorem~\cite[Theorem~1.17.1]{Triebel:1978} then allows to argue as follows:
    \begin{multline*}
        \Bigl[W^{1,s_0}_{\per}(Y;\R^d),
            W^{1,s_1}_{\per}(Y;\R^d)\Bigr]_\theta =  \Bigl[W^{1,s_0}_{\per,Y}(Y^\#;\R^d),
        W^{1,s_1}_{\per,Y}(Y^\#;\R^d)\Bigr]_\theta \\ =
        \Bigl[W^{1,s_0}(Y^\#;\R^d) \cap W^{1,\max(s_0,s_1)}_{\per,Y}(Y^\#;\R^d),
            W^{1,s_1}(Y^\#;\R^d)\cap
        W^{1,\max(s_0,s_1)}_{\per,Y}(Y^\#;\R^d)\Bigr]_\theta \\ =
        \Bigl[W^{1,s_0}(Y^\#;\R^d), W^{1,s_1}(Y^\#;\R^d)\Bigr]_\theta \cap
        W_{\per,Y}^{1,\max(s_0,s_1)}(Y^\#;\R^d) \\ = W^{1,s}(Y^\#;\R^d)\cap
        W_{\per,Y}^{1,\max(s_0,s_1)}(Y^\#;\R^d) =
        W_{\per,Y}^{1,s}(Y^\#;\R^d) = W^{1,s}_{\per}(Y;\R^d).
    \end{multline*}
    Here, the interpolation of $W^{1,r}(Y^\#;\R^d)$ is classical since
    we have assumed $Y^\#$ to be regular. Overall, this gives the claim.
\end{proof}

\begin{lemma}
    \label{lem:interpolation-periodic-mean0}
    Let $\theta \in (0,1)$ and $s_0,s_1 \in (1,\infty)$ and set $\frac1s =
    \frac{1-\theta}{s_0} + \frac\theta{s_1}$. Then
    \begin{equation*}
        \Bigl[W^{1,s_0}_{\per,\perp}(Y;\R^N),
        W^{1,s_1}_{\per,\perp}(Y;\R^N)\Bigr]_\theta = W^{1,s}_{\per,\perp}(Y;\R^N).
    \end{equation*}  
\end{lemma}

\begin{proof}
    We again argue via the complemented subspace interpolation
    theorem. For every $1 < r < \infty$,
    the operator
    \begin{equation*}
        P_\perp u := u - \fint_{Y} u \, dy
    \end{equation*}
    is a projection of $W^{1,r}(Y;\R^d)$ onto $W^{1,r}_{\perp}(Y;\R^d)$.
    Note that $P_\perp$ maps the closed subspace $W^{1,r}_\per(Y;\R^d)$
    into itself, hence $W^{1,r}_{\per,\perp}(Y;\R^d)$ is a complemented
    subspace of $W^{1,r}_\per(Y;\R^d)$ by means of the projection
    $P_\perp$. Using~\cite[Theorem~1.17.1]{Triebel:1978} and
    \cref{lem:interpolation-periodic}, we thus obtain
    \begin{multline*}
        \Bigl[W^{1,s_0}_{\per,\perp}(Y;\R^d),
        W^{1,s_1}_{\per,\perp}(Y;\R^d)\Bigr]_\theta \\ =
        \Bigl[W^{1,s_0}_{\per}(Y;\R^d) \cap
            W^{1,\max(s_0,s_1)}_{\perp}(Y;\R^d), W^{1,s_1}_{\per}(Y;\R^d)\cap
        W^{1,\max(s_0,s_1)}_{\perp}(Y;\R^d)\Bigr]_\theta \\ =
        \Bigl[W^{1,s_0}_{\per}(Y;\R^d),
        W^{1,s_1}_{\per}(Y;\R^d)\Bigr]_\theta \cap
        W^{1,\max(s_0,s_1)}_{\perp}(Y;\R^d) \\= W^{1,s}_\per(Y;\R^d) \cap
        W^{1,\max(s_0,s_1)}_{\perp}(Y;\R^d) =
        W^{1,s}_{\per,\perp}(Y;\R^d),
    \end{multline*}
    as desired.
\end{proof}

\begin{theorem}
    \label{thm:interpolation-vs}
    Let $\theta \in (0,1)$ and $s_0,s_1 \in (1,\infty)$ and set $\frac1s =
    \frac{1-\theta}{s_0} + \frac\theta{s_1}$. Then
    \begin{equation*}
        \bigl[V_{s_0},V_{s_1}\bigr]_\theta = V_s.
    \end{equation*}
\end{theorem}

\begin{proof}
    By~\cite[Thm.~1.18.4]{Triebel:1978}, we have
    \begin{align*}
        \bigl[V_{s_0},V_{s_1}\bigr]_\theta  & =
        \Bigl[L^{s_0}\bigl(\Omega;W^{1,s_0}_{\per,\perp}(Y;\R^d)\bigr),
        L^{s_1}\bigl(\Omega;W^{1,s_1}_{\per,\perp}(Y;\R^d)\bigr)\Bigr]_\theta
        \\ & =   L^{s}\bigl(\Omega;\bigl[W^{1,s_0}_{\per,\perp}(Y;\R^d),
        W^{1,s_1}_{\per,\perp}(Y;\R^d)\bigr]_\theta\bigr) \\ & =
        L^s\bigl(\Omega;W^{1,s}_{\per,\perp}(Y;\R^d)\bigr) = V_s,
    \end{align*}
    where the interpolation identity for
    $W^{1,s}_{\per,\perp}(Y;\R^d)$ is a consequence of
    \cref{lem:interpolation-periodic-mean0}. 
\end{proof}


\section{Auxiliary Results}

\begin{lemma}
    \label{lem:pointwiseConvergenceInpliesUniformlyConvergence}
    Let $\mathspace{M}$ be a compact metric space and $\mathspace{N}$
    a metric space. Furthermore, let
    $\sequence{G}{n} \subset C(\mathspace{M};\mathspace{N})$,
    $G \in C(\mathspace{M};\mathspace{N})$ with
    $G_n(x) \rightarrow G(x)$ for all $x \in \mathspace{M}$ and
    suppose that $G_n$ is uniformly Lipschitz continuous, that is,
    there exists a constant $L$ such that
    \begin{align*}
        d_\mathspace{N}(G_n(x), G_n(y)) \leq L d_\mathspace{M}(x,y)	
    \end{align*}
    for all $n \in \mathbb{N}$, $x,y \in \mathspace{M}$. \\
    Then $G_n \rightarrow G$ in $C(\mathspace{M};\mathspace{N})$.
\end{lemma}
\begin{proof} We argue by contradiction. Assume that there exists
    $\varepsilon > 0$ and a strictly monotonically increasing function
    $n \colon  \mathbb{N} \rightarrow \mathbb{N}$, such that for all
    $k \in \mathbb{N}$ there exists $x_k \in \mathspace{M}$ with
    \begin{align*}
        \varepsilon \leq d_\mathspace{N}(G_{n(k)}(x_k), G(x_k))
    \end{align*}
    for all $k \in \mathbb{N}$. Since $\mathspace{M}$ is compact, we
    can extract a subsequence $x_{k_j}$ of $x_k$ such that
    $x_{k_j} \rightarrow x$ in $\mathspace{M}$, thus
    \begin{align*}
        d_\mathspace{N}(G_{n({k_j})}(x_{k_j}), G(x_{k_j}))
        &\leq d_\mathspace{N}(G_{n({k_j})}(x_{k_j}), G_{n({k_j})}(x))
        + d_\mathspace{N}(G_{n({k_j})}(x), G(x_{k_j})) \\
        &\leq L d_\mathspace{M}(x_{k_j}, x) + d_\mathspace{N}(G_{n({k_j})}(x), G(x_{k_j}))
        \rightarrow 0,
    \end{align*}
    which gives the contradiction.  
\end{proof}

\begin{lemma}
    \label{lem:compactSequenceSets}
    Let $\mathspace{M}$ be a compact metric space and $\mathspace{N}$
    a metric space. Furthermore, let
    $\sequence{G}{n} \subset C(\mathspace{M};\mathspace{N})$,
    $G \in C(\mathspace{M};\mathspace{N})$ with $G_n \rightarrow G$ in
    $C(\mathspace{M};\mathspace{N})$.  Define
    $U_n \defn  G_n(\mathspace{M})$ and $U_0 \defn  G(\mathspace{M})$.  Then
    the set $U \defn  \cup_{n = 0}^\infty U_n$ is compact.
\end{lemma}
\begin{proof} Let $\sequence{y}{k} \subset U$. Since a finite
    union of compact sets is compact, we can assume that there exists
    a subsequence $\{ y_{k_j} \}_{j \in \mathbb{N}}$ and a strictly
    monotonically increasing function
    $n \colon  \mathbb{N} \rightarrow \mathbb{N}$, such that
    $y_{k_j} \in U_{n(j)}$.  Then there exists a sequence
    $\sequence{x}{j} \subset \mathspace{M}$, with
    $y_{k_j} = G_{n(j)}(x_j)$. Because $\mathspace{M}$ is compact we
    can select a subsequence, again denoted by $x_j$, and a limit
    $x \in \mathspace{M}$, such that $x_j \rightarrow x$, hence,
    \begin{align*}
        d_\mathspace{N}(y_{k_j}, G(x))
        \leq d_\mathspace{N}(y_{k_j}, G(x_j)) + d_\mathspace{N}(G(x_j), G(x)) \rightarrow 0,
    \end{align*}
    thus the proof is complete.  
\end{proof}

\bibliographystyle{jnsao}
\bibliography{references}

\end{document}